\documentclass[10pt]{amsart}

\usepackage{amsmath,amsthm} 
\usepackage{amssymb}
\usepackage[all]{xypic} 
\usepackage{amsbsy}
\usepackage{graphicx}  
\usepackage{subfigure} 
\usepackage[all,cmtip]{xy} 
\usepackage{multirow}  
\usepackage{rotating} 
\usepackage{comment} 
\usepackage{tikz} 
\usetikzlibrary{arrows, arrows.meta, decorations.markings, decorations.pathmorphing, shapes} 
\usepackage{relsize} 
\usepackage{hyperref} 

\numberwithin{equation}{section}

\newtheorem{thm}{Theorem}[section]
\newtheorem{cor}[thm]{Corollary}
\newtheorem{conj}[thm]{Conjecture}
\newtheorem{lem}[thm]{Lemma}
\newtheorem{prop}[thm]{Proposition}
\theoremstyle{definition} 
\newtheorem{defn}[thm]{Definition}
\newtheorem{rem}[thm]{Remark}
\newtheorem{exam}[thm]{Example}

\newtheorem*{thm-Dynkin}{Theorem \ref{thm:Dynkin}}
\newcommand{\thmDynkin}{Let $(Q, D_{Q})$ be Dynkin. Suppose $D_{Q}$ Frobenius degenerates flatly to $D_{Q}'$. Then, $\Pi(Q, D_{Q})$ and $\Pi(Q, D_{Q}')$ are both Frobenius algebras and $\Pi(Q, D_{Q})$ Frobenius degenerates flatly to $\Pi(Q, D_{Q}')$.
}

\newtheorem*{conj-strongflatness}{Conjecture \ref{conj:strongflatness}}
\newcommand{\conjstrongflatness}{
The Hilbert--Poincar\'{e} series of $\Pi(Q, D_{Q})$, with $D_{Q} = (A_{i}, M^{\alpha})$ is given by,
$$
h_{\Pi(Q, D_{Q})}(t) =
\begin{cases} 
(I + t^{\text{max}_{t}+2}P)D(1-At+t^{2})^{-1}
& \text{for } (Q, D_{Q}) \text{ Dynkin} \\
D(1-At+t^{2})^{-1} & \text{otherwise}
\end{cases}
$$
where, in the Dynkin case, $\text{max}_{t}$ is the maximum path length of a non-zero homogeneous element in $\Pi(Q, D_{Q})$ and $P$ is the permutation matrix for the permutation of the set $\{ 1_{A_{i}} \}_{i \in Q_{0}}$ given by the Nakayama automorphism. 
}

\newtheorem*{conj-flatness}{Conjecture \ref{conj:flatness}}
\newcommand{\conjflatness}{
Suppose $D_{Q}$ flatly Frobenius deforms to $D_{Q}'$. Then with respect to the path length grading,
$$
h_{\Pi(Q, D_{Q})}(t) = h_{\Pi(Q, D_{Q}')}(t).
$$  
}

\newcommand{\opr}{\operatorname{r}} 
\newcommand{\oplt}{\operatorname{lt}} 
\newcommand{\irr}{\text{irr}}

\newcommand{\bF}{\mathbb{F}}

\newcommand{\bZ}{\mathbb{Z}}
\newcommand{\bN}{\mathbb{N}}

\newcommand{\cB}{\mathcal{B}}
\newcommand{\cF}{\mathcal{F}}
\newcommand{\cM}{\mathcal{M}} 

\newcommand{\fg}{\mathfrak{g}} 
 
\newcommand{\fl}{\mathfrak{l}}

\newcommand{\Hom}{\text{Hom}}
\newcommand{\RHom}{\text{RHom}}

\newcommand{\End}{\text{End}}
\newcommand{\Span}{\text{Span}}
\newcommand{\Rep}{\text{Rep}}
\newcommand{\Mat}{\text{Mat}}

\newcommand{\Aut}{\text{Aut}}
\newcommand{\degen}{
\resizebox{1.6cm}{!}{
\begin{tikzpicture}[decoration=snake]
\node(A) at (0,0) {};
\node(B) at (1.25,0) {};
\draw[->,decorate] (A) -- (B) 
	node[midway, above]{\small{deg}} ;
\end{tikzpicture}
}
}
\newcommand{\fold}{
\overset{\text{fold}}{
\resizebox{1cm}{!}{$\Longrightarrow$}} 
}

\begin{document}

\title{Frobenius Degenerations of Preprojective Algebras}

\author{Daniel Kaplan}
\address{University of Birmingham, Edgbaston, Birmingham B15 2TT,  United Kingdom}
\email{d.kaplan@bham.ac.uk}


\date{April 19, 2018}

\begin{abstract}
In this paper, we study a preprojective algebra for quivers decorated with $k$-algebras and bimodules, which generalizes work of Gabriel for ordinary quivers, work of Dlab and Ringel for $k$-species, and recent work of de Thanhoffer de V\"{o}lcsey and Presotto, which has recently appeared from a different perspective in work of K\"{u}lshammer. As for undecorated quivers, we show that its moduli space of representations recovers the Hamiltonian reduction of the cotangent bundle over the space of representations of the decorated quiver. These algebras yield degenerations of ordinary preprojective algebras, by folding the quiver and then degenerating the decorations. We prove that these degenerations are flat in the Dynkin case, and conjecture, based on computer results, that this extends to arbitrary decorated quivers. 
\end{abstract}
\maketitle

\tableofcontents

\pagebreak

\section{Introduction}
A quiver, coined by Peter Gabriel, is a directed graph, whose representations are defined by placing a vector space at each vertex and a linear transformation at each arrow. Gabriel classified connected quivers with finite representation type: these are the quivers whose underlying graph is given by the simply-laced Dynkin diagrams, two infinite families $A_{n}, n \geq 1$ and $D_{n}, n \geq 4$, and three exceptions $E_{6}, E_{7}, E_{8}$ \cite{Gabriel}. Subsequently, Dlab and Ringel proved a more general result for ``species'', quivers labelled with a division ring at each vertex and a bimodule at each edge. They extended Gabriel's classification by showing that species with finite representation type are precisely the Dynkin ones \cite{Dlab}. This paper explores a further generalization to decorated quivers -- quivers with vertices labelled by $k$-algebras and arrows labelled by bimodules. An original motivation for this work is \cite{Presotto}, which in our language considers the affine $D_{4}$ case, while we focus on the Dynkin case.

One approach to representations of quivers involves studying the variety of all such representations and considering its cotangent bundle. More precisely, fixing a field $k$, a quiver $Q$ with vertices $Q_{0}$ and arrows $Q_{1}$, and a dimension vector 
$d = (d_{i})_{i \in Q_{0}} \in \bN^{Q_{0}}$, define $\Rep_{d}(Q)$ to be the vector space of $k$-representations of $Q$ with vertex $i$ labelled by the vector space $k^{d_{i}}$. There is a natural action of the group $GL_{d}(k) := \prod_{i \in Q_{0}} GL_{d_{i}}(k)$ on $\Rep_{d}(Q)$. Letting $\fg \fl_{d}(k) := \oplus_{i \in Q_{0}} \fg \fl_{d_{i}}(k)$ denote the Lie algebra of $GL_{d}(k)$, one studies the moduli space
$$
\cM_{d} := T^{*} \Rep_{d}(Q) /\! / \! / GL_{d}(k) := \mu^{-1}(0) /\! / GL_{d}(k)
$$
where $\mu: T^{*} \Rep_{d}(Q) \rightarrow \fg \fl_{d}(k)^{*}$ is the moment map for the associated action. 

The preprojective algebra gives an algebraic construction of $\cM_{d}$. First, one defines the double quiver $\overline{Q}$ to have the same vertices as $Q$ but with a reverse arrow $\alpha^{*} \in \overline{Q}_{1}$ for each $\alpha \in Q_{1}$. Then the preprojective algebra, $\Pi(Q)$ is defined as the quotient
$$
\Pi(Q) := P(\overline{Q}) / \left \langle \sum_{\alpha \in Q_{1}} \alpha \alpha^{*} - \alpha^{*} \alpha \right \rangle 
$$
where $P(\overline{Q})$ is the path algebra of $\overline{Q}$ \cite{Gelfand}. Then we recover $\cM_{d}$ as the moduli space of representations of $\Pi(Q)$ with dimension vector $d$:
\begin{equation} \label{eq moduli of reps}
\cM_{d} \cong \Rep_{d} \Pi(Q) / \! / GL_{d}(k).
\end{equation}

Anachronistically, one can view this as a litmus test for a correct definition of a preprojective algebra. For a quiver $Q$ together with a decoration $D_{Q}$ consisting of a $k$-algebra at each vertex and a bimodule at each arrow, we would like a notion of preprojective algebra $\Pi(Q, D_{Q})$ such that the analogue of (\ref{eq moduli of reps}) holds:
 \begin{equation} \label{eq moduli of decorated reps}
\Rep_{d} \Pi(Q, D_{Q}) / \! / G_{d} \cong 
 \mu^{-1}(0) /\! / G_{d}.
\end{equation}

We will define an algebra $\Pi(Q, D_{Q})$ and a group $G_{d}$ such that (\ref{eq moduli of decorated reps}) holds in the presence of certain restrictions on the representations and the decorations.

\begin{rem}
A definition of $\Pi(Q, D_{Q})$ was recently given in \cite{Julian}. Thanks to K\"{u}lshammer for pointing this out. Under condition (F) below this recovers our definition. Note that in \cite{Julian}, the viewpoint of representation varieties was not considered and it's an open question to what extent condition (F) can be relaxed while retaining a moment map interpretation. See Remark \ref{rem:weakerassumptions} for a more precise discussion.
\end{rem}

In more detail, a decoration $D_{Q}$ consists of 
\begin{itemize}
\item $\{ A_{i} \}_{i \in Q_{0}}$ with each $A_{i}$ a $k$-algebra and 
\item
 $\{ M^{\alpha} \}_{\alpha \in Q_{1}}$ with $M^{\alpha}$ an $(A_{i}, A_{j})$-bimodule, if $\alpha$ is an arrow from $i$ to $j$. 
\end{itemize}
A decoration satisfies condition (F) if each $A_{i}$ is symmetric Frobenius and, for every $\alpha$, $M^{\alpha}$ is either $A_{i} \otimes_{k} A_{j}$, or $M^{\alpha} = A_{i} = A_{j}$. \\
\\
A representation of a decorated quiver $(Q, D_{Q})$ with $D_{Q} = ( A_{i} , M^{\alpha} )$ consists of
\begin{itemize}
\item $\{ V_{i} \}_{i \in Q_{0}}$ with each $V_{i}$ a \emph{right} $A_{i}$-module and 
\item
 $\{ \rho_{\alpha}: M^{\alpha} \rightarrow \Hom_{k}(V_{i}, V_{j})\}_{\alpha: i \rightarrow j \in Q_{1}, i, j \in Q_{0}}$ with each $\rho_{\alpha}$ an $(A_{i}, A_{j})$-bimodule map. 
\end{itemize}
We fix a dimension vector $d = (d_{i}) \in \bN^{Q_{0}}$ and restrict to representations where $V_{i} = A_{i}^{d_{i}}$ for each $i \in Q_{0}$. Denote this vector space by $\Rep_{d}(Q, D_{Q})$.  

We define the  decorated path algebra to be $P(Q, D_{Q}) := T_{\oplus_{i} A_{i}}( \oplus_{\alpha} M^{\alpha})$, so modules for the decorated path algebra are representations of the decorated quiver. We then double the quiver by adding an arrow $\alpha^{*}: j \rightarrow i$ for each $\alpha: i \rightarrow j \in Q_{1}$ and decorate the new arrow $\alpha^{*}:j \rightarrow i$ with the bimodule dual,  
$$
M^{\alpha^{*}} := \Hom_{A_{i} \otimes A_{j}}(M^{\alpha}, A_{i} \otimes A_{j}).
$$ 
A decorated preprojective algebra satisfying (\ref{eq moduli of decorated reps}) cannot be defined in this level of generality. Hence we restrict to decorations satisfying condition (F) and representations with $V_{i} := A_{i}^{d_{i}}$ for each $i \in Q_{0}$. The decorated preprojective algebra $\Pi(Q, D_{Q})$ is the quotient of the path algebra for the double quiver by the two-sided ideal generated by the element $r = \sum_{\alpha \in Q_{1}} r_{\alpha}$. Here $r_{\alpha}$ for $\alpha: i \rightarrow j$ is a commutator $[1_{M^{\alpha}}, 1_{M^{\alpha^{*}}}]$ in the case $M^{\alpha} = A_{i} = A_{j}$, and $\sum_{l} e^{i}_{l} \otimes 1 \otimes f^{i}_{l}  - \sum_{m} e^{j}_{m} \otimes 1 \otimes f^{j}_{m}$ in the case $M^{\alpha} = A_{i} \otimes A_{j}$, where $\{ e^{i}_{l} \}$ and $\{ f^{i}_{l} \}$ are dual bases of $A_{i}$ under the Frobenius form, for all $i \in Q_{0}$. See Definition \ref{def:dec preproj alg} for more detail and Remark \ref{rem:well-defined} for independence on the choice of basis. 

There is a natural action of $G_{d} := \prod_{i} GL_{A_{i}}(V_{i}) \cong \prod_{i} GL_{A_{i}}(A_{i}^{d_{i}})$ on $\Pi(Q, D_{Q})$. Denoting the corresponding Lie algebra by $\fg_{d} = \oplus_{i} \fg \fl_{d_{i}}(A_{i})$ one can again consider the moment map $\mu_{d} : T^{*} \Rep_{d}(Q, D_{Q}) \rightarrow \fg_{d}^{*}$. Finally, we can state (\ref{eq moduli of decorated reps}) precisely as a theorem.   

\begin{thm} \label{Main}
$\Rep_{d}(\Pi(Q, D_{Q}))$ is the zero fiber of the moment map 
$$ 
\mu_{d}: T^{*} \Rep_{d}(Q, D_{Q}) \rightarrow \fg^{*}_{d}.
$$ 
Moreover, the coarse moduli space of representations of $\Pi(Q, D_{Q})$ of dimension vector $d$ is given by the Hamiltonian reduction $\mu_{d}^{-1}(0) / \! /G_{d}$.
\end{thm}
\noindent This theorem follows from a more precise statement, Theorem \ref{thm moment map}, given in Section 3.\\
\\
\indent A second goal of this paper is to study 1-parameter families of decorated preprojective algebras whose generic fibers are ordinary preprojective algebras. In this sense, decorated preprojective algebras are often (formal) degenerations of preprojective algebras, another motivation for their study. 

This perspective is best understood in a simple example. Consider the decorated quiver 
$$
(Q, D_{Q}) = k \overset{k}{\longleftarrow} k \overset{k}{\longrightarrow} k.
$$ 
Since right $k$-modules are vector spaces and $k$-bimodule maps are linear transformations, $\Rep_{d}(Q, D_{Q}) \cong \Rep_{d}(A_{3})$ and $ \Pi(Q, D_{Q}) = \Pi(A_{3})$ is an ordinary preprojective algebra. View these two arrows as a single arrow into the direct sum, folding the decorated quiver as follows:
$$
k \overset{k}{\longleftarrow} k \overset{k}{\longrightarrow} k
\hspace{1cm} \fold  \hspace{1cm}
 k \overset{k \oplus k}{\longrightarrow} k \oplus k. 
 $$
Folding the decorated quiver does not change the decorated preprojective algebra, $\Pi(A_{3})$. Furthermore, $k \oplus k$ can be viewed as a 2-dimensional $k$-algebra with pointwise multiplication where it has a non-trivial flat degeneration to the $k$-algebra $S =k[x]/(x^{2})$. Hence one can degenerate the decoration 
 $$
 k \overset{k \oplus k}{\longrightarrow} k \oplus k
\hspace{1cm} \degen \hspace{1cm}
 k \overset{S}{\longrightarrow} S
$$ 
which induces a degeneration of $\Pi(A_{3}) = \Pi(k \overset{k \oplus k}{\longrightarrow} k \oplus k)$ to $\Pi(k \overset{S}{\longrightarrow} S)$. We denote this degeneration by $\Pi(B_{2})$ since the $B_{2}$ quiver is obtained from $A_{3}$ by folding the arrows as above. 

Both the formal and filtered perspectives of this degeneration prove fruitful. In the formal setting, $k[[t]][x]/(x^{2}-t)$ is an explicit formal degeneration from $k \oplus k$ to $S$. The corresponding formal degeneration of decorated quivers is
$$
\xymatrix{
& ( k[[t]]  \ar[rrr]^-{k[x][[t]]/(x^{2}-t)}   \ar[ld]_{``t=1"} & & & k[x][[t]]/(x^{2}-t) ) \ar[rd]^{t=0}  & \\
( k \overset{k \oplus k}{\longrightarrow} k \oplus k )
 & & & & & ( k \overset{S}{\longrightarrow} S ),
}
$$
which in turn gives a formal degeneration $\Pi_{k[[t]]}( k[[t]] \longrightarrow  k[x][[t]]/(x^{2}-t) )$ from $\Pi(A_{3})$ to $\Pi(k \rightarrow S)$. 
Alternatively, viewing $k \oplus k \cong k[x]/(x^{2}-1)$ as a filtered algebra with associated graded algebra $S$, realizes $\Pi(A_{3})$ as a filtered algebra with associated graded algebra $\Pi(k \rightarrow S).$ The filtration on $\Pi(A_{3})$ is given by $
\{ 0 \} \subset \Pi(A_{3})^{\Phi} \subset \Pi(A_{3})
$
where $\Phi: A_{3} \rightarrow A_{3}$ is the graph automorphism exchanging the outer vertices and $\Pi(A_{3})^{\Phi}$ is the subspace of $\Phi$-invariant elements. 

This motivates the study of decorated quivers as a means of interpreting a large class of degenerations of preprojective algebras. But one should not conflate these notions. The decorated preprojective algebra $\Pi(k \rightarrow \Mat_{2}(k) )$ is \emph{not} a degeneration of $\Pi(\tilde{D}_{4})$, where $\tilde{D}_{4}$ is the affine $D_{4}$ quiver. Conversely, $\Pi(A_{2})$ has a flat Frobenius degeneration to $\bigwedge(k^{2})$, see Example \ref{Graph in dim 4}, which is \emph{not} a decorated preprojective algebra for any decorated quiver except $Q = A_{1}$ and $D_{Q}= (\bigwedge(k^{2}))$.  

Most of the decorated quivers appearing in this paper arise as degenerations of ordinary quivers by folding vertices and then degenerating, as we now explain. In general, a group $H$ acts on a quiver $Q$ if $H$ acts on the sets $Q_{1}$ and $Q_{0}$ such that the source and target maps are $H$-equivariant. In this case, the source and target maps descend to the quotient quiver $Q/H := (Q_{0}/H, Q_{1}/H, s, t)$. In the case $H= \Aut(Q)$ is the group of graph automorphisms of the underlying undirected graph of $Q$, we call the quotient $Q/ \Aut(Q)$ the folding of $Q$ by automorphisms.  	

If $D_{Q} =(A_{i}, M^{\alpha})$ is a decoration of $Q$ then one can decorate $Q/ \Aut(Q)$ by $D_{Q}/\Aut(Q) = (B_{j}, N^{\beta})$ defined by:
$$
B_{j} := \bigoplus_{i \in \Aut(Q)\cdot j} A_{i} 
\hspace{1cm} \text{and} \hspace{1cm} 
N^{\beta} := \bigoplus_{\alpha \in \Aut(Q) \cdot \beta} M^{\alpha}.
$$
In this case, we say $(Q, D_{Q})$ \emph{folds} to $(Q/\Aut(Q), D_{Q}/\Aut(Q))$. Notice that a folding of decorated quivers induces an equivalence 
$$
\Rep(Q, D_{Q}) \cong \Rep(Q/\Aut(Q), D_{Q}/\Aut(Q))
$$ 
and an isomorphism 
$$
\Pi(Q, D_{Q}) \cong \Pi(Q/\Aut(Q), D_{Q}/\Aut(Q)). \vspace*{.3cm}
$$ 

The second aspect of this procedure is degenerating the decoration, which amounts to degenerating the algebras at each vertex and the bimodules at each arrow. Since the definition of the preprojective algebra requires each algebra to be Frobenius, one needs to degenerate the algebras in a way that preserves a Frobenius form, see Section 4 for more detail.
 
We suspect that these degenerations of preprojective algebras are always flat. In other words, we present the following conjecture in Section 4.

\begin{conj-flatness} 
\conjflatness
\end{conj-flatness}

\begin{rem}
The conjecture was proven in the case $Q = \tilde{D_{4}}$ for \emph{commutative} Frobenius algebras $A_{i}$ in \cite{Presotto}, by computing an explicit basis for the maximally folded and maximally degenerated decoration 
$$
\xymatrix{
& k \ar[d] & \\
k \ar[r] & k & k \ar[l] 
\hspace{.3cm} \fold \hspace{.3cm}  
k^{\oplus 4} \ar[r] & k \degen  k[x, y]/(xy, x^{2}-y^{2}) \ar[r] & k. \\
& k \ar[u] & 
}
$$ 
This work may be considered a sequel, in that we generalize their statement to arbitrary Frobenius decorated quivers. Moreover, we prove the conjecture in the Dynkin case. 
\end{rem}
 

In fact, we suspect that the existence of a deformation, while crucial for our techniques, may be an artifact of the proof style and unnecessary for the result. That is, assuming condition (F), the Hilbert--Poincar\'{e} series of the decorated preprojective algebra may depend purely combinatorially on the dimensions of the algebras and bimodules in the decoration. 

One can define a (not-necessarily symmetric) Cartan matrix $C:= 2I-A$ where $A = [a_{ij}]$ is a $Q_{0} \times Q_{0}$ matrix with entries 
$$
a_{ij} := \sum_{\alpha: i \rightarrow j} \frac{  \dim_{k}(M^{\alpha})}{ \dim_{k}(A_{i})} 
= \sum_{\alpha: i \rightarrow j} \dim_{A_{i}}(M^{\alpha}) \in \bZ.
$$ 
We say $(Q, D_{Q})$ is \emph{Dynkin} if the Cartan matrix $C$ is positive-definite. 

Observe, that if $D =[d_{ij}]$ with $d_{ij} := \delta_{ij} \dim_{k}(A_{i})$ then $DA$ is symmetric as 
$$
(DA)_{ij} = \sum_{\alpha:i \rightarrow j} \dim_{k}(M^{\alpha}) = \sum_{\alpha^{*}: j \rightarrow i} \dim_{k}(M^{\alpha^{*}}) = (DA)_{ji}. 
$$
Therefore, $DC$ is symmetric and from the perspective of \cite{GLS}, $C$ is a symmetrizable Cartan matrix with symmetrizer $D$.
 
\begin{conj-strongflatness} 
\conjstrongflatness
\end{conj-strongflatness}

Moreover, we will often have an additional grading on the decorated path algebra coming from placing a graded algebra at one or more vertices. Throughout the paper we call this the $x$-grading 
defined as the number of occurrences of $x$ in the path, since the graded algebras we consider at the vertices are of the form $k[x]/(x^{n})$ for some $n \in \bN$. The $x$-grading descends from the decorated path algebra to the decorated preprojective algebra in the $C_{n}$, $G_{2}$, and $B_{2}$ cases. The $x$-grading can be adjusted to descend in the $B_{n}$ with $n \geq 3$, and $F_{4}$ cases. With respect to path length and an additional $x$-grading the Hilbert--Poincar\'{e} series is conjecturally given by
$$
h_{\Pi(Q, D_{Q})}(t, s)
=
\begin{cases}
(1 + s^{\max_{s}}t^{\max_{t}+2}P) D (1 -At +Bt^{2})^{-1} & \text{for } (Q, D_{Q}) \text{ Dynkin} \\
D (1 - At+Bt^{2})^{-1} & \text{ otherwise} \\
\end{cases}
$$
where $\max_{s}$ is the maximum $x$-grading of a non-zero homogeneous element of $\Pi(Q, D_{Q})$,
$$
A = [a_{ij}], \ \ a_{ij} := \sum_{\alpha: i \rightarrow j} \frac{ h_{M^{\alpha}}(s) }{ h_{A_{i}}(s) },
$$ 
and
$$
B = [b_{ij}], \ \ b_{ij}:= \delta_{ij} h_{k \cdot 1_{A_{i}}r }(s).
$$ 
Here we say $(Q, D_{Q})$ is \emph{Dynkin} if, as before, the Cartan matrix $2I-A \mid_{s=1}$ is positive-definite.

In Sections 4 and 5 we prove the conjecture in the Dynkin case and show that the Frobenius form on the preprojective algebra degenerates to one on the decorated preprojective algebras. More precisely, we prove
 
\begin{thm-Dynkin}
\thmDynkin
\end{thm-Dynkin}



Moreover, letting $S :=k[x]/(x^{2})$ and $S':=k[x]/(x^{3})$, the most degenerate decorated non-simply laced Dynkin quivers are: 
\begin{align*}
B_{n} &:= (A_{n+1}, ( \{ k , S, \dots, S \}, 
\{ _{k} S_{S}  , _{S} S_{S} , \dots, _{S} S_{S} \} )), \ \ n \geq 2 \\
C_{n} &: = (A_{n+1}, ( \{ k , \dots, k , S \}, 
\{ _{k} k_{k}  , \dots, {}_{k} k_{k} , {}_{k} S_{S} \} )), \ \ n \geq 4 \\
F_{4} &:= (A_{4}, ( \{ k , k, S, S \}, 
\{ _{k} k_{k}  , {}_{k}k_{S} , \ _{S}S_{S} \} )) \\
G_{2} &:= (A_{2}, ( \{ k , S' \}, 
\{ _{k} S'_{S'}  \} )).
\end{align*}
The (path, $x$)-bigraded Hilbert--Poincar\'{e} series of the corresponding decorated preprojective algebras are:

\begin{align*}
\begin{split}
h_{e_{i} \Pi(B_{n}) e_{j}}(s, t) {}&=
 t^{|i-j|}s^{|i-j|/2}(1+t^{2 \min \{ i, j \} -2} s^{\min \{ i, j \} - 1})(1+s)\\
 &\hspace{1cm} \cdot (1+st^2 +s^2t^4 + \cdots + s^{n-\max \{ i, j \}} t^{2n-2 \max \{ i, j \}}), \ \ i,j \geq 2 
\end{split}  \\
\begin{split}
h_{e_{1} \Pi(B_{n}) e_{j}}(s, t) {}&= 
h_{e_{j} \Pi(B_{n}) e_{1}}(s, t) = t^{j-1} s^{(j-1)/2}(1+s)\\
 &\hspace{1cm} \cdot (1+st^2+s^2t^4+ \cdots + s^{n-j} t^{2n-2j}), \ \  j \geq 2
\end{split} \\
h_{e_{1} \Pi(B_{n}) e_{1}}(s, t) {}&=1 + st^2+ s^2t^4 + \cdots + s^{n-1}t^{2n-2} 
\end{align*}
$h_{e_{i} \Pi(C_{n}) e_{j}}(s,t) = 
(1+ s t^{2 |n-\max \{i, j \}|})
t^{|i-j|}(1+t^2+t^4+ \cdots + t^{2 \min \{i,j\}-2}) $ \\ 

$h_{\Pi(F_{4})}(s, t) = (1+t^4s) \cdot$  
$$
\footnotesize{
\left ( 
\arraycolsep=3pt
\medmuskip=1mu
\begin{array}{cccc}
(1+t^6s)  & t(1+t^2s+t^4s^s) & t^2(1+s)(1+t^2s) &  t^3s^{1/2}(1+s) \\
t(1+t^2s+t^4s^2) & (1+t^2)(1+t^2s+t^4s) & t(1+s)(1+t^2s) &
                                            t^2s^{1/2}(1+s)(1+t^2) \\
t^2(1+s)(1+t^2s) & t(1+s)(1+t^2s) & (1+s)(1+t^2+t^4)(1+t^2s)
                                      & ts^{1/2}(1+s)(1+t^2+t^4) \\
 t^3s^{1/2}(1+s) &  t^2s^{1/2}(1+s)(1+t^2) & ts^{1/2}(1+s)(1+t^2+t^4)  &             
                                              (1+s)(1+t^6s)\\ 
\end{array}
\right )
}
$$
$
h_{\Pi(G_{2})}( s, t) = (1+t^2s)\left (
\begin{array}{cc}
1+t^2s^2 & (1+s+s^2)t \\
(1+s+s^2)t & (1+s+s^2)(1+t^2) \\
\end{array}
\right )
$ \\
For more detail, see Section 5.

\begin{exam}
We conclude the introduction by returning to our simple yet illustrative $B_{2}$ example, i.e., the $A_{2}$ quiver with vertices labelled by the $k$-algebras $k$ and $S=k[x]/(x^{2})$ and the arrow labelled by $S$ viewed as a $(k, S)$-bimodule, 
$$
\xymatrix{
k \ar@/^/[rrr]^{S} &&& S.
}
$$
$S$ has Frobenius form $\lambda: S \rightarrow k$ given by $s + tx \mapsto t$, which is non-degenerate since the kernel is the subspace of constant polynomials, which does not contain any non-trivial left ideals. With respect to this form, the basis $\{ 1, x \}$ has dual basis $\{ x, 1 \}$ giving rise to the non-degenerate, $S$-central element $x \otimes_{k} 1 + 1 \otimes_{k} x$. \\
\\
Using this structure, the decorated preprojective algebra is defined as
$$
\Pi(Q, D_{Q}) := T_{k \oplus S}(S \oplus S) / \langle 1 \otimes_{S} 1, 1 \otimes_{k} x + x \otimes_{k} 1 \rangle
$$
where $S \oplus S$ is made into a $(k \oplus S)$-bimodule by viewing one $S$ as an $(S,k)$-bimodule and the other a $(k, S)$-bimodule, and defining all other actions to be zero. More explicitly, one can see that $\Pi^{\geq 3}(Q, D_{Q}) = 0$ and compute,
\begin{align*}
\Pi(Q, D_{Q}) &= \Pi^{0}(Q, D_{Q}) \oplus \Pi^{1}(Q, D_{Q}) \oplus \Pi^{2}(Q, D_{Q})\\
&=(S \oplus k) \oplus (S \oplus S) \oplus 
(S \oplus k)
\end{align*}
where the second graded piece is computed using the identifications
$$
S \otimes_{S} S / \langle 1 \otimes_{S} 1 \rangle \cong S/ \langle 1 \rangle \cong k
$$
as vector spaces and 
$$
S \otimes_{k} S / \langle 1 \otimes_{k} x + x \otimes_{k} 1, x \otimes_{k} x \rangle \cong S
$$ 
as $S$-bimodules. Consequently, the matrix-valued Hilbert--Poincar\'{e} series in variables $t$ and $s$ corresponding to path length and $x$-grading respectively is given by
$$
h_{\Pi(B_{2})}(t, s) = \left (
\begin{array}{cc}
1 + t^2s & t(1+s) \\
t(1+s) & (1+t^2)(1+s) \\
\end{array}
\right ).
$$ 
This agrees with the formula in Conjecture \ref{conj:strongflatness}.
\end{exam}

The paper is organized as follows. We recall the theory of representations of quivers and decorated quivers in Section 2. We define the preprojective algebra for quivers and decorated quivers in Section 3, and prove Theorem \ref{Main}. In Section 4, we explain the degeneration perspective more carefully, first defining the notion for Frobenius algebras. We conjecture that a flat degeneration of the decorations gives rise to a flat degeneration of the corresponding decorated preprojective algebras. Finally, we present a diverse class of examples in low dimensions where the conjecture has been verified using Magma. 
In Section 5, we prove the conjecture in the case of Dynkin quivers by reducing to a single, most degenerate decorated quiver for each Dynkin quiver and producing an explicit basis for its decorated preprojective algebra.\\
\\
{\bf Acknowledgements} \\ 
The author is indebted to Travis Schedler for suggesting the project and for his patience and guidance throughout. The Max Planck Institute for Mathematics and Hausdorff Institute for Mathematics provided excellent working conditions. 
The author was financially supported by the Roth scholarship at Imperial College. 
I would like to thank Sue Sierra for reading an earlier version of this paper, and suggesting several improvements and clarifications. I'm grateful to Peter Crooks, Julian K\"{u}lshammer, Jos\'{e} Simental, and Michel Van den Bergh for helpful comments. I would also like to Karen Erdmann and Daniel Erman useful email correspondences. The exposition has been clarified following conversations with Bill Crawley-Boevey, Sam Gunningham, Michael McBreen, Pavel Safronov, and Michael Wong.

\section{Representations of Decorated Quivers}

\subsection{Recollections on Quivers}
For completeness of exposition and to fix notation, we recall the definition of a representation of a quiver, then show how the theory is subsumed by the representation theory of associative algebras. \\
\\
Fix a ground field $k$.
\begin{defn}
A \emph{quiver} is a quadruple $Q= (Q_{0}, Q_{1}, s, t)$ where $s, t: Q_{1} \rightarrow Q_{0}$ are the source and target maps respectively from the set of arrows, $Q_{1}$, to the set of vertices, $Q_{0}$. 
\end{defn}

A quiver is merely a directed graph, but the terminology comes from Gabriel, who considered representations of such graphs, defined as follows:

\begin{defn}
A \emph{representation} of a quiver $Q$ is $\rho = (V_{i}, \rho_{\alpha})_{i \in Q_{0}, \alpha \in Q_{1}}$ where $\rho_{\alpha}: V_{s(\alpha)} \rightarrow V_{t(\alpha)}$ is a $k$-linear map of vector spaces, for each $\alpha \in Q_{1}$. 
\end{defn} 


We consider exclusively quivers with finitely many vertices and arrows and representations with each vector space finite-dimensional. The function $\dim(V_{(-)}): Q_{0} \rightarrow \bN$ defines an $\bN^{Q_{0}}$-grading on the set of representations. Fixing $d = (d_{i}) \in \bN^{Q_{0}}$ we define
$$
\Rep_{d}(Q) := \{ \rho = ( k^{d_{i}}, \rho_{\alpha}) \} =  \bigoplus_{\alpha \in Q_{1}} \Hom(k^{d_{s(\alpha)}}, k^{d_{t(\alpha)}}).
$$
$\Rep_{d}(Q)$ is a finite-dimensional vector space with an action of $G_{d} := \prod_{i \in Q_{0}} GL_{d_{i}}(k)$ given by 
$$
(g_{i})_{i \in Q_{0}} \cdot (V_{i}, \rho_{\alpha})_{i \in Q_{0}, \alpha \in Q_{1}} = (V_{i}, \ g_{t(\alpha)} \circ \rho_{\alpha} \circ g^{-1}_{s(\alpha)} ).
$$
One says two representations are isomorphic if they lie in the same $G_{d}$-orbit. 

Alternatively, one can collectively study representations of $Q$ by considering the category $\mathcal{R}(Q)$ whose objects are representations of $Q$ and morphisms are commuting diagrams between representations. Invertible morphisms are given by $|Q_{0}|$-tuples of (compatible) invertible linear maps. Hence, for the full subcategory $\mathcal{R}_{d}(Q) \subset \mathcal{R}(Q)$ whose objects are $d$-dimensional representations, one has a natural identification between the moduli space of isomorphism classes of semisimple representations in $\mathcal{R}_{d}(Q)$ and the GIT quotient $\Rep_{d}(Q) // G$.

The category $\mathcal{R}(Q)$ is equivalent to the category of finitely-generated modules over an algebra, as we now explain. 

\begin{defn}
The \emph{path algebra} $P_{k}(Q)$ for a fixed quiver $Q=(Q_{0}, Q_{1}, s, t)$ is defined to be 
$$
P_{k}(Q):= T_{kQ_{0}}(kQ_{1}),
$$
where $kS$ denotes the $k$-vector space with basis $S$ and $kQ_{1}$ is made into a $kQ_{0}$-bimodule by $k$-linear extension of the action
$$
e_{i} \cdot \alpha \cdot e_{j} := \delta_{i s(\alpha)} \alpha \delta_{t(\alpha) j}.
$$ 
Here $e_{(\cdot)}: Q_{0} \rightarrow P_{k}(Q)$ regards a vertex as a length zero path. We will usually write $P(Q) := P_{k}(Q)$ suppressing the dependence on the field.
\end{defn}

\begin{rem}
As defined above, multiplication in $P(Q)$ is identified with concatenation of tensors in $T_{kQ_{0}}(kQ_{1})$ and hence is read left to right. This choice is convenient when performing computations with long paths, but requires considering right modules, as explained below.
\end{rem}

The path algebra is graded with $n$th graded piece given by $k$-linear combinations of paths of length $n$, i.e. $n$-tuples of composable arrows in the quiver. Since vertices are idempotents in $P(Q)$ summing to the identity element, there is a further decomposition of the path algebra by source and target vertices,
$$
P(Q) = 1 P(Q) 1 = \left (\sum_{i \in Q_{0}} e_{i} \right ) P(Q) \left ( \sum_{j \in Q_{0}} e_{j} \right ) = \bigoplus_{i, j \in Q_{0}} e_{i} P(Q) e_{j}.
$$
\indent Representations of $Q$ can be identified with right modules for the path algebra, $P(Q)$. Explicitly, if $V$ is a right $P(Q)$-module then one can form the representation $\rho = ( V_{i}, \rho_{\alpha})$ where $V_{i} := V \cdot e_{i}$ and $\rho_{\alpha}$ is the right action by $\alpha$, a linear map $V_{s(\alpha)} \rightarrow V_{t(\alpha)}$. Conversely, if $(V_{i}, \rho_{\alpha})$ is a representation of $Q$ one can form the right module $V:=\bigoplus_{i \in Q_{0}} V_{i}$ with length one paths $\alpha$ acting by $\rho_{\alpha}$ and paths $\gamma = \alpha_{1} \cdot \cdots \cdot \alpha_{n}$ acting by $\rho_{\alpha_{n}} \circ \cdots \circ \rho_{\alpha_{1}}$.   

To describe equivalence classes of representations one needs to compute the orbit space of the action of $G$ on $\Rep_{d}(Q)$. 

\subsection{Decorated quivers}

\begin{defn}
A \emph{decorated quiver} is a pair $(Q, D_{Q})$ where $Q=(Q_{0}, Q_{1}, s, t)$ is a quiver and 
$$
D_{Q}= ( A_{i}, M^{\alpha})_{i \in Q_{0}, \alpha \in Q_{1}}
$$
where $A_{i}$ is a $k$-algebra and $M^{\alpha}$ is an $(A_{s(\alpha)}, A_{t(\alpha)})$-bimodule. Fixing $Q$, a \emph{decoration} is a choice of $D_{Q}$. The decoration $D_{Q}$ is \emph{finite} if each $A_{i}$ is a finitely-generated $k$-algebra and each $M^{\alpha}$ is finitely generated as a left $A_{s(\alpha)}$-module and as a right $A_{t(\alpha)}$-module. 
\end{defn} 

\begin{rem}
This notion appears in the literature under the names $k$-species or modulated graph 
when each $A_{i}$ is a division ring \cite{Dlab}.
One can alternatively choose a presentation for each $A_{i}$ and instead consider the quiver with relations where a loop is added for each generator of each algebra and relations are imposed to capture identical information: see \cite{GLS} and its sequels. 
\end{rem}

\begin{rem}
One can take a categorical perspective by viewing $Q$ as a category where:
\begin{itemize}
\item objects are vertices $Q_{0}$,
\item morphisms are paths, and  
\item composition is given by concatenation of paths, when possible, and zero otherwise.
\end{itemize}
We denote this category by $F(Q)$ to emphasize the structural differences from the original quiver $Q$. Using this language, a decoration is a functor
$D_{(-)}: F(Q) \rightarrow \text{Alg}$ into the Morita category of algebras, objects are $k$-algebras, morphisms are bimodules, and composition of morphisms is given by the tensor product. This is the approach taken in \cite{Julian} under the name prospecies of algebras, in the case where the bimodule is projective as a left module and as a right module. 
\end{rem}
 
We now proceed to develop all of the notions in the previous section for decorated quivers. The original source for much of this material in the setting of species is Dlab and Ringel \cite{Dlab}. 

\begin{defn}
A \emph{representation} of a decorated quiver $(Q, D_{Q})$ is 
$$\rho = (V_{i}, \rho_{\alpha})_{i \in Q_{0}, \alpha \in Q_{1}}$$ where $V_{i}$ is a right $A_{i}$-module and $\rho_{\alpha}: M^{\alpha} \rightarrow \Hom_{k}(V_{s(\alpha)}, V_{t(\alpha)})$ is an $(A_{s(\alpha)}, A_{t(\alpha)})$-bimodule map.
\end{defn}

 Note that $\Hom_{k}(V_{s(\alpha)}, V_{t(\alpha)})$ has a right $A_{t(\alpha)}$-action by postcomposition and a left $A_{s(\alpha)}$-action by precomposition.
 
\begin{rem}
This definition reduces to the ordinary definition of a representation of an (undecorated) quiver when each $A_{i}=k$ and $M^{\alpha}={}_{k} k_{k}, $ in which case each $\rho_{\alpha}$ is determined by a single linear map $\rho_{\alpha}({}_{k}1_{k}) : V_{s(\alpha)} \rightarrow V_{t(\alpha)}$. 
\end{rem}

Fixing a quiver and a finite decoration $(Q, D_{Q})$, we denote by $\mathcal{R}(Q, D_{Q})$ the category of representations of finite type: each $V_{i}$ is finitely generated as a right $A_{i}$-module. Morphisms from $\rho$ to $\rho'$ in this category are given by a collection of maps $( f_{i}: V_{i} \rightarrow V_{i}')$ such that for all $\alpha \in Q_{1}$ and for all $m \in M^{\alpha}$, the diagram of right $(A_{s(\alpha)} \oplus A_{t(\alpha)})$-modules
$$
\xymatrix{
V_{s(\alpha)} \ar[rr]^{\rho_{\alpha}(m)} \ar[d]^{f_{s(\alpha)}} & &  V_{t(\alpha)}  \ar[d]^{f_{t(\alpha)}} \\
V'_{s(\alpha)} \ar[rr]^{\rho'_{\alpha}(m)}  & &  V'_{t(\alpha)} 
}
$$
commutes. As in the undecorated case, for a fixed $d \in \bN^{Q_{0}}$, we consider the full subcategory of ``locally free'' representations 
$\mathcal{R}_{d}(Q, D_{Q})$ whose objects are representations $\rho = ( V_{i}, M^{\alpha})$ with $V_{i} = A_{i}^{d_{i}}$ for all $i \in Q_{0}$.

As before, one can compute equivalence classes of objects in $\mathcal{R}_{d}(Q, D_{Q})$ as orbits under a group action on the vector space of all representations with $V_{i} = A_{i}^{d_{i}}$ for all $i \in Q_{0}$, denoted $\Rep_{d}(Q, D_{Q})$. In more detail, there is an action of $G_{d} = \prod_{i \in Q_{0}} GL_{A_{i}}(A_{i}^{d_{i}})$ on $\Rep_{d}(Q, D_{Q})$ by
$$
(g_{i})_{i \in Q_{0}} \cdot (V_{i}, \rho_{\alpha}) = (V_{i} , g_{t(\alpha)} \rho_{\alpha} g_{s(\alpha)}^{-1}). 
$$
One says two representations are isomorphic if they lie in the same $G_{d}$-orbit. Isomorphism classes of semisimple objects in $\mathcal{R}_{d}(Q, D_{Q})$ can be identified with $\Rep_{d}(Q, D_{Q})/ \!/ G_{d}$.\\


The category $\mathcal{R}_{d}(Q, D_{Q})$ is equivalent to a category of modules for an algebra, generalizing the path algebra in the previous subsection.

\begin{defn}
The \emph{decorated path algebra} $P_{k}(Q, D_{Q})$ for $D_{Q} = ( A_{i}, M^{\alpha})$ is the tensor algebra $T_{\oplus_{i} A_{i}}( \oplus_{\alpha} M^{\alpha})$ where each $M^{\alpha}$ is viewed as an $(\oplus_{i} A_{i}, \oplus_{i} A_{i})$-bimodule by defining 
$$
\left ( \sum_{i} a_{i} \right ) \cdot m_{\alpha} \cdot \left ( \sum_{i} a'_{i} \right )
:= a_{s(\alpha)} \cdot m_{\alpha} \cdot a'_{t(\alpha)} 
$$
for $a_{i}, a_{i}' \in A_{i}$ and $m^{\alpha} \in M^{\alpha}$. We will usually write $P(Q, D_{Q}) := P_{k}(Q, D_{Q})$ suppressing the dependence on the field. 
\end{defn}

\begin{prop}
The category of representations of the decorated quiver $(Q, D_{Q})$ is equivalent to the category of right modules over the path algebra $P(Q, D_{Q})$.
\end{prop}

\begin{proof}

Given a representation $\rho =(V_{i}, \rho_{\alpha})$ of the decorated quiver $D_{Q} =(A_{i}, M^{\alpha})$ one can form a right module $V:=\oplus_{i \in Q_{0}} V_{i}$ of $P(Q, D_{Q})$ with the action defined on generators by 
\begin{align*}
\left ( \sum_{i \in Q_{0}} v_{i} \right ) \cdot a_{j} &:= v_{j} \cdot a_{j} \hspace{1cm} a_{j} \in A_{j}, v_{i} \in V_{i} \\
\left ( \sum_{i \in Q_{0}} v_{i} \right ) \cdot m^{\alpha} &:= \rho_{\alpha}(m^{\alpha}) (v_{s(\alpha)}) \hspace{1cm} m^{\alpha} \in M^{\alpha}. 
\end{align*}
By the universal property of the tensor algebra, this extends to a right module over all of $P(Q, D_{Q})$ since $\rho_{\alpha}$ is an $(A_{s(\alpha)}, A_{t(\alpha)})$-bimodule map, which can be viewed as an $\oplus_{i} A_{i}$-bimodule map.

One can upgrade this map of objects to a functor defined by taking $\rho$ to $V$ as above and by taking a morphism $(f_{i}: V_{i} \rightarrow V_{i}')$ from $\rho$ to $\rho'$ to the map $\oplus_{i} f_{i}: \oplus V_{i} \rightarrow \oplus_{i} V'_{i}$. The condition that the diagrams commutes (i.e. $\rho_{\alpha}' \circ f_{s(\alpha)} = f_{t(\alpha)} \circ \rho_{\alpha}$) implies that $\oplus_{i} f_{i}$, apriori a $\oplus_{i} A_{i}$-module map, is a right $P(Q, D_{Q})$-module map. We now construct the inverse functor on objects, which in fact gives an \emph{isomorphism} of categories in the sense that the compositions equal the identity functor, as opposed to merely being naturally isomorphic.  

A right module for the decorated path algebra is a map $\varphi: T_{\oplus_{i} A_{i}}( \oplus_{\alpha} M^{\alpha})
\rightarrow \End_{k}(V)$, which is determined by its restriction to generators
$$
T^{\leq 1}_{\oplus_{i} A_{i}}( \oplus_{\alpha} M^{\alpha})
= \bigoplus_{i} A_{i} \oplus \bigoplus_{\alpha} M^{\alpha}.
$$
Then one defines $\rho := (\varphi(1_{A_{i}})V, \varphi \mid_{M^{\alpha}})$ where $ \varphi(1_{A_{i}})V$ is given a right $A_{i}$-module structure by 
$$
 \varphi(1_{A_{i}})(v)  \cdot a_{i} :=  \varphi(a_{i})(v) 
$$ 
for $v \in V, a_{i} \in A_{i}$. Clearly $\varphi \mid_{M^{\alpha}}$ is an $(A_{s(\alpha)}, A_{t(\alpha)})$-bimodule map since $\varphi$ is an $(\oplus_{i} A_{i})$-bimodule map.    
\end{proof}




As in the undecorated case, one would like to study the quotient $\Rep_{d}(Q, D_{Q})// G_{d}$. In the next section we consider the symplectic action of $G$ on the cotangent bundle to $\Rep_{d}(Q, D_{Q})$, but in order to describe the moment map we need to impose restrictions on the decoration. This definition will be justified a posteriori by our study of the zero fiber of the moment map.

\begin{defn}
A decoration $D_{Q} = (A_{i}. M^{\alpha})$ satisfies \emph{condition (F)} if 
\begin{itemize}
\item $A_{i}$ is symmetric Frobenius (i.e. $A_{i} \cong \Hom_{k}(A_{i}, k)$ as $A_{i}$-bimodules)
or all $i \in Q_{0}$.
\item $M^{\alpha}= A_{i} \otimes_{k} A_{j}$ or $M^{\alpha} = A_{i} = A_{j}$, for all $\alpha \in Q_{1}$ with $\alpha: i \rightarrow j$.
\end{itemize}
\end{defn} 
 
We will have occasional need to distinguish between the two allowable bimodule decorations in condition (F) and hence write $Q_{1} = Q_{1}' \sqcup Q_{1}''$ where the decorations for $Q_{1}'$ are of the form $A_{j} \otimes_{k} A_{i}$ and the decorations for $Q_{1}''$ are of the form $A_{i} = A_{j}$.


\begin{rem}
In subsection \ref{decorated preproj}, we explain why the Frobenius condition is necessary to define a preprojective algebra satisfying Theorem \ref{Main}. One can alternatively consider additional structure on the bimodules instead of the algebras, as is done in \cite{Dlab}. Symmetry of the Frobenius form is needed to identify the moment map with an evaluation map equivariantly, see Remark \ref{rem:symmetry}.
\end{rem}


\section{Representations of Decorated Preprojective Algebras}  \label{decorated preproj}

\subsection{Recollections on Preprojective Algebras}
Given a quiver $Q= (Q_{0}, Q_{1}, s, t)$ one can define the opposite quiver to be $Q^{op} := (Q_{0}, Q_{1}, t, s)$, which has the same vertices but with all arrows reversed. From this the double quiver is $\overline{Q} := (Q_{0}, Q_{1} \sqcup Q_{1}, s \sqcup t, t \sqcup s)$, which has the same vertices as $Q$ but with both the arrows and the inverted arrows and consequently is insensitive to the orientation of the quiver $Q$. If $\alpha \in Q_{1}$, we write $\alpha^*$ for the corresponding arrow in $Q_{1}^{op}$, and vice versa. 

The representation theory of the double quiver can be described in terms of the representation theory of the original quiver. Fix a dimension vector $d$ and denote by $\Rep_{d}(Q)$ the vector space of representations of $Q$ with dimension vector $d$. Then
\begin{align*}
\Rep_{d}(\overline{Q}) &\overset{\cong}{\longrightarrow} \Rep_{d}(Q) \oplus \Rep_{d}(Q^{op}) \\
\rho &\longmapsto  \rho \mid_{Q} \oplus \rho \mid_{Q^{op}}.
\end{align*}
One has the non-degenerate pairing
$$
\Rep_{d}(Q) \times \Rep_{d}(Q^{op}) \rightarrow k \hspace{1cm} (\rho, \rho') \mapsto \sum_{\alpha \in Q_{1}} tr(\rho(\alpha) \circ \rho'(\alpha^{*})) 
$$
which induces a $k$-linear isomorphism
$$
\Rep_{d}(Q^{op}) \cong \Rep_{d}(Q)^* \ \ (V_{i}, \rho_{\alpha}) \mapsto \left ( \psi \mapsto \sum_{\alpha \in Q_{1}} tr( \psi (\alpha^{*}) \circ \rho_{\alpha}(\alpha)) \right ).
$$
We conclude that,
$$
\Rep_{d}(\overline{Q}) \cong  \Rep_{d}(Q) \oplus \Rep_{d}(Q^{op}) \cong \Rep_{d}(Q) \oplus \Rep_{d}(Q)^{*} =: T^{*}(\Rep_{d}(Q)).
$$



For any finite-dimensional vector space $V$, the vector space $V \oplus V^{*} \cong T^{*}(V)$ carries a canonical symplectic form. Taking $V := \Rep_{d}(Q)$, the symplectic form on $\Rep_{d}(\overline{Q}) \cong T^{*}(\Rep_{d}(Q))$ is explicitly given by:
$$
\omega_{d}: \Rep_{d}(\overline{Q}) \times \Rep_{d}(\overline{Q}) \rightarrow k
$$
$$
\omega_{d}(\rho, \rho') = \sum_{\alpha \in Q_{1}} tr( \rho(\alpha^{*}) \rho'(\alpha)) - tr(\rho(\alpha) \rho'(\alpha^{*}))
$$ 
where $tr: \End_{k}( \oplus_{i \in Q_{0}} k^{d_{i}}) \rightarrow k$ is the usual trace of matrices. 

The action of $G_{d} = \prod_{i \in Q_{0}} GL_{d_{i}}(k)$ on $\Rep_{d}(Q)$ is extended to $\Rep_{d}( \overline{Q})$ by,
$$
g \cdot (x, \phi) = (g \cdot x, \phi \circ (d_{x} \varphi_{g} )^{-1} )
$$
where $d_{x} \varphi_{g}$ is the differential of the action map. This action is Hamiltonian with $G_{d}$-equivariant moment map $\mu_{d}: \Rep_{d}(\overline{Q}) \rightarrow \fg_{d}^{*}$ given by
\begin{align*}
\mu_{d}(\rho)(x) &= \frac{1}{2} \omega( \rho, x \cdot \rho)
= \frac{1}{2} \sum_{\alpha \in Q_{1}} tr( \rho(\alpha^{*}) x \cdot \rho(\alpha) - x \cdot \rho(\alpha)\rho(\alpha^{*})) \\
&= \frac{1}{2} \sum_{\alpha \in Q_{1}} tr ( x \cdot \rho ([\alpha, \alpha^{*}])),
\end{align*}
where $\fg_{d} := \text{Lie}(G_{d})$. We view the moment map as valued in $\fg_{d}$ by post-composing with twice the trace form,
$$
\xymatrix{
\rho \in \Rep_{d}(\overline{Q}) \ar[r] &  \frac{1}{2} tr \left ( - \cdot \sum_{\alpha \in Q_{1}} \rho( [\alpha, \alpha^{*}] \right ) \in \fg^{*} \ar[r] &   \sum_{\alpha \in Q_{1}} \rho( [\alpha, \alpha^{*}] ) \in \fg  
}
$$
Since the second map is an isomorphism, we can identify the moment map with evaluation at $r := \sum_{\alpha \in Q_{1}} [\alpha, \alpha^{*}]$. 

Returning to the algebraic setting, representations of $\overline{Q}$ are right modules over the path algebra 
$$
P(\overline{Q}) = T_{k Q_{0}}(k Q_{1} \oplus k Q_{1}^{op}).
$$
The condition that a representation, $\rho$, is in the zero fiber of the moment map is precisely the condition that $\rho(r) = 0$, i.e. the module descends to a module for the quotient algebra 
$P(\overline{Q})/ \langle r \rangle,$
where $\langle r \rangle$ denotes the two-sided ideal generated by $r$. This algebra was originally defined and studied by Gelfand and Ponomarev, prior to the moment map perspective, since in the Dynkin case it contains the path algebra as a subalgebra as well as all indecomposable preprojective modules for the path algebra \cite{Gelfand}. 

Summarizing this discussion:
\begin{defn}
The \emph{preprojective algebra} associated to a quiver $Q$ and a field $k$ is
$$
\Pi_{k}(Q) := P(\overline{Q})/ \left \langle \sum_{\alpha \in Q_{1}} [\alpha, \alpha^{*}] \right \rangle.
$$
We will usually write $\Pi(Q) := \Pi_{k}(Q)$ suppressing the dependence on the field. 
\end{defn}

\begin{prop} \label{ord moment map}
With the identifications induced from the non-degenerate pairings,
$$
\Rep_{d}(Q) \otimes \Rep_{d}(Q^{op}) \rightarrow k
\hspace{1cm} (\rho, \rho') \mapsto \sum_{\alpha \in Q_{1}} tr( \rho(\alpha) \circ \rho'(\alpha^{*}) )
$$
and
$$
\fg_{d} \otimes \fg_{d} \rightarrow k \hspace{1cm} (X, Y) \mapsto tr(XY)
$$
the following diagram commutes,
$$
\xymatrix{
& T^{*}\Rep_{d}(Q) \ar[rr]^{\mu_{d}} & & \fg_{d}^{*} \ar[rrd]^{\cong} & & \\
\Rep_{d}(\overline{Q}) \ar[ru]^{\cong} \ar[rrrrr]^{\rho \mapsto \rho(\sum_{\alpha \in Q_{1}} [\alpha, \alpha^{*}])} &&&&& \fg_{d}
}
$$
and hence one can view the moment map as evaluation on the element $\sum_{\alpha \in Q_{1}} [\alpha, \alpha^{*}]$. 
\end{prop}

\begin{cor}
The set of representations of dimension vector $d$ in $\Pi(Q)$ is the zero fiber of the moment map $\mu_{d}$.
\end{cor} 

\begin{rem}
Crucially, one needs to identify $\Rep_{d}(\overline{Q}) \cong T^{*}\Rep_{d}(Q)$ as symplectic representations of $G_{d}$. The symplectic structure on $\Rep_{d}(\overline{Q})$ is defined through the equivalence. The preservation of the $G_{d}$-action follows since $G_{d}$ acts on $\Rep_{d}(\overline{Q})$ by conjugation and the pairing uses the trace, which is conjugation-invariant. Additionally, $\fg_{d} \cong \fg^{*}_{d}$ as $G_{d}$-representations with the adjoint and coadjoint actions.  
\end{rem}



\subsection{Decorated Preprojective Algebras}

In this subsection we extend the definition of the preprojective algebra of a quiver to the decorated setting. We justify the term preprojective algebra using the moment map for the action of $G = \oplus_{i \in Q_{0}} GL_{A_{i}}(V_{i})$ on the cotangent bundle of the space of all representations. Namely, we show that the set of representations of the decorated preprojective algebra is the zero fiber of the moment map. 

Recall the decorated path algebra for a decorated quiver $(Q, D_{Q})$ with decoration $D_{Q} = ( A_{i}, M^{\alpha} )$ is $P(Q, D_{Q}) = T_{\oplus A_{i}}(\oplus_{\alpha} M^{\alpha})$. The decorated preprojective algebra is a quotient of the decorated path algebra for the \emph{double} decorated quiver, which we define now.

\begin{defn}
Given a decorated quiver $(Q, D_{Q})$, the \emph{opposite} decorated quiver $(Q^{op}, D_{Q}^{op})$ is defined by
$$
D_{Q}^{op} := ( A_{i}, M^{\alpha^{*}}), \ \ M^{\alpha^{*}}:= \Hom_{A_{s(\alpha)} \otimes A_{t(\alpha)}}(M^{\alpha}, A_{s(\alpha)} \otimes A_{t(\alpha)}) 
$$
and the \emph{double} decorated quiver $(\overline{Q}, \overline{D_{Q}})$ is defined by $\overline{D_{Q}} := D_{Q} \cup_{D_{Q_{0}}} D_{Q}^{op}$.
\end{defn} 

Suppose $D_{Q}$ satisfies condition (F) and partition the arrows into $Q_{1}' \sqcup Q_{1}''$ where the former has label $M^{\alpha} =A_{s(\alpha)} \otimes A_{t(\alpha)}$ and the latter has label $M^{\alpha} = A_{s(\alpha)}= A_{t(\alpha)}$. Additionally, pick a basis $\{ e^{i}_{j} \}_{j}$ for $A_{i}$ as a vector space and use the Frobenius form $\lambda_{i}: A_{i} \rightarrow k$ to build a dual basis $\{ f^{i}_{l} \}_{l}$ such that $\lambda_{i}(e^{i}_{j} f^{i}_{l}) = \delta_{j, l}$, for all $i \in Q_{0}$.

\begin{defn} \label{def:dec preproj alg}
Using notation from the previous paragraph, the \emph{decorated preprojective algebra} $\Pi_{k}(Q, D_{Q})$ is defined to be the quotient of the decorated path algebra of the double $P(\overline{Q}, \overline{D_{Q}})$ by the two-sided ideal generated by
$$
r:= \sum_{\alpha \in Q_{1}'}
 \left ( \sum_{i} e^{s(\alpha)}_{i} \otimes 1 \otimes f^{s(\alpha)}_{i}  - \sum_{j} e^{t(\alpha)}_{j} \otimes 1 \otimes f^{t(\alpha)}_{j} \right )
 + 
 \sum_{\alpha \in Q_{1}''} [1_{M^{\alpha}}, 1_{M^{\alpha^{*}}}].
$$
We will usually write $\Pi(Q, D_{Q}) := \Pi_{k}(Q, D_{Q})$, suppressing the dependence on the field. 
\end{defn}

The usual preprojective relation in $k \overline{Q}_{1} \otimes_{kQ_{0}} k\overline{Q}_{0}$ is $kQ_{0}$-central, since it is concentrated on the diagonal. Similarly, we want $r$ to be an $(\oplus A_{i})$-central element of $\oplus M^{\alpha} \otimes_{(\oplus A_{i})} \oplus M^{\alpha}$, so we choose for each $i$, the dual element to the Frobenius pairing, $k \rightarrow A_{i} \otimes A_{i}$, which when written in bases $\{ e^{i}_{j} \}$ and dual bases $\{ f^{i}_{j} \}$ is the $A_{i}$-central element $\sum_{j} e^{i}_{j} \otimes_{k} f^{i}_{j}$. For arrows in $Q_{1}''$, $M^{\alpha} \otimes_{A} M^{\alpha^*} \cong A \otimes_{A} A \cong A$ has the $A$-central element $1_{M^{\alpha} \otimes_{A} M^{\alpha^*}} = 1_{M^{\alpha}} \otimes 1_{M^{\alpha^*}}$, so one need not use the Frobenius form in the definition of the decorated preprojective algebra, yet we still require the existence of such a form for desirable properties, see Remark \ref{rem:symmetry}.

\begin{rem} \label{rem:well-defined}
The above definition depends only on the decorated quiver $(Q, D_{Q})$ and the field $k$. 
To see this, observe that for each $\alpha: i \rightarrow j \in Q'_{1}$ the summand of $r$ corresponding to $\alpha$ lies in
$$
A_{i} \otimes A_{i} \cong A_{i} \otimes k \otimes A_{i} \subset A_{i} \otimes A_{j} \otimes A_{i}
\cong (A_{i} \otimes A_{j}) \otimes_{A_{j}} (A_{j} \otimes A_{i}) \cong M^{\alpha} \otimes_{A_{j}} M^{\alpha^{*}}
$$
and hence it suffices to show that, in any Frobenius algebra $A$ with basis $\{ e_{i} \}$ and dual basis $\{ f_{i} \}$, the ideal generated by $c := \sum_{i} e_{i} \otimes f_{i}$ doesn't depend on the choice of basis nor the form. For independence of basis, notice that $c$ can be represented canonically as the comultiplication of the unit. Next, observe that a Frobenius form can be viewed as the image of $1_{A}$ under a left $A$-module isomorphism $\varphi: A \rightarrow A^{*}$. Hence, two Frobenius forms $\lambda_{1}, \lambda_{2}: A \rightarrow k$ each give rise to $\varphi_{1}, \varphi_{2}: A \rightarrow A^{*}$ such that $\varphi_{2}^{-1} \circ \varphi_{1}: A \rightarrow A$ is a left $A$-module isomorphism. Such a map is determined by $u := \varphi_{2}^{-1} \circ \varphi_{1}(1_{A})$, which is invertible. Hence 
$$
\lambda_{1} = \varphi_{1}(1) = \varphi_{2}(u) = \lambda_{2} \circ R_{u}
$$ 
where $R_{u}$ is right multiplication by $u \in A^{\times}$. We conclude that the canonical sums $c_{1}$ and $c_{2}$ satisfy $c_{1} = c_{2}u$ and hence generate the same two-sided ideal in $A \otimes_{k} A$. \\
\end{rem}

\begin{exam}
This phenomenon is already visible in the setting of the four element algebra $A = \bF_{2}[x]/(x^{2})$. Here $\lambda: A \rightarrow \bF_{2}$ is Frobenius if $(x) \not \subset \text{ker}(\lambda)$ or equivalently if $\lambda(x) \neq 0$ and hence is 1. So there are two Frobenius forms: $\lambda_{1}(1) = 0$ and $\lambda_{2}(1) = 1$. They are related by the invertible element $1+x \in A$ as $\lambda_{1}(1) = 0 = \lambda_{2}( 1(1+x) )$ and $\lambda_{1}(x) = 1= \lambda_{2}(x(1+x))$. The basis $\{ 1, x \}$ has dual basis $\{ x, 1 \}$ under $\lambda_{1}$ and dual basis $\{ 1, 1+x \}$ under $\lambda_{2}$. Hence $c_{1} = 1 \otimes x + x \otimes 1$ and 
$$
c_{2} = 1 \otimes x + x \otimes (1 + x) = 1 \otimes x(1+x) + x \otimes 1(1+x) = c_{1}(1+x)
$$
both generate the ideal $(c_{1})= \{ c_{1}, c_{2}, c_{1}+c_{2} \} \subset A \otimes A$. 
\end{exam}

\begin{rem} \label{rem: motivation for dec preproj definition}
This explains the reason for the dichotomy between types of arrows and the resulting asymmetry in the definition. Namely, it's an artifact of the identifications,
$$
\End_{A}(A) \cong A \cong M^{\alpha} \otimes_{A} M^{\alpha^{*}}  \hspace{1cm} id \mapsto 1 \mapsto 1 \otimes_{A} 1
$$
as $k$-bimodules for $M^{\alpha} ={}_{A}A_{A} \cong M^{\alpha^{*}}$ and 
$$
\End_{k}(A) \cong A \otimes_{k} A^{*} \cong M^{\alpha} \otimes_{k} M^{\alpha^{*}} 
\hspace{1cm} id \mapsto c = \sum_{i} e_{i} \otimes_{k} f_{i} \mapsto \sum_{i} e_{i} \otimes 1 \otimes f_{i}
$$
as $A$-bimodules for $M^{\alpha} ={}_{A}A_{k}$ an $(A, k)$-bimodule. 
\end{rem}

The key result motivating the definition is the following. Recall fixing $(Q, D_{Q})$ and a dimension vector $d$, we have the action of $G_{d} = \prod_{i \in Q_{0}} GL_{A_{i}}(A_{i}^{d_{i}})$ on $\Rep_{d}(Q, D_{Q})$ with Lie algebra $\fg_{d} = \oplus_{i \in Q_{0}} \fg \fl_{A_{i}}(A_{i}^{d_{i}})$. Let $\mu_{d} : T^{*}(\Rep_{d}(Q, D_{Q})  \rightarrow \fg_{d}^{*}$ denote the moment map corresponding to the Hamiltonian action on the cotangent space coming from the action on the base. In this setting, we can generalize Proposition \ref{ord moment map} as follows:

\begin{thm} \label{thm moment map}
With the identifications induced from the non-degenerate pairings,
$$
\Rep_{d}(Q, D_{Q}) \otimes \Rep_{d}(Q^{op}, D_{Q}^{op}) \rightarrow k
\hspace{3cm}
\fg_{d} \otimes \fg_{d} \rightarrow k 
\hspace{.5cm} 
$$
$$
(\rho, \rho') \mapsto 
\sum_{\alpha \in Q_{1}} \lambda_{A_{s(\alpha)}} \circ tr_{A_{s(\alpha)}}( \rho(1_{M^{\alpha}}) \circ \rho'(1_{M^{\alpha^{*}}}) )
\hspace{.8cm}
(X, Y) \mapsto \sum_{i \in Q_{0}} \lambda_{i} \circ tr_{A_{i}}(XY) 
$$
the following diagram commutes,
$$
\begin{xy}
(20, 15)*+{T^{*}\Rep_{d}(Q, D_{Q})} = "A" ; 
(70, 15)*+{\fg_{d}^{*}} = "B" ;
(0,0)*+{\Rep_{d}(\overline{Q}, \overline{D_{Q}})} = "C" ;
(90, 0)*+{\fg_{d}} = "D" ;
{\ar@{->}^{\mu_{d}} "A"; "B"}  ;
{\ar@{->}^{\cong} "C"; "A"}  ;
{\ar@{->}^{\cong} "B"; "D"}  ;
{\ar@{->}^{(\rho(1_{M^{\alpha}}), \rho(1_{M^{\alpha^{*}}}) ) \mapsto \rho(r)} "C"; "D"}  ;
\end{xy}
$$
and hence one can view the moment map as evaluation on the element $r$. 
\end{thm}

Note that $\rho(r)$ is, a priori, an element of $\oplus_{i \in Q_{0}} \fg \fl_{k}(A_{i}^{d_{i}})$, but in fact is $\oplus A_{i}$-linear since $e_{i}r$ is $A_{i}$-central, hence an element of $\fg_{d}$. Additionally, since each pairing is $G_{d}$-invariant, we have identified the moment map with evaluation on $r$, $G_{d}$-equivariantly
 
\begin{rem}
For each $A_{i}$, to apply $tr_{A_{i}}: \Mat_{d_{i}}(A_{i}) \rightarrow A_{i}$ one needs to precompose with an identification $\fg \fl_{A_{i}}(A_{i}^{d_{i}}) \cong \Mat_{d_{i}}(A_{i})$. Although $tr_{A_{i}}$ depends on the choice of identification, $\lambda_{A_{i}} \circ tr_{A_{i}}$ does not, as $\lambda_{A_{i}}$ is symmetric. 
\end{rem}
 
\begin{cor}
The set of representations of dimension vector $d$ in $\Pi(Q, D_{Q})$ is the zero fiber of the moment map $\mu_{d}$.
\end{cor} 

The remainder of this section is dedicated to a proof of the above theorem, justifying the terminology preprojective algebra in a geometric context. 

Since $Q$ is a finite quiver, $r =  \sum_{i \in Q_{0}} e_{i} r$ can be computed locally at each vertex, where it is given by a sum over all incoming arrows. Hence it suffices to prove the above result in the context of a quiver with a single decorated arrow. Such decorated quivers satisfying condition (F) take the form,
$$
\xymatrix{
\text{(I)} \ \  A_{1} \ar[rr]^{ A_{1} \otimes_{k} A_{2} } & & A_{2} 
&    
\text{(II)} \ \ A \ar[rr]^{A} &  & A
}
$$
$$
\xymatrix{
\text{(III)} \ \ A \ar[rr]^{A \otimes_{k} A} & & A &
\text{(IV)} \ \ A \circlearrowleft A   
& & 
\text{(V)} \ \ A \circlearrowleft A \otimes_{k} A
}
$$
where $A, A_{1},$ and $A_{2}$ are symmetric Frobenius algebras. The formula for the moment map in cases (IV) and (V) follow from the (II) and (III) cases, respectively, by post-composing with the map $d\Delta^{*}: \fg_{d}^{*} \oplus \fg_{d}^{*} \rightarrow \fg_{d}^{*}$, induced from the diagonal map $\Delta: G_{d} \rightarrow G_{d} \times G_{d}$. Next notice (III) is a special case of (I), when $A_{1} = A_{2} = A$. Moreover the definition of $e_{1}r$ uses only the unit $k \subset A_{2}$ and hence one can further reduce to verifying the theorem in the two cases:
$$
\xymatrix{
Q' := A \ar[rr]^{_{A} A_{k}} & & k 
& \text{or} &  
Q'' := A \ar[rr]^{A} &  & A.
}
$$
In these cases, representations of the double decorated quiver have a concrete description, owing to the cyclicity of the bimodules. That is, if $d=(d_{1}, d_{2})$, then
\begin{align*}
&\Rep_{d}( \overline{Q'}) \\
& \ = 
\Hom_{A-\text{mod}}( {}_{A} A_{k}, \Hom_{k}(A^{d_{1}}, k^{d_{2}})) \oplus \Hom_{\text{mod}-A}( {}_{k} A_{A}, \Hom_{k}(A^{d_{1}}, k^{d_{2}}) ) \\
& \ \cong \Hom_{k}(A^{d_{1}}, k^{d_{2}}) \oplus \Hom_{k}(k^{d_{2}}, A^{d_{1}})
\end{align*}
and
\begin{align*}
&\Rep_{d}(\overline{Q''}) \\
& \ = \Hom_{A-\text{bimod}}( _{A} A_{A}, \Hom_{k}(A^{d_{1}}, A^{d_{2}})) \oplus \Hom_{A-\text{bimod}}( _{A} A_{A}, \Hom_{k}(A^{d_{2}}, A^{d_{1}}) ) \\
& \ \cong \Hom_{A}(A^{d_{1}}, A^{d_{2}}) \oplus \Hom_{A}(A^{d_{2}}, A^{d_{1}}).
\end{align*} 
In each case, the theorem reduces to much more concrete statements, given respectively by Proposition \ref{prop:momentmapcase} and Proposition \ref{prop:momentmapcase2}.

\begin{prop} \label{prop:momentmapcase}
With the identifications induced from the non-degenerate pairings,
\begin{align*}
\Hom_{k}(A^{d_{1}}, k^{d_{2}}) \otimes \Hom_{k}(k^{d_{2}}, A^{d_{1}}) \rightarrow k &
\hspace{1cm} (f, g) \mapsto tr(f \circ g)  \\
\End_{A}(A^{d_{1}}) \otimes \End_{A}(A^{d_{1}}) \rightarrow k & \hspace{1cm} (X, Y) \mapsto \lambda_{A} \circ tr_{A}(X \circ Y) \\
\End_{k}(k^{d_{2}}) \otimes \End_{k}(k^{d_{2}}) \rightarrow k & \hspace{1cm} (X, Y) \mapsto tr(X \circ Y)
\end{align*}
the following diagram commutes:
$$
\begin{xy}
(8, 15)*+{\Hom_{k}(A^{d_{1}}, k^{d_{2}}) \oplus \Hom_{k}(A^{d_{1}}, k^{d_{2}})^{*}} = "A" ; 
(68, 15)*+{\End_{A}(A^{d_{1}})^{*}\oplus \End_{k}(k^{d_{2}})^{*}} = "B" ;
(0,0)*+{\Hom_{k}(A^{d_{1}}, k^{d_{2}}) \oplus \Hom_{k}(k^{d_{2}}, A^{d_{1}})} = "C" ;
(76, 0)*+{\End_{A}(A^{d_{1}}) \oplus \End_{k}(k^{d_{2}})} = "D" ;
{\ar@{->}^{\hspace{.75cm} \mu_{d}} "A"; "B"}  ;
{\ar@{->}^{\cong} "C"; "A"}  ;
{\ar@{->}^{\cong} "D"; "B"}  ;
{\ar@{->}^{\hspace{.75cm}(\rho({}_{A}1_{k}), \rho({}_{k}1_{A})) \mapsto \rho(r)} "C"; "D"}  ;
\end{xy}
$$
where $r = 1 \otimes_{A} 1 - \sum e_{i} \otimes_{k} f_{i}$.
\end{prop}

Roughly speaking, in the ordinary case two applications of the trace pairing cancel, yielding a simplified expression for the moment map. From this perspective, the main technical difficulty arises from the need to relate the trace of an $A$-linear map (regarded as a $k$-linear map) with the Frobenius form. 

More precisely, if $(f, g) \in \Hom_{k}(A^{d_{1}}, k^{d_{2}}) \oplus \Hom_{k}(k^{d_{2}}, A^{d_{1}})$ and  \\
$(\varphi, \psi) \in \End_{A}(A^{d_{1}}) \oplus \End_{k}(k^{d_{2}})$ 
then in the above diagram,
$$
\begin{xy}
(10, 25)*+{(f, tr(g \circ -))} = "A" ; 
(80, 25)*+{ tr(f \circ g \circ -) + tr(-g \circ f \circ -)} = "B" ;
(80, 15)*+{\lambda_{A} \circ tr_{A}(\varphi \circ -) + tr(\psi \circ -)} = "B'" ;
(80, 20) *+{ \rotatebox{90}{=} } = "B''" ;
(0,0)*+{(f, g)} = "C" ;
(90, 0)*+{(\varphi, \psi)} = "D" ;
{\ar@{|->} "A"; "B"}  ;
{\ar@{|->} "C"; "A"}  ;
{\ar@{|->} "D"; "B'"}  ;
{\ar@{|->}^{?} "C"; "D"}  ;
\end{xy}
$$
In our setting, $f = \rho({}_{A}1_{k})$ and $g= \rho({}_{k}1_{A})$ and so we need to find $\varphi \in \End_{A}(A^{d_{1}})$ and $\psi \in \End_{k}(k^{d_{2}})$ such that,
$$
tr( \rho(-{}_{A}1 \otimes_{k} 1_{A} ) \circ -)
= \lambda_{A} \circ tr_{A}(\varphi \circ -)
$$
and
$$
tr( \rho( {}_{k} 1 \otimes_{A} 1_{k} ) \circ -)
= tr(\psi \circ -).
$$
Clearly, $\psi = \rho({}_{k}1 \otimes_{A} 1_{k} )$ and hence the moment map at the vertex decorated with $k$ is evaluation at ${}_{k} 1 \otimes_{A} 1_{k}$, playing the role of $\alpha \circ \alpha^{*}$ in the ordinary case.

However, $\varphi \neq \rho(-1 \otimes_{k} 1)$, which need not be $A$-linear in general. To remedy the situation, we will define a map $\Phi: \End_{k}(A) \rightarrow \End_{A}(A)$ such that 
$$
\varphi = \Mat_{d}(\Phi)(\rho(-1 \otimes_{k} 1))= \rho \left ( -\sum e_{i} \otimes f_{i} \right ),
$$ 
completing the proof of the proposition. 

We now present a series of lemmas about Frobenius algebras. Since the subject matter is classical and the proofs are elementary, the following is probably well known. Nonetheless we provide short proofs for completeness of exposition.

\begin{lem} \label{Def Phi}
Let $A$ be a Frobenius algebra with multiplication $\mu$ and form $\lambda: A \rightarrow k$. Define $\Phi: \End_{k}(A) \rightarrow \End_{A}(A)$ by $\Phi(\phi)(a) = \sum_{i} \phi(e_{i}) f_{i} a$, where $\{ e_{i} \}$ is a basis for $A$ over $k$ with dual basis $\{ f_{i} \}$. Then
$$
\xymatrix{
\End_{k}(A) \ar[d]_{\cong} \ar[rd]^{\Phi} \ar[rrr]^{\text{tr}} & & & k \\
A \otimes_{k} A^{*} \ar[d]_{\cong} & \End_{A}(A) \ar[d]_{\cong}   \\
A \otimes_{k} A \ar[r]^{\mu} & A  \ar[rruu]^{\lambda}
}
$$
commutes. In particular, $tr(\phi) = \lambda( \Phi( \phi)(1) )$ for $\phi \in \End_{k}(A)$.
\end{lem}


\begin{proof}
$\Phi$ is \emph{defined} so the diagram commutes. In more detail,
$$
\xymatrix@C-1.7mm{
\End_{k}(A) \ar[rrr] & & & A \otimes_{k} A^{*} \ar[rrr] & & & A \otimes_{k} A \ar[rrr] & & & A\hspace{1.5cm} &
}
$$
$$
\xymatrix@-.9mm{
\phi \ar@{|->}[rr] & &
 \sum_{i} \phi(e_{i}) \otimes_{k} \lambda \circ \mu(-, f_{i}) \ar@{|->}[r] & \sum_{i} \phi(e_{i}) \otimes_{k} f_{i} \ar@{|->}[rr] & & \sum_{i} \phi(e_{i}) f_{i}
}
$$
Since $\Phi(\phi)(1) := \sum_{i} \phi(e_{i}) f_{i}$ the lower left half of the diagram commutes. \\
\\
For the rest of the diagram, compute $tr$ using the basis $\{ e_{i} \}$. If $\phi(e_{i})$ is expanded as $\sum_{j} \alpha_{ij} e_{j}$ then,
$$
\alpha_{ii} = \lambda \circ \mu ( \sum_{j} \alpha_{ij} e_{j}, f_{i})
= \lambda \circ \mu ( \phi(e_{i}), f_{i}).
$$ 
and hence $tr(\phi) = \lambda \circ \mu ( \sum_{i} \phi( e_{i} ), f_{i}) = \lambda(\Phi(\phi)(1))$ as desired. 
\end{proof}

\begin{lem} \label{lem:factortrace}
The following diagram commutes:
$$
\xymatrix{
\End_{k}(A^{d}) \ar[d]^{\cong} \ar@/^2pc/[rrrrrd]_{tr} \\
\Mat_{d}(\End_{k}(A)) \ar[d]^{\Mat_{d}(\Phi)}
\ar[rrrr]^{\Mat_{d}(tr)} & & & & \Mat_{d}(k) \ar[r]^-{tr} & k \\
\Mat_{d}(\End_{A}(A)) \ar[rr]_-{\cong} & & \Mat_{d}(A) \ar[rru]^{\Mat_{d}(\lambda)} \ar[r]_-{tr_{A}} & A \ar[rru]_{\lambda}
}
$$
\end{lem}
 
\begin{proof}
Applying $\Mat_{d}(-)$ to the commutative diagram in Lemma \ref{Def Phi} gives commutativity of the lower left trapezoid. The commutativity of the upper triangle is the fact that the trace of a block matrix is the sum of the traces of the diagonal blocks.  
The final quadrilateral commutes since if $M = (a_{i,j}) \in Mat_{d}(A)$ then \\
\hspace*{2cm}
$\displaystyle
\lambda \circ tr_{A}(M) = \sum_{i} \lambda(a_{i,i}) = tr((\lambda(a_{i,j}))) 
= tr \circ \Mat_{d}(\lambda)(M). 
$
\end{proof} 

\begin{lem} \label{lem:restrictiondual}
The following diagram commutes:
$$
\xymatrix{
\End_{k}(A^{d}) \ar[d]^{\Mat_{d}(\Phi)} \ar[r]^{\cong} & \End_{k}(A^{d})^{*} \ar[d]^{rest.} \\
\End_{A}(A^{d}) \ar[r]^{\cong} & \End_{A}(A^{d})^{*} 
}
$$
where the horizontal identifications are given by the non-degenerate pairings,
$$
\End_{k}(A^{d}) \otimes \End_{k}(A^{d}) \rightarrow k \hspace{1cm} (\phi_{1} \otimes \phi_{2}) \mapsto tr(\phi_{1} \circ \phi_{2})
$$
$$
\hspace{.85cm} \End_{A}(A^{d}) \otimes \End_{A}(A^{d}) \rightarrow k \hspace{1cm} (\psi_{1} \otimes \psi_{2}) \mapsto \lambda \circ tr_{A}(\psi_{1} \circ \psi_{2}).
$$
\end{lem}

\begin{proof}
Commutativity amounts to the equivalence
$$
tr (\phi \circ \psi) = \lambda \circ tr_{A} ( \Mat_{d}(\Phi)(\varphi) \circ \psi) 
$$
for all $\phi \in \End_{k}(A^{d})$ and $\psi \in \End_{A}(A^{d})$. By Lemma \ref{lem:factortrace}, one can factor trace as,
$$
tr =  \lambda \circ tr_{A} \circ \Mat_{d}(\Phi)
$$
and hence it suffices to show $\Mat_{d}(\Phi)(\phi \circ \psi) =  \Mat_{d}(\Phi)(\phi) \circ \psi$. By comparing components of the mappings, one reduces to the case $d=1$, where the result follows from the fact that $\Phi$ is an $A$-bimodule map and $\psi$ can be viewed as an element of $A$. 
\end{proof}

Note that this completes the proof of Proposition \ref{prop:momentmapcase}, since 
$$
\Mat_{d}(\Phi) \rho( 1 \otimes_{k} 1 ) 
=\rho ( \Phi( 1 \otimes_{k} 1 ) ) 
=\rho \left(\sum e_{i} \otimes_{k} f_{i} \right).
$$
  
\begin{prop} \label{prop:momentmapcase2}
With the identifications induced from the non-degenerate pairings,
\begin{align*}
&\Hom_{A}(A^{d_{1}}, A^{d_{2}}) \otimes
\Hom_{A}(A^{d_{2}}, A^{d_{1}}) \rightarrow k 
\hspace{.6cm} \End_{A}(A^{d_{j}}) \otimes \End_{A}(A^{d_{j}}) \rightarrow k  \\  
(&\rho, \rho') \mapsto \sum_{\alpha \in Q_{1}} \lambda_{A} \circ tr_{A}( \rho(1_{M^{\alpha}}) \circ \rho'(1_{M^{\alpha^{*}}}) )  \hspace{.5cm}
 (X, Y) \mapsto \sum_{i \in Q_{0}} \lambda_{A} \circ tr_{A}(XY)
\end{align*}
for $j \in \{ 1, 2 \}$, the following diagram commutes:
$$
\begin{xy}
(8, 15)*+{\Hom_{A}(A^{d_{1}}, A^{d_{2}}) \oplus \Hom_{A}(A^{d_{1}}, A^{d_{2}})^{*}} = "A" ; 
(68, 15)*+{\End_{A}(A^{d_{1}})^{*}\oplus \End_{A}(A^{d_{2}})^{*}} = "B" ;
(0,0)*+{\Hom_{A}(A^{d_{1}}, A^{d_{2}}) \oplus \Hom_{A}(A^{d_{2}}, A^{d_{1}})} = "C" ;
(76, 0)*+{\End_{A}(A^{d_{1}}) \oplus \End_{A}(A^{d_{2}})} = "D" ;
{\ar@{->}^{\hspace{.75cm} \mu_{d}} "A"; "B"}  ;
{\ar@{->}^{\cong} "C"; "A"}  ;
{\ar@{->}^{\cong} "D"; "B"}  ;
{\ar@{->}^{\hspace{.75cm}(\rho(_{A}1_{A}), \rho(_{A}1_{A})) \mapsto \rho(r)} "C"; "D"}  ;
\end{xy}
$$
where $r = e_{1}r + e_{2} r = 1 \otimes^{2}_{A} 1 - 1 \otimes^{1}_{A} 1 = [ _{A} 1_{A}, _{A} 1_{A} ]$ where the labels $1$ and $2$ are used to distinguish between the two copies of $A$, one at each vertex.
\end{prop}

The difficulty in the previous proposition was relating, for a given $\alpha: i \rightarrow j$, the expressions $\lambda_{A_{i}} \circ tr_{A_{i}}$ and $\lambda_{A_{j}} \circ tr_{A_{j}}$. In this case, $A_{i} = A = A_{j}$ and hence that issue is not present. 

\begin{proof}
The result follows from showing that the moment map still has the same form as before. Whenever $G$ acts on a vector space $V$, the moment map for the action of $G$ on $T^{*}(V) = V \oplus V^{*}$ is given by 
$$
\mu: T^{*}(V) \rightarrow \fg^{*} \hspace{1cm}
\mu((v, \varphi))(X) = \varphi( X \cdot v). 
$$
In this case $\fg = \fg_{1} \oplus \fg_{2}$ is a sum and the action $X = (X_{1}, X_{2})$ on $v$ defined by differentiating the conjugation action is 
$$
X \cdot v = X_{2} \cdot v - v \cdot X_{1}.
$$
Therefore, 
$$
\begin{xy}
(10, 15)*+{(f, \lambda_{A} \circ tr_{A}(g \circ -))} = "A" ; 
(80, 15)*+{\lambda_{A} \circ tr_{A}( -g \circ f \circ -) +\lambda_{A} \circ tr_{A}( f \circ g \circ -)  } = "B" ;
(0,0)*+{(f, g)} = "C" ;
(90, 0)*+{(-g \circ f + f \circ g)} = "D" ;
{\ar@{|->} "A"; "B"}  ;
{\ar@{|->} "C"; "A"}  ;
{\ar@{|->} "D"; "B"}  ;
\end{xy}
$$
and hence the composition is giving by commutator and plugging in $f = \rho(_{A}1_{A})$ and $g= \rho(_{A}1_{A})$ gives the desired result.  
\end{proof}

This completes the proof of Theorem \ref{thm moment map}. 

\begin{rem} \label{rem:weakerassumptions}
A more general definition of preprojective algebra appears in \cite{Julian}, which puts a weaker restriction on the decorations than condition (F), although the Frobenius condition still arises naturally as we explain below. 
When can the representations of such algebras be identified with the zero fiber of a moment map, as in Theorem \ref{thm moment map}?

In more detail, in \cite{Julian}, decorations are required instead to have the property that each $M^{\alpha}$ is projective as a left $A_{t(\alpha)}$-module and a projective as a right $A_{s(\alpha)}$-module. In this setting, $M^{\alpha}$ has both left and right duals, and it is further required that these duals agree,
$$
\Hom_{A_{s(\alpha)}^{\text{op}}}(M^{\alpha}, A_{s(\alpha)})
\cong 
\Hom_{A_{t(\alpha)}}(M^{\alpha}, A_{t(\alpha)})
$$ 
as $(A_{t(\alpha)}, A_{s(\alpha)})$-bimodules. One says $M^{\alpha}$ is \emph{dualizable} in this case, and one says $D_{Q}$ is \emph{dualizable} if each $M^{\alpha}$ is dualizable. 

Notice that if a quiver has an arrow labelled with $M :={}_{A}A_{k}$ then dualizability at that arrow says,
$$
A \cong \Hom_{A}(M, A) \cong \Hom_{k}(M, k) \cong \Hom_{k}(A, k) 
$$
as an $(k, A)$-bimodule, and hence $A$ is Frobenius. More generally, if $M :={}_{A} A_{R}$ where $R$ is a $k$-algebra and $A$ is an algebra over $R$ then dualizability says
$$
A \cong \Hom_{R}(A, R)
$$
and hence $A$ is Frobenius relative to $R$, the assumption used in \cite{Presotto}. 
\end{rem}

\begin{rem} \label{rem:symmetry}
In the case $M ={}_{A} A_{A}$ dualizability is vacuous, yet we still need $A$ to be Frobenius in order to have a non-degenerate pairing 
$$
\Hom_{A}(A^{d_{1}}, A^{d_{2}}) \otimes \Hom_{A}(A^{d_{2}}, A^{d_{1}}) \overset{\circ}{\longrightarrow} \End_{A}(A^{d_{1}}) \overset{tr_{A}}{\longrightarrow} A \overset{\lambda}{\longrightarrow} k.
$$
Further, for the pairing to be $\text{GL}_{A}(A^{d_{1}})$-invariant then $\lambda$ needs to be symmetric. 
This justifies the assumption that the algebra at each vertex be symmetric Frobenius.
\end{rem}
   
 \begin{rem} \label{why free}
If $\alpha: i \rightarrow j$ is decorated with $M^{\alpha}$ an $(A_{i}, A_{j})$-bimodule, then
 $$
 \RHom_{A_{i} \otimes_{k} A_{j}^{op}}(M^{\alpha}, \Hom_{k}(V_{i}, V_{j})) \cong \Hom_{A_{i} \otimes_{k} A_{j}^{op}}(M^{\alpha}, \Hom_{k}(V_{i}, V_{j}))
 $$ 
 if $M^{\alpha}$ is a projective $(A_{i}, A_{j})$-bimodule or $\Hom_{k}(V_{i}, V_{j})$ is an injective $(A_{i}, A_{j})$-bimodule. Hence the restriction to free $A_{i}$-modules, $V_{i}$, implies that the underived representation spaces are actually derived, hence well-behaved and natural to study. One could alternatively consider $M^{\alpha}$ perfect and remove restrictions on the representations. Since our main objective is to study preprojective algebras, we prefer to allow for general $M^{\alpha}$.  
 \end{rem}
 
 


\section{Decorated Preprojective Algebras as Degenerations}
Many of the examples of decorated quivers in this paper have an interpretation as degenerations of ordinary quivers. In this section we explain this interpretation by first defining a notion of Frobenius deformation and proving that every finite-dimensional Frobenius algebra degenerates to a Frobenius algebra determined by a vector space with a non-degenerate bilinear form. In the commutative case, over an algebraically closed field of characteristic not two, there is a unique such degenerate algebra. Then we define a notion of Frobenius deformation for decorated quivers and use this to realize decorated preprojective algebras as degenerations. Finally, we conjecture that these degenerations are flat, and use upper semi-continuity of dimension under deformations to reduce the conjecture to a single verification for each quiver. 

\subsection{Frobenius Degenerations}
We want to degenerate (or deform) a decoration by degenerating both the $k$-algebras labelling each vertex and the bimodules labelling each edge. In the presence of condition (F), the degeneration of the decoration is determined by the degeneration of the algebras at the vertices. Moreover, to ensure condition (F) is satisfied when degenerating a decoration satisfying condition (F) we need to ensure each symmetric Frobenius algebra degenerates to a symmetric Frobenius algebra. 

We first recall the notions of formal and filtered deformations of associative algebras and bimodules. 

\begin{defn}
Let $A$ and $B$ be unital associative algebras over a field $k$. One says $B$ is a \emph{formal deformation} of $A$ and $A$ is a \emph{formal degeneration} of $B$ if there is a $k[[t]]$-algebra $D$ such that 
\begin{itemize}
\item $D/Dt \cong A$ as $k$-algebras and 
\item $D[t^{-1}] \cong B((t))$ as $k((t))$-algebras. 
\end{itemize}
Moreover, one says a formal deformation is \emph{flat} if
$D \cong A[[t]]$ as left $k[[t]]$-modules.
\end{defn}

\begin{defn}
Let $A$ and $B$ be unital associative $k$-algebras. One says $B$ is a \emph{filtered deformation} of $A$ and $A$ is a \emph{filtered degeneration} of $B$ if there exists:
\begin{itemize}
\item a filtration $B_{0} \subset B_{\leq 1} \subset B_{\leq 2} \subset \cdots \subset B_{\leq n} = B$,
\item a grading $A \cong \bigoplus_{i=0}^{n} A_{i}$, and
\item an isomorphism of graded algebras,
$$
gr(B) := \bigoplus_{i=0}^{n} B_{\leq i}/B_{\leq (i-1)} \cong \bigoplus_{i=0}^{n} A_{n} \cong A.
$$
\end{itemize}
\end{defn}

Any filtered deformation $B$ with $gr(B) \cong A$ gives rise to a formal deformation,
$$
\widehat{R}B :=
\left \{  \sum_{i \geq 0} b_{i} t^{i} : b_{i} \in B_{\leq i} \right \},
$$
called the Rees algebra. We present theory in the more general setting of formal deformations, but strengthen results by presenting proofs in the more specific setting of filtered deformations.\\
\\ 
Since our objective is to degenerate decorated quivers, we want to consider degenerations preserving condition (F). Notice that a (symmetric) Frobenius algebra can degenerate to a non-Frobenius algebra, e.g $k[[t]][x,y]/(x^{2}-t, y^{2} -t, xy)$ degenerates $k^{\oplus 3}$ to $k[x,y]/(x^{2}, y^{2}, xy)$, which is not Frobenius as it is not self-injective since $xA \rightarrow A$ sending $x \mapsto y$ does not extend to all of $A$.\\
\\
Hence, at each vertex we need to degenerate the symmetric Frobenius algebra to another symmetric Frobenius algebra. 

\begin{defn}
Let $D$ be a formal deformation of associative $k$-algebras from $A$ to $B$.
We say the deformation is \emph{Frobenius} if $D$ is Frobenius as a $k[[t]]$-algebra and $A$ and $B$ are Frobenius as $k$-algebras. Further we say the deformation is \emph{symmetric Frobenius} if $D$, $A$, and $B$ are all symmetric Frobenius. 
\end{defn}

\begin{rem} \label{rem:controllingdeformations}
This definition appears in \cite{Presotto} and the notion was studied in \cite{Connes}, where it is shown that the cohomology of the dg-Lie algebra controlling Frobenius deformations of a fixed Frobenius algebra $(A, \lambda)$ is the cyclic cohomology of $A$, see \cite{Connes}.
\end{rem}

Notice that we do \emph{not} require that the isomorphisms $A \cong D/tD$ or $B((t)) \cong D[t^{-1}]$ preserve the Frobenius forms. This added flexibility is convenient and innocuous for our purposes since different Frobenius forms on the same decoration yield isomorphic preprojective algebras, as explained in Remark \ref{rem:well-defined}. However, one should be mindful about some non-intuitive consequences of not fixing Frobenius forms, as explained in Example \ref{exam cliff} (4).

\begin{rem} \label{rem:FormOnDeformations}
A Frobenius form $\lambda: D \rightarrow k[[t]]$ on $D$ gives rise to Frobenius forms $\lambda_{0}$ on the $k$-algebra $A \cong D/tD$ and $\lambda_{1}$ on the $k((t))$-algebra $B((t)) \cong D[t^{-1}]$ defined so that the following diagram commutes:
$$
\xymatrix{
&  D/tD \ar[r]^{\lambda_{0}} & k \\
D \ar[r]^{\lambda} \ar[ru]^{\text{mod } t} 
\ar@{^{(}->}[rd] & k[[t]] \ar[ru]_{\text{mod } t} \ar@{^{(}->}[rd] \\
& D[t^{-1}] \ar[r]_{\lambda_{1}} & k((t)).
}
$$
\end{rem}

\begin{rem}
If $B$ is a filtered deformation of a finite-dimensional algebra $A = \oplus_{i=0}^{n} A_{i}$, then a Frobenius form on $A$ induces a Frobenius form on $B$. To see this, observe $\lambda: A \rightarrow k$ non-degenerate implies $\ker(\lambda) \not \subset A_{n} \cong B/B_{\leq (n-1)}$. Hence the lift $\tilde{\lambda}: B \rightarrow k$ to $B$ by defining $\tilde{\lambda}(B_{\leq (n-1)}) = 0$ is non-degenerate. 
\end{rem}

\begin{exam} \label{exam defs}
\begin{itemize}
\item[]
\item[(1)] Assume $k$ contains $n$th roots of unity, write $\zeta$ for a primitive $n$th root of unity, and define the surjective algebra homomorphism
$$
k[x] \rightarrow k^{\oplus n} \hspace{1cm} p(x) \mapsto ( p(\zeta), p(\zeta^{2}), \dots, p( \zeta^{n} ) )
$$
with kernel $(x^{n}-1)$. Using this map one can view $k^{\oplus n} \cong k[x]/(x^{n}-1)$ as a filtered deformation of $k[x]/(x^{n})$. The Frobenius form $\lambda: k[x]/(x^{n}) \rightarrow k$ given by $\lambda(\sum_{j} a_{j} x^{j}) = a_{n-1}$ lifts to $\tilde{\lambda}(\sum_{j} a_{j} x^{j}) = a_{n-1}$. Precomposing with $k^{\oplus n} \cong k[x]/(x^{n}-1)$ gives the Frobenius form 
$$
k^{\oplus n} \rightarrow k \hspace{1cm}
(a_{1}, a_{2}, \dots, a_{n}) \mapsto \frac{1}{n} \sum_{j=1}^{n} \zeta^{j} a_{j}
$$ 

\item[(2)] Fix a field $k$ of characteristic not two, a $k$-vector space $V$, and a quadratic form $Q: V \rightarrow k$. Then define the Clifford algebra by
$$
Cl_{k}(V, Q) := T_{k}(V)/\langle v \otimes_{k} v - Q(v)1 \rangle_{v \in V}.
$$ 
This is a filtered algebra with associated graded algebra given by the exterior algebra, $\bigwedge(V) = Cl_{k}(V, 0)$. The Frobenius form on the exterior algebra given by projecting onto the top wedge power remains non-degenerate when viewed as a form on $Cl(V, Q)$.
\item[(3)] $A \cong Mat_{2 \times 2}(k)$ is separable as an algebra over $k$ so
$$
HH^{2}(A, A) := Ext^{2}_{A \otimes A^{op}-mod}(A, A)=0.
$$
We conclude that $A$ has no deformations as a (symmetric Frobenius) algebra. In particular, $A$ does not deform to $k^{4}$ and more generally in all dimensions $n \geq 4$, $A \oplus k^{n-4}$ does not deform to $k^{n}$. 

\item[(4)] \label{exam cliff}
$\Mat_{2 \times 2}(k)$ can be viewed as a Clifford algebra, and hence has a degeneration to the exterior algebra. However, it does \emph{not} degenerate as a Frobenius algebra using the trace form, and in fact does not degenerate as a symmetric Frobenius algebra.

In more detail, if $k$ is a field of characteristic not two and containing $i := \sqrt{-1}$, then one realizes $\Mat_{2 \times 2}(k)$ as a Clifford algebra using the surjective algebra map
$$
\hspace{1.35cm} 
\varphi: k \langle x, y \rangle \rightarrow \Mat_{2}(k)
\hspace{.4cm}
\text{given by}   
\hspace{.4cm}
x \mapsto  
\left (
\begin{array}{cc}
0 & -1 \\
1 & 0 \\
\end{array}
\right )
\hspace{.5cm}
y \mapsto 
\left (
\begin{array}{cc}
i & 0 \\
0 & -i \\
\end{array}
\right )
$$  
with kernel $(x^{2}+1, y^{2}+1, xy +yx)$. Therefore, 
$$
\Mat_{2 \times 2}(k) \cong k \langle x, y \rangle /(x^2+1, y^2+1, xy+yx) \cong Cl(V, Q)
$$
where $V \cong k^{2}$ with basis $\{x, y\}$ and $Q: V \rightarrow k$ is determined by $Q(x) = Q(y) = Q(x+y)/2 = -1$. Hence $\Mat_{2 \times 2}(k)$ degenerates as an associative algebra to $\bigwedge(V)$.

Consider two Frobenius forms $\lambda_{1}, \lambda_{2}:\Mat_{2 \times 2}(k) \rightarrow k$ given by,
$$
\lambda_{1} \left ( \begin{array}{cc}
a & b \\
c & d \\
\end{array} \right ) = a + d
\hspace{1cm}
\lambda_{2} \left ( \begin{array}{cc}
a & b \\
c & d \\
\end{array} \right ) = b + c.
$$
Observe $\lambda_{1}$ is the usual trace and hence symmetric, but  $\lambda_{1}(x) = \lambda_{1}(xy) = 0$ implies it degenerates to a form on $\bigwedge(V)$ whose kernel contains the left ideal generated by $x$. Therefore, the degeneration as associative algebras from $\Mat_{2}(k)$ to $\bigwedge(V)$ cannot be made into a Frobenius degeneration when $\Mat_{2}(k)$ is given the Frobenius form $\lambda_{1}$. In fact, one can show that there is no associative algebra degeneration taking $\lambda_{1}$ to a non-degenerate form on $\bigwedge(V)$.

However, $\lambda_{2}$ degenerates to $\pi^{2}:\bigwedge(V) \rightarrow \bigwedge^{2}(V) \cong k$, the usual Frobenius form on $\bigwedge(V)$. We conclude that $(\Mat_{2}(k), \lambda_{2})$ Frobenius degenerates to $(\bigwedge(V), \pi^{2})$.  \\
 \end{itemize}
\end{exam}

\noindent In the remainder of this subsection we address the following questions:
\begin{quote}
Consider the directed graph $G_{n, k}$ for each $n \in \bN$ and field $k$ with a vertex for each isomorphism class of $n$-dimensional Frobenius $k$-algebra and an arrow from $A$ to $B$ if $A$ deforms to $B$ as a Frobenius $k$-algebra, with some choice of forms. 
What is the shape of $G_{n, k}$? More specifically, viewed as a partially ordered set, does it have a greatest and least element? If not, what are the minimal and maximal elements? 
\end{quote}    
 
\begin{exam}
The subgraph of commutative Frobenius algebras, denoted $G^{\text{com}}_{4,k} \subset G_{4, k}$, for $k$ algebraically closed appears in \cite{Presotto}. $G^{\text{com}}_{4, k}$ has least element $k[x, y]/(x^{2}-y^{2}, xy)$ and greatest element $k^{\oplus 4}$, and hence, by composing deformations, every four-dimensional commutative Frobenius $k$-algebra deforms to $k^{\oplus 4}$ and degenerates to $k[x, y]/(x^{2}-y^{2}, xy)$. See Example \ref{Graph in dim 4} for the entire $G_{4, k}$, with characteristic of $k$ not two.
\end{exam}

By Remark \ref{rem:controllingdeformations}, if an algebra has vanishing second cyclic cohomology, then it has no deformations as a Frobenius algebra, and hence is maximal. Any Frobenius deformation gives rise to a deformation of associative algebras by forgetting the Frobenius form and hence rigid algebras are maximal. In particular, this includes all semisimple algebras. 

A commutative $n$-dimensional algebra is called \emph{smoothable} if it deforms to $k^{\oplus n}$. The smallest non-smoothable commutative algebra is 8-dimensional, e.g.
$$
k[x, y, z, w]/(x^{2}, xy, y^{2}, z^{2}, zw, w^{2}, xw-yz),
$$
see \cite{Poonen}. The smallest non-smoothable commutative Frobenius algebra is 14-dimensional, see \cite{Notari}. Said differently, $k^{\oplus n}$ is the unique $n$-dimensional, rigid commutative algebra if $n < 8$ and the unique $n$-dimensional rigid commutative Frobenius algebra if $n < 14$. Classifying rigid algebras becomes intractable in higher dimensions. For a nice survey of the techniques used and difficulties encountered, see \cite{Picard}, where 8-dimensional rigid algebras are classified.  

One can generalize this notion to the non-commutative case by asking when an algebra deforms to a semisimple algebra. The smallest non-smoothable algebra is the 3-dimensional path algebra $P(A_{2})$. The smallest non-smoothable Frobenius algebras are 4-dimensional, e.g. $k \langle x, y \rangle /(x^{2},y^{2}, xy- 2yx)$, see Example \ref{Graph in dim 4}. 

We now classify minimal elements in $G_{n, k}$, i.e. most degenerate Frobenius algebras. Additionally, we characterize Frobenius algebras among associative algebras, as filtered deformations of this class of algebras.

One can view a finite-dimensional vector space $V$ with a bilinear form $(-, -): V \times V \rightarrow k$ as a multiplication $\mu$ on the graded vector space $V \oplus k$ with $V$ in degree 1 and $k$ in degree 2. The multiplication is trivially associative since $\mu \circ (\mu \otimes 1) = 0 = \mu \circ (1 \otimes \mu)$ and commutative if $(-,-)$ is symmetric. One can add a unit, a copy of the ground field in degree 0, to obtain a $\dim(V)+2$ dimensional algebra, 
$$
A(V, (-, -)) := k \oplus V \oplus k
$$ 
as a graded vector space with multiplication,
$$
(a, b, c) \cdot (a', b', c') := (a a', ab' + b a', (b, b') + a c' + c a').
$$
Moreover, if $(-,-)$ is non-degenerate, then projecting onto the copy of $k$ in the second graded piece is a Frobenius form. These algebras are the only candidates for the most degenerate Frobenius algebras, as made precise in the following proposition. 

\begin{prop} \label{prop:MostDegenerate}
Let $A$ be a Frobenius algebra of dimension $n>1$ over any field $k$. Then there exists a Frobenius form $\lambda$ on $A$ such that $A$ degenerates as a Frobenius algebra to $A(V, (-,-))$ for some non-degenerate, bilinear form $(-,-): V \otimes  V \rightarrow k$.
\end{prop}


How general is this class of algebras? If $V$ is finite-dimensional then choose a basis $\{ e_{i} \}$ for $V$ and write the bilinear form $(-, -)$ as a matrix $M$ with $(i,j)$th entry $m_{ij} := ( e_{i}, e_{j})$. Over a field of characteristic different from 2, if $(-,-)$ is symmetric then one can change bases so that $M$ is diagonal. Moreover, if all square roots of the diagonal entries are contained in the field, then one can rescale any element $a$ to $a/ \sqrt{(a, a)}$ so that the matrix is given by the identity. Hence $A(V, (-,-)) \cong A(k^{n}, (e_{i},e_{j})= \delta_{i,j} )$ as Frobenius algebras over an algebraically closed field of characteristic not two. 
To condense notation we write $Z_{n} := A(k^{n}, (e_{i},e_{j})= \delta_{i,j} )$ and note that, 
$$
Z_{n} := k[1] \oplus k[x_{1}, \dots, x_{n-2}] \oplus k[w] \hspace{.8cm} x_{i} \cdot x_{j} = \delta_{ij} w, \ \ x_{i} \cdot w = 0 \hspace{.8cm}
\lambda: Z_{n} \rightarrow k[w].
$$
for $n \geq 2$ and we define $Z_{1} := k$.
Therefore, in the commutative case, over an algebraically closed field of characteristic not two, there is a single most degenerate Frobenius algebra.

\begin{cor}  \label{cor:MostDegenerateCommutative}
Every $n$-dimensional commutative Frobenius algebra over an algebraically closed field $k$ with characteristic not two Frobenius degenerates, with some choice of form, to $Z_{n}$.
\end{cor}

The idea of the proof of the proposition is to write down an explicit 3-term filtration on the Frobenius algebra $A$ with form $\lambda$ by,
$$
k1 \subset \ker(\lambda) \subset A
$$
whose associated graded algebra is of the form $A(V, \lambda \circ \mu)$. Therefore, we first establish a technical lemma that says one can always find a Frobenius form such that $k1 \subset \ker(\lambda)$. 

\begin{lem} \label{lem:FormVanishesOnUnit}
Let $A$ be a Frobenius algebra with form $\lambda$ over $k \neq \bF_{2}$. Then either $A = k$ or there exists a unit $u \in A^{\times}$ such that $\lambda(u) =0$. Consequently, there exists a Frobenius form $\lambda'$ on $A$ such that $\lambda'(1) = 0$.
\end{lem}

\begin{proof}
If $\lambda$ vanishes on a unit $u \in A^{\times}$ then define $\lambda' := \lambda \circ L_{u}$ where $L_{u}$ is left multiplication by $u \in A$, so $\lambda'(1) = \lambda(u) =0$ as desired. \\
\\
Suppose $\lambda$ is non-zero on every unit. 
Notice that each element $x \in A$ with $x^{n}=0$ gives rise to a unit $1+x \in A^{\times}$ as $$
(1+x) \left ( \sum_{j=0}^{n-1} (-x)^j \right ) = (1 \pm x^{n}) = 1.
$$
Then $\lambda(1 \cdot \lambda(x) - \lambda(1) \cdot x ) = 0$ implies $1 \cdot \lambda(x) - \lambda(1) \cdot x$ is not a unit and hence $\lambda(x) =0$. We conclude that $\lambda$ vanishes on all nilpotent elements and hence $\lambda$ vanishes on the Jacobson radical, $J(A)$. Non-degeneracy of $\lambda$ implies $J(A) =0$. The vanishing of the Jacobson radical and the fact that $A$ is Artinian together imply that $A$ is semisimple. By the Artin-Wedderburn Theorem, $A \cong \prod_{i=1}^{m} \Mat_{n_{i}}(D_{i})$ is a product of matrix algebras over division algebras over $k$.

Therefore, $A$ has a subalgebra $k^{\sum_{i=1}^{m} n_{i}}$ of diagonal matrices with coefficients in $k$. Splitting the inclusion of this subalgebra and summing the components gives a Frobenius form 
$$
\lambda': A \cong  \prod_{i=1}^{m}  \Mat_{n_{i}}(D_{i}) \longrightarrow  k^{\sum n_{i}} \longrightarrow k,
$$
which by construction agrees with the usual trace for matrices in $\prod_{i} \Mat_{n_{i}}(k)$. In particular, $\lambda'$ vanishes on the units given by products of diagonal matrices with non-zero entries summing to zero. Such entries can be found since $k \neq \bF_{2}$. 

As any two Frobenius forms differ by multiplication by a unit, see Remark \ref{rem:well-defined}, we have $\lambda' = \lambda \circ L_{u}$ for some $u \in A^{\times}$. In particular, $\lambda$ vanishes on a unit if and only if $\lambda'$ does. So we conclude that $m=1$ and $n_{1} = 1$. 

Now $A \cong D_{1}$ and $\lambda: A \rightarrow k$ doesn't vanish on a unit only if $\lambda$ is injective. Hence $A = k$, completing the proof. 
\end{proof}

\begin{rem} \label{rem:F2case}
If $k=\bF_{2}$, the above proof holds for $m$ even, and for $m$ odd can be modified to show $A = \bF_{2}^{m}$, a product of copies of $\bF_{2}$ with pointwise multiplication.
Non-degeneracy of a form $\lambda: A \rightarrow \bF_{2}$ implies $\lambda(1, 0, \dots, 0) = \lambda(0,1,0, \dots, 0) = \cdots = \lambda(0,\dots, 0,1) = 1$, in which case $\lambda(1, 1, \dots, 1) = m \equiv 1$ (mod 2). Hence $A$ has a unique Frobenius form, which takes the unit to one. Therefore, the assumption that $k \neq \bF_{2}$ in Lemma \ref{lem:FormVanishesOnUnit} is necessary. However, if $m \geq 3$ is odd, then $\bF_{2}^{m} \cong \bF_{2}^{2} \oplus \bF_{2}^{m-2}$ Frobenius degenerates to $\bF_{2}[x]/(x^{2}) \oplus \bF_{2}^{m-2}$, which has Frobenius form $\lambda( b_{1} + b_{2}x, a_{1}, \dots, a_{m-2}) = b_{1} + b_{2} + a_{1} + \cdots + a_{m-2}$ sending the unit to zero. Therefore, every Frobenius algebra Frobenius degenerates to a one whose form vanishes on a unit.  
\end{rem}

\begin{proof}(of Prop \ref{prop:MostDegenerate})
If $k \neq \bF_{2}$, then by Lemma \ref{lem:FormVanishesOnUnit} $A$ has a Frobenius form which vanishes on the unit. If $k = \bF_{2}$ then $A$ either has such a form, or can be Frobenius degenerated to an algebra $B$ with form $\lambda$ satisfying $\lambda(1) = 0$, by Remark \ref{rem:F2case}. Then the filtration,
$$
k1 \subset \ker(\lambda) \subset B
$$
has associated graded algebra
$$
gr(B) = k1 \oplus \ker(\lambda)/k1 \oplus B/\ker(\lambda).
$$
Define $V := \ker(\lambda)/k1$ and observe $B/\ker(\lambda) \cong \text{im}(\lambda) \cong k$, so $gr(B) = k \oplus V \oplus k$ as a graded vector space. $\lambda$ is still a well-defined Frobenius form on $gr(B)$. And multiplication is defined on $gr(B)$ so that $(1, 0, 0)$ is the unit and $(0, v, 0) \cdot (0, w, 0) = \lambda(\mu(v, w))$. Hence $A$ degenerates as a Frobenius algebra to $A(V, \lambda \circ \mu)$.
\end{proof}

The upshot is that every $n$-dimensional Frobenius algebra has a filtered Frobenius degeneration to some $A(V, (-,-))$. In the commutative case, over an algebraically closed field of characteristic not two, $A(V, (-,-)) \cong Z_{n}$. Conversely, Frobenius algebras are characterized among associative algebras as being filtered deformations of some $A(V, (-,-))$, which necessarily preserves the Frobenius form by Remark \ref{rem:FormOnDeformations}.

\begin{rem} \label{rem:bilinearformdeformations}
Frobenius degenerations from $A(V, (-,-))$ to $A(V, (-,-)')$ by definition preserve the unit and top-graded piece and hence necessarily arise as degenerations of the underlying bilinear form. In terms of the corresponding matrices $M$ and $M'$ of the forms, one must have $M'$ in the closure of the congruence class of $M$, i.e.
$$
M' = \lim_{t \rightarrow 0} N_{t} M N_{t}^{\intercal},
$$
for some $N_{t} \in Mat_{|V|}(k((t)))$. For instance,
$$
\left (
\begin{array}{cc}
0 & 1 \\
-1 & 0 \\
\end{array}
\right )
=
\lim_{t \rightarrow 0}
\left (
\begin{array}{cc}
t^{2} & 1 \\
-1 & 0 \\
\end{array}
\right ) 
=
\lim_{t \rightarrow 0}
\left (
\begin{array}{cc}
t & 0 \\
0 & 1/t \\
\end{array}
\right )
\left (
\begin{array}{cc}
1 & 1 \\
-1 & 0 \\
\end{array}
\right )
\left (
\begin{array}{cc}
t & 0 \\
0 & 1/t \\
\end{array}
\right )
$$
realizes the skew-symmetric bilinear form as a degeneration of the form on $k^{2}_{x,y}$ given by $(x,x) = (x, y) = -(y, x) = 1+ (y,y) = 1$. This is the only non-trivial deformation among non-degenerate bilinear forms on $k^{2}$, see \cite{Sergeichuk}. The question of determining the closures of congruence classes of matrices is studied in \cite{Dmytryshyn}, by first establishing a canonical form for every congruence class of matrices. 
\end{rem}

\begin{exam} \label{Graph in dim 4} 
Here is $G_{4, k}$ for characteristic of $k$ not 2, where we have omitted arrows that can be obtained as a composition of the given arrows.  The commutative component appeared in \cite{Presotto}, and is classical.
$$
\begin{xy}
(0, 60)*+{\mathlarger{ \frac{k[x,y]}{(x^{2}-y^{2}, xy)}}} = "A1" ; 
(25, 60)*+{\mathlarger{\frac{k[x]}{(x^{4})}}} = "A2" ;
(50, 70)*+{\mathlarger{\frac{k[x]}{(x^{3})}} \oplus k} = "A3" ;
(50, 50)*+{\mathlarger{\frac{k[x]}{(x^{2})}} \oplus \mathlarger{\frac{k[y]}{(y^{2})}} } = "A4" ;
(85, 60)*+{\mathlarger{\frac{k[x]}{(x^{2})}} \oplus k \oplus k} = "A5" ;
(110, 60)*+{k^{\oplus 4}} = "A6" ;
(-5, 30)*+{\bigwedge(k^{2})} = "B1" ;
(100, 30)*+{\Mat_{2}(k)} = "B3" ;
(51, 30)*+{ 
 \left \{ \mathsmaller{ \mathsmaller{
\begin{pmatrix}
a & 0 & 0 & 0 \\
0 & a & 0 & d \\
c & 0 & b & 0 \\
0 & 0 & 0 & b \\
\end{pmatrix} } }
: a, b, c, d \in k \right \}  
 \cong \Pi(A_{2})} = "B2" ;
(5, -5)*+{\mathlarger{\frac{k \langle x,y \rangle}{(x^{2}, y^{2}, xy- \lambda yx)},}} = "D1"; 
(7, 10)*+{\mathlarger{\frac{k \langle x,y \rangle}{(y^{2}, x^{2}+yx, xy+yx)}}} = "C1";
(37, -5)*+{ \lambda \in k \backslash \{ -1, 0, 1 \} } = "D2";
{\ar@{->} "A1"; "A2"}  ;
{\ar@{->} "A2"; "A3"}  ;
{\ar@{->} "A2"; "A4"}  ;
{\ar@{->} "A3"; "A5"}  ;
{\ar@{->} "A4"; "A5"}  ;
{\ar@{->} "A5"; "A6"}  ;
{\ar@{->} "B1"; "B2"}  ;
{\ar@{->} "B2"; "B3"}  ;
{\ar@{->} "B1"; "C1"}  ;
\end{xy}
$$ 
Notice that the algebras on the far left are of the form $A(k^{2}, (-,-))$:
\begin{align*}
&A(k^{2}, \left ( \begin{array}{cc} 1 & 0 \\ 0 & 1 \end{array} \right ) ) \cong k[x, y]/(x^{2}-y^{2}, xy) \\
&A(k^{2}, \left ( \begin{array}{cc} 0 & 1 \\ -1 & 0 \end{array} \right ) ) \cong \bigwedge(k^{2}) \\
&A(k^{2}, \left ( \begin{array}{cc} 0 & 1 \\ \lambda & 0 \end{array} \right ) ) \cong k \langle x, y \rangle /(x^{2}, y^{2}, xy- \lambda yx) \\
&A(k^{2}, \left ( \begin{array}{cc} 1 & 1 \\ -1 & 0 \end{array} \right ) ) \cong k \langle x, y \rangle /(y^{2}, x^{2}+yx, xy+yx). 
\end{align*}
The deformation of $\Pi(A_{2}) = k[e_{1}, e_{2}] \oplus k[\alpha, \alpha^{*}]$ to the matrix algebra is given by deforming the preprojective relation $[\alpha, \alpha^{*}] = t(e_{1} - e_{2})$. The degeneration to $\bigwedge(k^{2})$ is the filtered Frobenius degeneration produced in the proof of Proposition \ref{prop:MostDegenerate}, with $\lambda$ given by summing the non-diagonal entries, which vanishes on the identity. The deformation from $\bigwedge(k^{2})$ to $k \langle x, y \rangle / (y^{2}, x^{2} +yx, xy+yx)$ comes from a deformation of the underlying bilinear forms, as explained in Remark \ref{rem:bilinearformdeformations}. 

\end{exam}

\subsection{Degenerations of Preprojective Algebras}
We define a notion of degeneration of decorations, which gives rise to degenerations at the level of decorated path and preprojective algebras. We use this notion to view decorated preprojective algebras as degenerations of ordinary preprojective algebras. \\
\\
We want to deform pairs $(A, M)$ where $A$ is a $k$-algebra and $M$ is an $A$-bimodule.  

\begin{defn} \label{module deg}
Let $A$, $B$ be associative algebras and $M$ an $A$-bimodule and $N$ a $B$-bimodule. One says $(A, M)$ \emph{deforms} to $(B, N)$ if there exists a pair $(D, P)$ where $D$ is a $k[[t]]$-module  deforming $A$ to $B$ and $P$ is a $D$-bimodule such that,
\begin{itemize}
\item $P/tP \cong M$ as $A$-bimodules
\item $P[t^{-1}] \cong N(( t))$ as $B((t))$-bimodules
\end{itemize}
Moreover, one says the deformation is \emph{flat} if $(D, P) \cong (A[[t]], M[[t]])$ as $k[[t]]$-modules. 
\end{defn}

One can modify this definition to work for left or right modules as well. Additionally, if $D$ is a symmetric Frobenius deformation then we say the pair $(D, P)$ is a symmetric Frobenius deformation.

\begin{rem}
In the presence of condition (F), one only needs to degenerate $A$ as an $(A, k)$-bimodule, $(k, A)$-bimodule, or $(A, A)$-bimodule. All such degenerations come from degenerations of $A$ as an algebra over $k$. Moreover, flatness of the module deformation is precisely flatness of the algebra deformation.
\end{rem}

Hence, we arrive at the notion of degeneration of a decoration $(Q, D_{Q})$ by degenerating each $k$-algebra and bimodule.

\begin{defn}
Let $D_{Q} = (A_{i}, M^{\alpha})$ and $D'_{Q} = (B_{i}, N^{\alpha})$ be decorations of $Q$. One says $D_{Q}$ \emph{deforms} to $D'_{Q}$ if the pair $( \oplus_{i \in Q_{0}} A_{i}, \oplus_{\alpha \in Q_{1}} M^{\alpha})$ deforms to $(\oplus_{i \in Q_{0}} B_{i}, \oplus_{\alpha \in Q_{1}} N^{\alpha})$.
\end{defn}

Notice that degenerations of decorated quivers need not preserve condition (F), even for the quiver with a single vertex and no arrows, see Example \ref{exam cliff}. Consequently, we exclusively consider \emph{symmetric} Frobenius degenerations of decorations, which do preserve condition (F). 

Conceptually, one can view the pair $(\oplus_{i \in Q_{0}} D_{i}, \oplus_{\alpha \in Q_{1}} P^{\alpha})$ as a decoration for Q by symmetric Frobenius $k[[t]]$-algebras. Now the path algebra, $P_{k[[t]]}(Q, (D_{i}, P^{\alpha}))$,
is a deformation from $P_{k}(Q, (A_{i}, M^{\alpha}))$ to $P_{k}(Q, (B_{i}, N^{\alpha}))$. And, using the Frobenius form on the deformation, the preprojective algebra
$\Pi_{k[[t]]}(Q, (D_{i}, P^{\alpha}))$ is a deformation from $\Pi_{k}(Q, (A_{i}, M^{\alpha}))$ to $\Pi_{k}(Q, (B_{i}, N^{\alpha}))$.

Having defined degenerations of decorations we now define foldings of decorations. If $Q$ is a quiver with graph automorphisms $\Aut(Q)$, then one can form the quotient quiver $Q/\Aut(Q)$ and decorate it with decoration $D_{Q}/\Aut(Q) = (B_{j}, N^{\beta})$ defined by:
$$
B_{j} := \bigoplus_{i \in \Aut(Q)\cdot j} A_{i} 
\hspace{1cm} \text{and} \hspace{1cm} 
N^{\beta} := \bigoplus_{\alpha \in \Aut(Q) \cdot \beta} M^{\alpha}.
$$
In this case, we say $(Q, D_{Q})$ \emph{folds} to $(Q/\Aut(Q), D_{Q}/\Aut(Q))$.

Consider $Q$ a quiver with the decoration 
$$
C_{Q} := \{ A_{i} = k, M^{\alpha} = k \},
$$
known as the \emph{constant} $k$-decoration. With this decoration, one recovers the classical notions,
$$
\Rep(Q, C_{Q}) \cong \Rep(Q) 
\hspace{1cm}
P(Q, C_{Q}) = P(Q)
\hspace{1cm}
\Pi(Q, C_{Q}) = \Pi(Q).
$$
Next suppose $Q$ has non-trivial graph automorphisms and fold $(Q, C_{Q})$ to \\
$(Q/\Aut(Q), C_{Q}/\Aut(Q))$, so the decoration consists of sums of copies of $k$ at each vertex. Notice,
\begin{align*}
\Rep(Q) \cong \Rep(Q, C_{Q}) &\cong \Rep(Q/\Aut(Q), C_{Q}/\Aut(Q)) \\
P(Q) = P(Q, C_{Q}) &\cong P(Q/\Aut(Q), C_{Q}/ \Aut(Q)) \\
\Pi(Q) = \Pi(Q, C_{Q}) &\cong \Pi(Q/\Aut(Q), C_{Q}/\Aut(Q))
\end{align*}
so we've built a decorated quiver with (1) the same preprojective algebra as an ordinary quiver and (2) a non-trivial degeneration to a decorated quiver for each symmetric Frobenius degeneration at each folded vertex. 

It is natural to ask which decorated quivers, and hence decorated preprojective algebras, arise from ordinary quivers, and hence ordinary preprojective algebras, in this way. In the case of a quiver with no arrows, this is the class of Frobenius algebras with a Frobenius deformation  to $k^{\oplus n}$, denoted $\cF_{n}$. More generally, any decorated quiver $(Q, D_{Q})$ satisfying condition (F) with $\oplus_{i} A_{i} \in \cF_{\sum_{i} \dim_{k}(A_{i})}$ deforms to an ordinary quiver. Consequently, a large class of decorated preprojective algebras deform to ordinary preprojective algebras, including many in \cite{GLS}. 

\begin{exam}
For any dimension vector $d=(d_{i}) \in \bN^{n}$, and choice of Frobenius algebra $F_{i} \in \cF_{d_{i}}$ for $i \in \{ 1, \dots, n \}$ can obtain a decorated $A_{n}$ quiver with the $k$-algebras given by $F_{i}$, and the arrows decorated by $F_{i} \otimes F_{i+1}$ as a degeneration of the $k$-constant decoration of the quiver,
$$ 
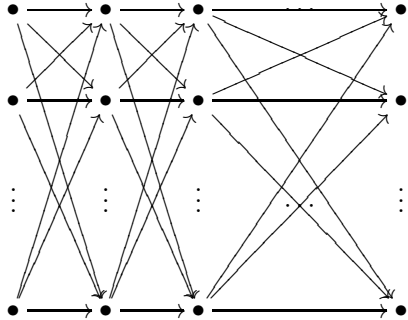
\vcenter{ 
\xymatrix{
\bullet \ar[r] \ar[rd] \ar[rddd] & \bullet 
\ar[r] \ar[rd] \ar[rddd] & \bullet \ar[rr] \ar[rrd] \ar[rrddd] & \cdots & \bullet \\
\bullet \ar[ru] \ar[r] \ar[rdd] & \bullet \ar[ru] \ar[r] \ar[rdd] & \bullet \ar[rru] \ar[rr] \ar[rrdd] & \cdots & \bullet\\
\vdots & \vdots  & \vdots & \cdots & \vdots \\
\bullet \ar[r] \ar[ruu] \ar[ruuu] & \bullet \ar[r] \ar[ruu] \ar[ruuu] & \bullet \ar[rr] \ar[rruu] \ar[rruuu] & \cdots & \bullet
}
} 
$$
with $d_{i}$ vertices in the $i$th column. 
\end{exam}

\subsection{Flatness Conjecture}
In this subsection, we conjecture that all flat Frobenius degenerations of decorations induce flat degenerations of the preprojective algebras. 

Recall that a deformation of associative algebras $D$ deforming $A$ to $B$ is flat if $D \cong A[[t]]$ as $k[[t]]$-modules. 
If $D$ is finitely-generated over $k[[t]]$ then flatness is equivalent to $\dim_{k}(A) = \dim_{k}(B)$. In the absence of flatness, the vector space underlying $B$ is no larger than that of $A$. 

\begin{prop} \label{inequality}
Let $A$ and $B$ be finite-dimensional $k$-algebras. If $A$ deforms to $B$ as $k$-algebras, then $\dim_{k}(A) \geq \dim_{k}(B)$. 
\end{prop}

In all of the cases where we use this inequality one has a finite filtration on $B$, $B_{\leq 0} \subset B_{\leq 1} \subset \cdots \subset B_{\leq m} = B$ such that $A$ is the associated graded algebra and hence there is a surjection, $\varphi_{n}: A_{n} \rightarrow B_{\leq n}/B_{\leq n-1}$ for each $n$ and hence,
\begin{align*}
\dim(A) &= \sum_{n=0}^{m} \dim(A_{n}) \geq \sum_{n=0}^{m} \dim(B_{\leq n}/B_{\leq n-1}) 
= \sum_{n=0}^{m}  \dim(B_{\leq n}) - \dim(B_{\leq n-1} ) \\
&= \dim(B_{\leq m}) - \dim(B_{\leq -1}) = \dim(B).
\end{align*}
In general, the inequality can be strict, e.g. $k[[t]][x]/(xt, x^{2})$ deforms $k[x]/(x^{2})$ to $k$.

Additionally, in our set-up $A$ and $B$ both have $\bN$-gradings, distinct from the grading already on $A$ as a filtered degeneration, by path length. The deformation preserves the length of each element and so we will have the more refined inequality on the level of Hilbert--Poincar\'{e} series
$$
h_{A}(t) := \sum_{j \in \bN} \dim(A_{j})t^{n} 
$$
given by $h_{A}(t) \geq h_{B}(t)$ in the sense that the difference $h_{A}(t)-h_{B}(t)$ has non-negative coefficients. 

In the interest of computing the Hilbert--Poincar\'{e} series of decoration preprojective algebras, we would like to know when the degeneration from the ordinary preprojective algebra, defined in the previous section, is flat. Flatness of the length zero and one paths says the algebras and the bimodules in the decoration have the same dimension. We conjecture that this is all that is needed to ensure flatness in any length. 

\begin{conj} \label{conj:flatness}
\conjflatness
\end{conj}

For ordinary preprojective algebras of Dynkin quivers, one can use the Euler-Poincar\'{e} principle applied to the Schofield resolution to get a formula for the Hilbert--Poincar\'{e} series:
$$
h_{\Pi(Q)}(t) = \frac{1-Pt^{\text{max}}}{1- At + t^{2}} 
$$
where $A$ is the adjacency matrix for $Q$, and $P$ is the permutation matrix of the vertices induced from the Nakayama automorphism, see \cite{MOV}. In the non-Dynkin case, one can again use the Euler-Poincar\'{e} principle, now applied to the Koszul resolution, to get a formula for the Hilbert--Poincar\'{e} series:
$$
h_{\Pi(Q)}(t) = \frac{1}{1- At + t^{2}}. 
$$

If we assume the Schofield and Koszul resolutions degenerate to \emph{resolutions} in the decorated setting, then the same analysis produces the following formula. Let $A = (a_{i,j})$ be a $|Q_{0}| \times |Q_{0}|$ matrix with coefficients in $\bZ$ defined by 
$$
a_{i,j}:= \sum_{\alpha: i \rightarrow j \in Q_{1}} \frac{\dim_{k}(M^{\alpha})}{\dim_{k}(A_{i})}.
$$
Let $D = (d_{i,j})$ be a diagonal matrix defined by $d_{i,j} = \delta_{ij} \dim_{k}(A_{i})$. We conjecture the following formula holds:
\begin{conj} \label{conj:strongflatness}
\conjstrongflatness
\end{conj}

By Proposition \ref{inequality}, it suffices to prove Conjecture \ref{conj:flatness}, at the extremes. That is, for a fixed quiver $Q$, it suffices to prove the conjecture for $D_{Q} = (A_{i}, M^{\alpha})$ with each $A_{i}$ a most degenerate Frobenius algebra and with each $A_{i}$ a most deformed Frobenius algebra. By Proposition \ref{prop:MostDegenerate} and Remark \ref{rem:bilinearformdeformations}, the most degenerate Frobenius algebras are given by a vector space with a non-degnerate billinear form without degenerations to an inequivalent non-degenerate bilinear form. 

In the commutative case, if $k$ is algebraically closed and characteristic zero, with $n < 14$, $k^{\oplus n}$ is the unique most deformed Frobenius algebra and $Z_{n}$ is the unique most degenerate Frobenius algebra, by \cite{Notari} and Corollary \ref{cor:MostDegenerateCommutative}. Therefore, to prove the conjecture for commutative decorations of total dimension less than 14,
one only needs to establish a single equality of Hilbert--Poincar\'{e} series. Additionally, one Hilbert--Poincar\'{e} series is equal to an ordinary preprojective algebra and hence is known and the other cannot have any smaller coefficients. Therefore, the conjecture reduces to finding a spanning set (of a specified size) for a single algebra. 
 
In particular, in the Dynkin setting, the automorphism groups of the underlying diagrams are: (1) trivial in types $B, C, F, G$, (2) $\bZ/2 \bZ$ in types $A, D, E$ except $D_{4}$ and (3) the symmetric group on the three valence one vertices in type $D_{4}$. Therefore, by folding the constant decoration at these vertices, one gets the Frobenius algebras $k, k^{\oplus 2}, k^{\oplus 3}$.  Consequently, in the Dynkin setting, we only need to consider the maximal Frobenius degeneration of decorations $k, k^{\oplus 2}, k^{\oplus 3}$ to $k$,  $k[x]/(x^{2})$, and $k[x]/(x^{3})$ respectively. In Section 5, we exploit these observations to prove the conjecture in the Dynkin case. 

\begin{rem}
By generalizing a result of Etingof and Eu, in \cite{Etingof}, we believe that the conjecture can be reduced to the Dynkin and extended Dynkin cases. 
In more detail, for any connected non-Dynkin quiver $Q$ there exists an extended Dynkin subquiver $Q_{E}$. Filter $\Pi(Q)$ by giving arrows in $\overline{Q_{E}}$ degree one, while all others are degree zero. For any subquiver of $Q' \subset Q$, denote by $\widehat{Q'}$ the quiver with arrow set $Q'_{1}$ but vertex set $Q_{0}$. Then,
$$
gr(\Pi(Q)) = \Pi(\widehat{Q_{E}}) * \Pi(\widehat{Q \backslash Q_{E}}, (Q_{E})_{0})
$$ 
where $\Pi(\widehat{Q \backslash Q_{E}}, (Q_{E})_{0})$ is the algebra obtained from $\Pi(\widehat{Q \backslash Q_{E}})$ by deleting the relation at each vertex in $Q_{E}$. This determines the Hilbert--Poincar\'{e} series of $\Pi(Q)$ in terms of the known polynomial for $Q_{E}$ and a partial preprojective algebra. Future work will attempt to generalize this argument to the decorated setting.   
\end{rem}

Evidence for Conjecture \ref{conj:flatness} consists of the following:
\begin{itemize}
\item[(1)] a proof of the conjecture in the case $Q$ is Dynkin, see Theorem \ref{thm:Dynkin},
\item[(2)] a proof of the conjecture in the case $Q = \tilde{D}_{4}$ with commutative algebras at the vertices see \cite{Presotto}, and computer evidence suggesting that commutativity can be removed,
\item[(3)] and computer evidence in the case 
$$
Q = (\{1, \cdots, n \}, \{\alpha_{1}, \cdots, \alpha_{n-1} \}, s(\alpha_{i}) = i, t(\alpha_{i}) = n )
$$ 
for $n \leq 50$, in low path-graded pieces.  
\end{itemize}

Here we present computations in Magma providing evidence for the conjecture. \\
\\
Let $Q_{n+1}$ be a quiver with $n+1$ vertices, $n$ sources and one sink. $Q_{n+1}$ can be folded to $A_{2}$ with decoration $k^{\oplus n}$ and $k$ at the vertices. Then one can maximally degenerate the decoration to $Z_{n}$ and $k$. Denote the decorated preprojective algebra by $\Pi(Z_{n})$ and its $j$-path graded piece by $\Pi^{j}(Z_{n})$.\\
 \\
Using Magma we've shown,
\begin{prop} \label{prop:computation-star}
$$
\dim \Pi^{j}(Q_{n+1}) = \dim \Pi^{j}(Z_{n})
$$
for $(j, n) \in \{1, \dots, 10 \} \times \{ 1, \dots, 10 \}
\cup \{ 1, \dots, 5 \} \times \{ 11, \dots 50 \}.$ 
\end{prop}

\begin{cor}
Every degeneration of $\Pi(Q_{n})$ by degenerating the decoration is flat in the $j$th graded piece, for $j \leq 5$ and $n \leq 50$.
\end{cor}


Additionally, we computed the dimension of the decorated preprojective algebra of the $A_{2}$ quiver, with labels $k$ and $\Mat_{n}(k)$, denoted $\Pi(\Mat_{n}(k))$.
\begin{prop} \label{prop:computation-matrix}
$$
\dim \Pi^{j}(\Mat_{n}(k)) = \dim \Pi^{j}(Z_{n^2})
$$
for $(j, n) \in \{1, \dots, 10 \} \times \{ 1, 2, 3 \}
\cup \{ 1, \dots, 5 \} \times \{ 4, 5, 6, 7 \}$.
\end{prop}

Combining the propositions, one has
$$
\dim \Pi^{j}(\Mat_{n}(k)) = \dim \Pi(Q_{n^{2} +1})
$$
despite there being no deformation of $k \rightarrow \Mat_{n}(k)$ to $k \rightarrow k^{\oplus n^{2}}$.

This suggests that Theorem 3.8 in \cite{Presotto}, giving the dimensions $\dim \Pi^{j}(F)$ for all $j \in \bN$ and for any $F$, a four-dimensional \emph{commutative} Frobenius algebra, may hold for non-commutative Frobenius algebras as well.

In the next section we prove the conjecture in the case of degenerations of Dynkin quivers.

\begin{thm} \label{thm:Dynkin}
\thmDynkin
\end{thm}

\begin{rem} \label{rem: clarification of the theorem}
Theorem \ref{thm:Dynkin} only addresses Frobenius degenerations of preprojective algebras coming from degenerations of decorated quivers. This does not include, for instance, the degeneration of $\Pi(A_{2})$ to $\wedge(k^{2})$ from Example \ref{Graph in dim 4}, which in degree zero degenerates $ke_{1} \oplus ke_{2}$ to $k$.  
\end{rem}

\begin{proof}
For each non-simply laced Dynkin diagram $X$ obtained by folding $Q$ denote by $\Pi(X)$ the quiver obtained in the statement above with the most degenerate decorations. To show flatness it suffices to compute the dimension of $\Pi(X)$, apriori no smaller than the dimension of $\Pi(Q)$ by Proposition \ref{inequality}. Hence we need to show:
\begin{itemize}
\item $\dim(\Pi(G_{2})) \leq \dim(\Pi(D_{4})) = 28$
\item $\dim(\Pi(F_{4})) \leq \dim(\Pi(E_{6})) = 156$
\item $\dim(\Pi(B_{n})) \leq \dim(\Pi(A_{2n-1})) = \frac{1}{3} n(2n-1)(2n+1)$
\item $\dim(\Pi(C_{n})) \leq \dim(\Pi(D_{n+1})) = \frac{1}{3} n(n+1)(2n+1)$
\end{itemize}  
This is done in each subsection of Section 5. \\
\\
In each case, the Frobenius structure on $\Pi(X)$ can be viewed as a composition:
$$
\lambda: \Pi(X) \overset{\pi}{\twoheadrightarrow} \Pi^{\text{max}_{t}}(X) \cong \Pi^{0}(X) \overset{\oplus \lambda_{i}}{\longrightarrow} k,
$$ 
where $\pi$ is the projection to the maximum length paths and $\oplus_{i} \lambda_{i}$ is the Frobenius form on the sum of the decorations at the vertices. Since $\oplus \lambda_{i}$ is non-degenerate by assumption, it suffices to show $\pi$ is non-degenerate. In other words, we need to show any path can be extended to a maximum length path, i.e. for all $q \in \Pi(X)$ there exists $p \in \Pi(X)$ such that $pq \in \Pi^{\text{max}_{t}}(X)$. This is apparent in the explicit bases we produce, where each basis element is visibly a subpath of a maximum length basis element. From this perspective the deformation to an ordinary preprojective algebra is a Frobenius deformation, coming from the Frobenius deformations of the decorations at the vertices.  
\end{proof}

\begin{rem} \label{rem:arbitraryfield}
These arguments hold over any field, as the classification of 3-dimensional Frobenius algebras over $k$ is independent of $k$, and the calculations do not involve dividing by integers. 
\end{rem}

\section{The Dynkin Case}
This section is dedicated to a proof of Theorem \ref{thm:Dynkin}, which proves Conjecture \ref{conj:flatness} in the case of $Q$ Dynkin. In the previous section, we reduced Theorem \ref{thm:Dynkin} to computing spanning sets for the decorated quivers,
$$
\Pi(B_{n}), \Pi(C_{m}), \Pi(F_{4}), \Pi(G_{2})
$$
for $n \geq 2$ and $m \geq 4$, defined by folding and maximally degenerating $(Q, C_{Q})$ with $Q$ respectively $A_{2n-1}$, $D_{m+1}$, $E_{6}$, and $D_{4}$.
In diagrams, 
$$
 \xymatrix@=1em{
& k  \ar[r] & \cdots \ar[r] & k & & &   \\
A_{2n-1}= k  \ar[]!<0ex,-4ex>;[ur]!<-1ex, -1.3ex>  \ar[]!<0ex,4ex>;[dr]!<0ex, 1ex> & & &   \fold & 
k \ar[r]  &  k \oplus k \ar[r] & \cdots \ar[r] & k \oplus k \degen  B_{n} = k \ar[r]  & S  \ar[r]& \cdots \ar[r] & S  \\
& k  \ar[r]  & \cdots  \ar[r] & k & & & \\
}
$$
$$
\xymatrix@=1em{
& & &  k  & & &  \\
D_{m+1} = k \ar[r] & \cdots \ar[r] & k \ar[ru] \ar[rd] & \fold & k \ar[r]  & \cdots \ar[r] & k \ar[r] & k \oplus k  \degen  C_{m} = k \ar[r]  & \cdots \ar[r] & k \ar[r] & S &  \\
& & &  k  & & & \\
}
$$
$$
\xymatrix@=1em{
& &  &  k & & \\
& &  k  \ar[ru] & & &  \\
E_6 = k \ar[r] & k \ar[ru] \ar[rd] & & \fold & k \ar[r]  & k \ar[r] & k \oplus k \ar[r] & k \oplus k  \degen  F_4 =k \ar[r] & k \ar[r] & S \ar[r] & S   \\
& &  k  \ar[rd] & & & \\
& &  &  k  & & \\
}
$$
$$
\xymatrix@=1em{
 && k & & & & & &  \\
D_{4} = & k \ar[r] \ar[ru] \ar[rd] & k & \fold  & k \ar[r]  &  k \oplus k \oplus   k & \degen  & G_{2} = k \ar[r]  &  S'  \\
& & k & & & & & & \\
}
$$
where $S := k[x]/(x^{2})$ and $S' := k[x]/(x^{3})$ are the most degenerate Frobenius algebras of dimension 2 and 3, respectively. Now,
$$
\begin{array}{ll}
\Pi(B_{n}) := \Pi(k \rightarrow S \rightarrow \cdots \rightarrow S) & \hspace{1cm} 
\Pi(F_{4}) := \Pi(k \rightarrow k \rightarrow S \rightarrow S) \\
\Pi(C_{n}) := \Pi(k \rightarrow \cdots \rightarrow k \rightarrow S) & \hspace{1cm} 
\Pi(G_{2}) := \Pi(k \rightarrow S') 
\end{array}
$$


We present these examples in a pedagogical order. The first example, $\Pi(G_{2})$ can be carefully worked out by hand without machinery and mirrors the computations of $\Pi(D_{4})$. The second example, $\Pi(F_{4})$ requires more knowledge of preprojective algebras and is more difficult. The third and fourth examples, $\Pi(B_{n})$ and $\Pi(C_{n})$ use Gr\"{o}bner bases.

\begin{rem}
The $B_{n}$ and $C_{n}$ examples are long since we provide the reader with the additional information of how to arrive at a basis, instead of merely checking a set is a basis, in order to demystify the computation. We also give a proof of linear independence by applying the Buchberger's algorithm and showing that all ambiguities are resolvable. In light of the inequality in Proposition \ref{inequality} this is not strictly necessary, but highlights a technique available for computing the Hilbert--Poincar\'{e} series of decorated quivers not arising as degenerations. The reader can distill each subsection to half a page by first listing the Gr\"{o}bner basis of the ideal $(r) \subset P(Q, D_{Q})$ and the proposed spanning set. Then one can verify that the elements of the Gr\"{o}bner basis lie in the ideal generated by the relations and show that the proposed set spans modulo (multiples of) leading terms of the Gr\"{o}bner basis.
\end{rem}

\begin{rem} \label{signs}
Recall that flipping the orientation of an arrow $\alpha \in Q_{1}$ in a quiver $Q$ gives a new quiver, denote it $Q^{\alpha}$, with the same double $\overline{Q} = \overline{Q^{\alpha}}$. Moreover, one can identify the path algebras
$$
P(\overline{Q}) \rightarrow P(\overline{Q^{\alpha}}) 
$$ 
by sending
$$
\alpha \mapsto \alpha^{*} \hspace{1cm} \alpha^{*} \mapsto -\alpha.
$$

Now let $Q$ be a bipartite graph and fix a decomposition $Q_{0} = Q_{0}' \sqcup Q_{0}''$. By applying orientation flips, one can arrange for each source to lie in $Q_{0}'$ and each target in $Q_{0}''$. Such flips give equivalences on the path algebras of the doubled quivers, described above. Then observe for any arrow $\alpha \in Q_{1}$, one has
$$
\left ( \sum_{e' \in Q_{0}'} e' - \sum_{e'' \in Q_{0}''} e''
\right ) (\alpha \alpha^{*} - \alpha^{*} \alpha ) = \alpha \alpha^{*} + \alpha^{*} \alpha
$$
and
$$
\left ( \sum_{e' \in Q_{0}'} e' - \sum_{e'' \in Q_{0}''} e''
\right ) (\alpha \alpha^{*} + \alpha^{*} \alpha ) = \alpha \alpha^{*} - \alpha^{*} \alpha.
$$
Therefore, the two-sided ideal generated by $(\alpha \alpha^{*} - \alpha^{*} \alpha)$ equals that generated by $(\alpha \alpha^{*} + \alpha^{*} \alpha)$, and consequently one can drop the signs in the definition of the preprojective algebra. This argument applies equally well in the setting of decorated bipartite quivers, which includes all examples in the next section. Therefore, for computations in Section 5 we will drop all minus signs in the relations for convenience. 
\end{rem}

\subsection{Case: $\Pi(G_{2})$}
Let $Q = A_{2}$ with $D_{Q} = ( \{k, S:=k[x]/(x^{3}) \}, \{{}_{k} S_{S} \})$, visualized as
$$
\xymatrix{
k  \ar@/^1pc/[r]^{k \otimes_{k} S} & S. 
}
$$
To compute the preprojective algebra one doubles the quiver and doubles the decoration by adding the bimodule $\Hom_{k}( {}_{k} S_{S}, k) \cong {}_{S} S_{k}$.

We will begin by describing an explicit basis for the algebra $\Pi(G_{2}) := \Pi(Q, D_{Q})$. Since a degeneration cannot have smaller dimension, we have
$$
\dim \Pi(G_{2}) \geq \dim \Pi(D_{4}) = 28
$$
So we will produce a spanning set with 28 elements, which is automatically linearly independent by the above discussion on degenerations. 

\begin{rem}
To distinguish between the length zero paths $1 \in k$ and $1 \in S$ as well as the length one paths $1 \in {}_{S} S_{k}$ and $1 \in {}_{k} S_{S}$ we use subscripts: ${}_{k}1_{k}, {}_{S}1_{k}, {}_{k}1_{S}, {}_{S} 1_{S}$.  Often, the path $\gamma$ will be written $_{s(\gamma)} \gamma_{t(\gamma)}$, to delineate its location.
\end{rem}

\begin{prop} \label{Basis G_2}
The following set is a basis of the algebra 
$$
\Pi(G_{2})=
 T_{S \oplus k}({}_{k} S_{S} \oplus {}_{S} S_{k})/ \langle {}_{k} 1 \otimes_{S} 1_{k}  + 1 \otimes_{k} x^2 + x \otimes_{k} x + x^{2} \otimes_{k} 1 \rangle 
$$ 
as a vector space  over $k$, presented by length with $\alpha:={}_{k} 1 \otimes_{S} x_{k}$, $\beta := {}_{k} 1 \otimes_{S} x^{2}_{k}$ and $\cB := \{ 1, x, x^{2} \} \subset S$ a basis for $S$ over $k$: 
$$
\begin{array}{ll}
\text{Length 0: } \cB, {}_{k} 1_{k} &
\text{Length 3: } \alpha \otimes_{k} 1 \cdot \cB, \ \cB \cdot 1 \otimes_{k} \alpha  \\
\text{Length 1: } \cB \cdot {}_{S} 1_{k}, {}_{k} 1_{S} \cdot \cB  & 
\text{Length 4: } \alpha \otimes_{k} \beta, \ 1 \otimes_{k} \alpha \otimes_{k} 1 \cdot \cB \\  
\text{Length 2: }  \alpha, \ \beta , \ 1 \otimes_{k} 1 \cdot \cB, \ x \otimes_{k} 1 \cdot \cB.  &
\end{array}
$$
\end{prop}

\begin{cor}
${}_{S}1_{S} \cdot \Pi(G_{2})$ is a free left $S$-module and $\Pi(G_{2}) \cdot {}_{S}1_{S}$ is a free right $S$-module
\end{cor}

\begin{cor}
The Hilbert--Poincar\'{e} series for $\Pi(G_{2})$ is 
$$
h_{\Pi(G_{2})}(t) = 4 + 6t + 8t^2 + 6t^3 + 4t^4
$$ 
where the coefficient of $t^{i}$ is the dimension $\Pi^{i}(G_{2})$, the subspace spanned by length $i$ paths. Using the bigrading $(t, s)$ by path length and $x$-degree, one obtains the matrix-valued Hilbert--Poincar\'{e} series
$$
(1+t^2s)\left (
\begin{array}{cc}
1+t^2s^2 & (1+s+s^2)t \\
(1+s+s^2)t & (1+s+s^2)(1+t^2) \\
\end{array}
\right )
$$
where enumerating the vertices the $(i,j)$th entry is the Hilbert--Poincar\'{e} series $h_{e_{i} \Pi(G_{2}) e_{j}}(t, s)$.
\end{cor}

Define $M := S$ as a $(k, S)$-bimodule and $N :=S$ as a $(S, k)$-bimodule. Before presenting a proof, we make a few useful observations. Notice that in length 2,
$$
T^{2}_{S \oplus k}(M \oplus N)
=(M \oplus N) \otimes_{S \oplus k} (M \oplus N)
= M \otimes_{S} N \oplus N \otimes_{k} M
$$
and similarly for higher length terms: the two non-vanishing tensor products alternate in $M$ and $N$.  

Moreover, in odd length the two non-zero tensor products are isomorphic via a map reading each tensor product from right to left. In even length $2m$, the term starting with $M$ is a vector space generated by applying $1 \otimes_{S} -$ to paths of length $2m-1$. The term starting with $N$ is an $S$-bimodule generated as left $S$-module by applying $1 \otimes_{k} -$ to paths of length $2m-1$ and hence generated as a vector space by left multiplication by ${}_{S} 1_{k}, {}_{S} x_{k},$ and ${}_{S} x^2_{k}$.

Finally, notice that the tensor algebra $T_{S \oplus k}(M \oplus N)$ has a decomposition
\begin{align*}
T_{S \oplus k}(M \oplus N) &= (1_{k} + 1_{S})T_{S \oplus k}(M \oplus N) (1_{k} + 1_{S}) \\
& = {}_{k} T_{S \oplus k}(M \oplus N)_{k} \oplus {}_{k} T_{S \oplus k}(M \oplus N)_{S} \\
& \hspace{1cm} \oplus {}_{S} T_{S \oplus k}(M \oplus N)_{k} \oplus {}_{S} T_{S \oplus k}(M \oplus N)_{S}
\end{align*}
and an $\bN$-grading given by the number of times $x$ appears in a monomial expression. Since the relations $1 \otimes_{S} 1$ and $Rel:= 1 \otimes_{k} x^{2} + x \otimes_{k} x + x^{2} \otimes_{k} 1$ are homogeneous with respect to this grading and respect the above decomposition, the quotient $\Pi(G_{2})$ inherits both. We conclude that all relations are generated by relations involving only paths with the same starting and ending vertex of the same length and $x$-degree. We write ${}_{A_{i}}\Pi^{d}(G_{2})_{A_{j}}$ for the subspace of linear combinations of length $d$ paths starting at vertex $i$ and ending at vertex $j$.  


\begin{proof}[Proof of Proposition \ref{Basis G_2}]

Once we show that the given set spans, then  linear independence will follow from Proposition \ref{inequality} together with the fact that $\Pi(D_{4})$ is 28-dimensional. So we will only need to show the spanning property. 

For lengths 0 and 1, it is clear that these elements span. For higher lengths we now explain how any tensor product of lower length basis elements can be realized in the span of our given basis.

In length 2, $M \otimes_{S} N \cong S$ is three dimensional generated by $1 \otimes_{S} 1, \alpha,$ and $\beta$. Since $1 \otimes_{S} 1$ is a relation in $\Pi(G_{2})$, it follows that $\alpha$ and $\beta$ generate ${}_{k}\Pi^{2}(G_{2})_{k}$. $N \otimes_{k} M$ is generated by $1 \otimes_{k} 1$ as an $S$-bimodule and hence is 9-dimensional as a vector space, generated by the six elements $1 \otimes_{k} 1 \cdot \cB$ and $x \otimes_{k} 1 \cdot \cB$ together with the elements
\begin{align*}
&Rel := 1 \otimes_{k} x^{2} + x \otimes_{k} x + x^{2} \otimes_{k} 1 \\
 &x \cdot Rel = x \otimes_{k} x^{2} + x^{2} \otimes_{k} x \\
 &x^{2} \cdot Rel = x^{2} \otimes_{k} x^{2} 
\end{align*}
all relations in $\Pi(G_{2})$.

In length 3, both direct summands have the same dimension, and applying $- \otimes_{k} M$ to the 2-dimensional vector space $\Span \{ \alpha, \beta \} = M \otimes_{S} N/ \langle 1 \otimes_{S} 1 \rangle $ yields the spanning set,
$$
\alpha \otimes_{k} \cB  \cup \beta \otimes_{k} \cB. 
$$
For each $n \in \{2, 3, 4 \}$, if we sum all of the above elements of $x$-degree equal to $n$, we obtain the relation $x^{n-2} \otimes_{S} Rel$.  Hence the smaller set $\alpha \otimes_{k} \cB$ spans ${}_{k}\Pi^{3}(G_{2})_{S}$, as desired.

In length 4, first notice that the quotient ${}_{k}\Pi^{4}(G_{2})_{k}$ of $M \otimes_{S} N \otimes_{k} M \otimes_{S} N$ has a spanning set
$$
\alpha \otimes_{k} \cB \otimes_{S} 1,
$$
the image of the spanning set $\alpha \otimes_{k} \cB$ under the surjective map $-\otimes_{S} 1$. The term $\alpha \otimes_{k} 1 \otimes_{S} 1$ is zero in the quotient, as is the term
$$
\alpha \otimes x \otimes_{S} 1 \equiv 1 \otimes_{S} Rel \otimes_{S} 1 \text{ mod } 1 \otimes_{S} 1.
$$
Next the quotient ${}_{S}\Pi^{4}(G_{2})_{k}$ of $N \otimes_{k} M \otimes_{S} N \otimes_{k} M$ is generated as a left $S$-module by
$$
1 \otimes_{k} 1 \otimes_{S} x \otimes_{k} \cB.
$$
But, this set also generates as a vector space since,
$$
x \cdot (1 \otimes_{k} 1 \otimes_{S} x \otimes_{k} 1) - (1 \otimes_{k} 1 \otimes_{S} x \otimes_{k} 1) \cdot x \equiv 1 \otimes_{k} 1 \otimes_{S} Rel - Rel \otimes_{S} 1 \otimes_{k} 1
$$
are congruent modulo $1 \otimes_{S} 1$, 
so the left action is equivalent to a right action, which produces no new terms as $S \otimes_{S} S \cong S$. 

To show $\Pi^{\geq 5}(G_{2})= 0$, it suffices to show that $\Pi^{5}(G_{2}) = 0$, and by the isomorphism between the direct summands, it suffices to show that ${}_{k}\Pi^{5}(G_{2})_{S}$ is zero. As a right $S$-module ${}_{k}\Pi^{5}(G_{2})_{S}$ is generated by the element,
$$
1 \otimes_{S} x \otimes_{k} 1 \otimes_{S} x^{2}  \otimes_{k} 1,
$$
which is the relation $1 \otimes_{S} x \otimes_{k} 1 \otimes_{S} Rel + 1 \otimes_{S} Rel \otimes_{S} 1 \otimes_{k} x$ modulo $1 \otimes_{S} 1$, and hence zero in $\Pi(G_{2})$.
We conclude that our given set spans and hence forms a basis for $\Pi(G_{2})$.
\end{proof}

\begin{rem}
Here is an alternative proof of linear independence that does not appeal to degeneration arguments, but rather examines the relations carefully. The technical computations are left to the reader.

Notice any two basis elements are in different bigradings or start at different vertices and hence are linearly independent with the following exceptions
$$
x \otimes_{k} 1 , \ 1 \otimes_{k} x
\hspace{1cm} \text{and} \hspace{1cm}
x \otimes_{k} x, \ 1 \otimes_{k} x^{2} 
$$
which can be seen to be linearly independent directly from the definition. To see that all listed paths are non-zero observe each is a subpath of $1 \otimes_{S} x \otimes_{k} 1 \otimes_{S} x^2$ or $1 \otimes_{k} 1 \otimes_{S} x \otimes_{k} x^2$, which can be shown non-zero by writing out the length four elements in the two-sided idea generated by the relations. 
\end{rem}

Next notice that $\Pi(G_{2})$ is Frobenius with the form given by projecting to the one-dimensional space, spanned by the element,
$$
1 \otimes_{S} x \otimes_{k} 1 \otimes_{S} x^{2} +  1 \otimes_{k} 1 \otimes_{S} x \otimes_{k} x^{2} 
$$
summing over both cycles of maximal bigrading (4, 3). Non-degeneracy of this form amounts to the fact that every path is a subpath of a maximal length cycle, which can be checked by hand. This Frobenius form deforms to the usual one on $\Pi(D_{4})$.

\begin{cor}
$\Pi(G_{2})$ is a flat Frobenius degeneration of $\Pi(D_{4})$.
\end{cor} 

Restricting to the above basis, the Frobenius pairing, given by first multiplying and then applying the Frobenius form, takes values in $\{ -1, 0, 1 \}$ with the $-1$ values given by the pairs:
$$
\begin{array}{lll} 
(1 \otimes_{S} x, \ 1 \otimes_{S} x^{2}) & 
(1 \otimes_{k} x, \ 1 \otimes_{k} x^{2}) & 
(x \otimes_{k} 1, \ x \otimes_{k} x)  \\ 
({}_{S} x_{k} , 1 \otimes_{S} x \otimes_{k} x) & 
({}_{k} x^{2}_{S}, {}_{S} 1 \otimes_{k} 1 \otimes_{S} x) & 
({}_{k} x_{S}, {}_{S} x \otimes_{k} 1 \otimes_{S} x) \\ 
({}_{k} 1_{S}, {}_{S} x^{2} \otimes_{k} 1 \otimes_{S} x) & &   \\
\end{array}
$$
The Nakayama automorphism is given by $a \mapsto -a$ on the elements appearing in the above pairings and is the identity on all other basis elements, and hence squares to the identity. \\
\\
Using computations from the proof of Proposition \ref{Basis G_2}, it is not hard to compute the center, $Z(\Pi(G_{2}))$. 

\begin{prop}
$Z(\Pi(G_{2})) = k \cdot 1 \oplus \Pi^{4}(G_{2}) \cong k[0] \oplus (k \oplus S)[4]$.
\end{prop}

\begin{proof}
The Frobenius structure allows one to identify $\Pi^{i}(G_{2})$ with $\Pi^{4-i}(G_{2})^{*}$ for $i \in \{ 0, 1, 2, 3, 4 \}$. In particular, the Frobenius pairing identifies $\Pi^{4}(G_{2})$ with $S \oplus k$, explaining the second isomorphism. 

$\Pi^{4}(G_{2}) \subset Z(\Pi(G_{2}))$ as one only needs to check that the commutator $[ {}_{S} x_{S}, 1 \otimes_{k} 1 \otimes_{S} x \otimes_{k} 1 ]$ vanishes, which is done in the proof of Proposition \ref{Basis G_2} above. Conversely, $Z(\Pi(G_{2}) \subset k[0] \oplus k[4] \oplus S[4]$ as the centralizer of ${}_{k} 1_{k}$ excludes all paths of odd length, the centralizer of ${}_{k} 1_{S}$ excludes all paths of even length less than 4 except the the span of the identity and the span of $\{ 1 \otimes_{k} \gamma \}_{\gamma \in \{ 1, x, x^2 \}}$. The latter span is not in the centralizer of ${}_{k} x_{S}$.
\end{proof}


\subsection{Case: $\Pi(F_{4})$}
Now we consider the example $Q= A_{4}$ with decoration \\
 $D_{Q}=(\{k, k, S, S\}, \{k, k \otimes_{k} S, S\})$. The doubled decorated quiver is,
$$
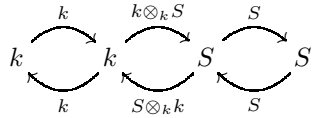
\xymatrix{
k \ar@/^1pc/[r]^{k} & k \ar@/^1pc/[l]^{k} \ar@/^1pc/[r]^{k \otimes_{k} S} & S \ar@/^1pc/[l]^{S \otimes_{k} k} \ar@/^1pc/[r]^{S} & S \ar@/^1pc/[l]^{S}
}
$$
We want to compute a basis for the decorated preprojective algebra \\
$\Pi(F_{4}) := \Pi( Q, D_{Q})$. The length zero paths
$$
e_{1} := {}_{1} 1_{1},  \ \
e_{2} := {}_{2} 1_{2}, \ \  
e_{3} := {}_{3}1_{3}, \ \ 
e_{4} := {}_{4} 1_{4},  \ \
{}_{3} x_{3}, \ \
 {}_{4} x_{4}
$$
and length one paths
$$
{}_{1} 1_{2}, \ \  {}_{2} 1_{1}, \ \ {}_{2} 1_{3}, \ \ {}_{3} 1_{2}, \ \ {}_{3} 1_{4}, \ \ {}_{4} 1_{3}
$$
form a generating set. The quadratic relation $R$ can be multiplied by the length zero path $e_{i}$ to get a relation $Ri$ at each vertex $i$, given by,

$$
\begin{cases}
R1:= {}_{1} 1_{2} \otimes_{k} {}_{2}1_{1} & \text{ at vertex } 1 \\
R2:= {}_{2} 1_{1} \otimes_{k} {}_{2} 1_{1} + {}_{2} 1_{3} \otimes_{S} {}_{3} 1_{2} & \text{ at vertex } 2 \\
R3:= {}_{3} 1_{2} \otimes_{k} {}_{2} x_{3} + {}_{3} x_{2} \otimes_{k} {}_{2} 1_{3} + {}_{3} 1_{4} \otimes_{S} {}_{4} 1_{3} & \text{ at vertex } 3 \\
R4:= {}_{4} 1_{3} \otimes_{S} {}_{3} 1_{4} & \text{ at vertex } 4. \\ 
\end{cases}
$$
Appealing directly to the definition, the relations at vertices 1, 2, and 4 are the usual commutators, ignoring signs in light of Remark \ref{signs}. At vertex 3, one needs to use the Frobenius form $\lambda: S=k[x]/(x^{2}) \rightarrow k$ by $\lambda(a + bx) := b$ to get the dual basis $\{ x, 1 \}$ to the basis $\{1, x\}$, and hence the canonical element ${}_{3}1_{2} \otimes_{k} {}_{2}x_{3} + {}_{3}x_{2} \otimes_{k} {}_{2}1_{3}$.

\begin{rem}  
Consider ${}_{3} 1_{4}$ and ${}_{4} 1_{3}$ to have $x$-grading $1/2$, so that the relation at vertex 3 is homogeneous in the $x$-grading. We make this choice so the $x$-grading is symmetric, but if one would like to realize this algebra as a filtered degeneration of $\Pi(E_{6})$ then one needs to choose integral $x$-gradings, such as $|{}_{3} 1_{4} |_{x} = 1$ and $|{}_{4} 1_{3} |_{x} = 0$.
\end{rem}

Here we explain the methodology to compute a basis for this algebra, leaving the details to the reader. Recall it suffices to provide a spanning set with exactly dim($\Pi(E_{6})$) = 156 elements. 

\begin{rem} \label{rem:notationforpaths}
To simplify notation for longer paths we define
$$
\gamma_{a,b} :=
\begin{cases}
{}_{a} 1_{a+1} \cdot {}_{a+1} 1_{a+2} \cdots  {}_{b-1} 1_{b} & \text{ if } a < b\\
{}_{a} 1_{a} & \text{ if } a = b \\
{}_{a} 1_{a-1} \cdot {}_{a-1} 1_{a-2} \cdots {}_{b+1} 1_{b} & \text{ if } a> b\\
\end{cases}
$$
using an ordering on the vertices $Q_{0}$.
Then define 
$$
\gamma_{i_{1}, i_{2}, \dots, i_{n} } := \gamma_{i_{1}, i_{2}} \gamma_{i_{2}, i_{3}} \cdots \gamma_{i_{n-1}, i_{n}}
$$
where if $i_{j} < i_{j+1} < i_{j+2}$ or $i_{j} > i_{j+1} > i_{j+2}$ then we remove $i_{j+1}$ from the notation in order to solely delineate changes in direction.
\end{rem}

Observe that every path is equivalent to one passing through vertex 3 except those belonging to $X := \{ {}_{1}1_{1}, {}_{2}1_{2}, {}_{4}1_{4}, {}_{4} x_{4}, {}_{1} 1_{2}, {}_{2} 1_{1} \}$. Exploiting this observation, for any pair of vertices $i$ and $j$ the map,
$$
\Phi_{i, j}: e_{3} \Pi(F_{4}) e_{3} \mapsto e_{i} \Pi(F_{4}) e_{j} / k \cdot e_{i}Xe_{j} \hspace{1cm} \alpha \mapsto  \gamma_{i,3} \alpha \gamma_{3,j}
$$ 
is surjective, where $k \cdot e_{i}Xe_{j}$ denotes the $k$-span of paths in $X$ from $i$ to $j$. Therefore, 
$$
e_{i} \Pi(F_{4}) e_{j} \cong \Phi_{i,j}( e_{3} \Pi(F_{4}) e_{3}) \oplus k \cdot e_{i}X e_{j}
$$
as $k$-vector spaces and we have the equality
\begin{align*}
\dim(\Pi(F_{4})) 
&= \sum_{i,j} \dim( e_{i} \Pi(F_{4}) e_{j}) \\
&= \dim k \cdot X + \sum_{i,j} \dim (\Phi_{i,j}(e_{3} \Pi(F_{4}) e_{3})) \\
&=  \dim k \cdot X + \sum_{i,j} \dim ( e_{3} \Pi(F_{4}) e_{3}) - \dim( \ker(\Phi_{i,j}) ) \\
&= 6 + 16 \dim ( e_{3} \Pi(F_{4}) e_{3})
- \sum_{i,j}  \dim( \ker(\Phi_{i,j}) ).
\end{align*}

Since $\dim( \ker(\Phi_{j,i}) ) = \dim( \ker(\Phi_{i,j}))$, we've reduced showing the inequality $\dim(\Pi(F_{4})) \leq 156$ to computing a spanning set for $e_{3} \Pi(F_{4}) e_{3}$ and linearly independent subsets for ten subspaces, $\ker( \Phi_{i, j} )$ for $i \leq j$. \\
\\
We compute a spanning set for $e_{3} \Pi(F_{4}) e_{3}$ by building a surjection $q: k \langle \alpha, \beta \rangle \twoheadrightarrow e_{3} \Pi(F_{4}) e_{3}$, and then reducing the spanning set by identifying elements in the kernel of $q$. Every path in $e_{3} \Pi(F_{4}) e_{3}$ is equivalent to one avoiding vertex 1 and vertex 4, using the relations at vertex 2 and 3, respectively. Therefore, $e_{3} \Pi(F_{4}) e_{3}$ is generated by $q(\alpha) :={}_{3} 1_{2} \otimes_{k} {}_{2} 1_{3}$ and $q(\beta) := {}_{3} x_{3}$. They satisfy the relations: 
\begin{align*}
q(\alpha)^{3} &= {}_{3} 1_{2} \otimes_{k} {}_{2} 1_{3} \otimes_{S}  
 {}_{3} 1_{2} \otimes_{k} {}_{2} 1_{3} \otimes_{S} {}_{3} 1_{2} \otimes_{k} {}_{2} 1_{3} \\
 & \overset{R2}{=} 
{}_{3} 1_{2} \otimes_{k} {}_{2} 1_{1} \otimes_{k} {}_{1} 1_{2} \otimes_{k} {}_{2} 1_{1} \otimes_{k} {}_{1} 1_{2} \otimes_{k} {}_{2} 1_{3}  \\
& \overset{R1}{=} 0, 
\end{align*}
$q(\alpha \beta + \beta \alpha)^{2}  \overset{R3}{=} ({}_{3} 1_{4} \otimes {}_{4} 1_{3})^{2} \overset{R4}{=} 0$,
and $q(\beta^{2}) = x^{2} = 0$.
Hence $q$ factors through the quotient, 
$$
q': \frac{ k \langle \alpha, \beta \rangle }
{(\alpha^{3}, \beta^{2}, (\alpha \beta + \beta \alpha)^{2} )} \twoheadrightarrow e_{3} \Pi(F_{4}) e_{3}.
$$
Let $C$ denote $k \langle \alpha, \beta \rangle /(\alpha^{3}, \beta^{2}, (\alpha \beta + \beta \alpha)^{2} )$. 
We find a $k$-basis for $C$ in Example \ref{exam:C}, after introducing the Diamond Lemma, given by:
\begin{itemize}
\item[(0)] $1,  \ \beta $
\item[(2)] $\alpha, \ \alpha \beta, \ \beta \alpha, \ \beta \alpha \beta$
\item[(4)] $\alpha^2, \ \alpha^2 \beta, \ \alpha \beta \alpha, \ \alpha \beta \alpha \beta, \ \beta \alpha^2, \ \beta \alpha^{2} \beta$
\item[(6)] $\alpha^2 \beta \alpha, \ \alpha^2 \beta \alpha \beta, \ \alpha \beta \alpha^2,  \ \alpha \beta \alpha^{2} \beta, \ \beta \alpha^{2} \beta \alpha, \ \beta \alpha^{2} \beta \alpha \beta$
\item[(8)] $\alpha^2 \beta \alpha^2, \ \alpha^{2} \beta \alpha^{2} \beta, \ \alpha^{2} \beta \alpha^{2} \beta, \ \alpha \beta \alpha^{2} \beta \alpha, \ \alpha \beta \alpha^{2} \beta \alpha \beta$
\item[(10)] $ \alpha^2 \beta \alpha^{2} \beta \alpha, \ \alpha^2 \beta \alpha^{2} \beta \alpha \beta.$
\end{itemize}
Here we have ordered the elements first by number of occurrences of $\alpha$ and then by alphabetical order. We conclude that $C$ is a free rank 12 right $S$-module and has dimension 24 as a $k$-module. This shows $q'$ is an isomorphism since,
$$
24 = \dim( (e_{3}+e_{5})\Pi(E_{6})(e_{3}+e_{5}) )
\leq \dim( e_{3} \Pi(F_{4}) e_{3} ) \leq \dim(C) = 24.
$$
The first inequality is a refinement of $\dim(\Pi(E_{6})) \leq \dim(\Pi(F_{4}))$ formed by keeping track of the decomposition of paths by starting and ending vertices. Here $E_{6}$ is folded to $F_{4}$ by identifying vertices 3 and 5 and vertices 4 and 6. In particular,
$$
\dim(\Pi(F_{4})) \leq 6 + 16(24) - \sum_{i,j} \dim( \ker(\Phi_{i,j})) = 390 - \sum_{i,j} \dim( \ker(\Phi_{i,j}))
$$
and it remains to bound the dimension of the kernel of each map $\Phi_{i,j}$ from below. 

Observe that $\gamma_{3, j} \alpha^{j} = 0 = \alpha^{j} \gamma_{j, 3}$ for all $j \in \{1, 2, 3\}$ and 
$\gamma_{3, 4} (\alpha \beta + \beta \alpha) = 0 = (\alpha \beta + \beta \alpha)\gamma_{3, 4}$. Hence we have the following containment,
\begin{align*}
\ker(\Phi_{i, j})&= \ker( \Phi_{i, 3} \circ \Phi_{3, j}) = \ker( L_{\gamma_{i, 3}} \circ R_{\gamma_{3, j}}) = \ker ( R_{\gamma_{3, j}} \circ L_{\gamma_{i, 3}}) \\
& \supset  \ker ( R_{\gamma_{3, j}}) \cup \ker( L_{\gamma_{i, 3}}) \\
& \supset \langle \alpha^{i}, (\alpha \beta + \beta \alpha)^{5-i} \rangle_{R} \cup \langle \alpha^{j}, (\alpha \beta + \beta \alpha)^{5-j}) \rangle_{L},
\end{align*}
where $\langle s \rangle_{R}$ and $\langle s \rangle_{L}$ denote the right and left ideals in $C$ generated by $s$ and $L_{s}$ and $R_{s}$ denote left and right multiplication by $s \in C$. 

For each $i, j$, one can find linearly independent elements of the above ideals such that the total number of elements is 390-156. For instance, if $i=4$ and $j=3$, then
$\ker( \Phi_{4, 3} ) \supset \langle \alpha^{i}, (\alpha \beta + \beta \alpha)^{5-i} \rangle_{R} \cup \langle \alpha^{j}, \alpha \beta + \beta \alpha)^{5-j} \rangle_{L}
= \langle \alpha \beta + \beta \alpha \rangle_{R}.$

Using a basis for $C$, it is not hard to find a basis for this ideal, given by 
\begin{align*}
\langle \alpha \beta + \beta \alpha \rangle_{R}
= \Span \{
\alpha \beta & + \beta \alpha, \
 \beta \alpha \beta, \
\alpha^{2} \beta + \alpha \beta \alpha, \
 \alpha \beta \alpha \beta, \ 
\alpha^{2} \beta \alpha, \
\alpha^{2} \beta \alpha \beta, \
\beta \alpha^{2} \beta \alpha, \\
&\alpha \beta \alpha^{2} \beta \alpha, \
\alpha \beta \alpha^{2} \beta \alpha \beta, \ \alpha \beta \alpha^{2} \beta \alpha \beta, \ \alpha^{2} \beta \alpha^{2} \beta \alpha, \ \alpha^{2} \beta \alpha^{2} \beta \alpha \beta
\}
\end{align*}
which is 12-dimensional. Hence $\Phi_{4,3}(C)$ is at most 12-dimensional. 

Note that the number of linearly independent elements needed for each $i, j$ is
$$
\dim C - \sum_{i' \mapsto i} \sum_{j' \mapsto j} \dim (e_{i'} \Pi(E_{6}) e_{j'})
- \dim( e_{i'} \Pi(E_{6} \backslash \{ 3, 5 \}))
e_{j'} )
$$
where the sum is taken over all vertices $i'$ and $j'$ that fold to $i$ and $j$, respectively. So $E_{6} \backslash \{ 3, 5 \} = A_{1} \sqcup A_{1} \sqcup A_{2}$ is the disconnected quiver obtained by discarding vertices 3 and 5 and all arrows with either 3 or 5 as source or target. If $i=4$ and $j=3$ this is precisely 
$$
24 - 2 \dim( e_{3} \Pi(E_{6}) e_{4}) - 2 \dim(e_{3} \Pi(E_{6}) e_{6}) = 12,  
$$
as desired. 

The remaining computations are done similarly and are left to the reader. The upshot is that the $4 \times 4$ matrix with $(i,j)$th entry given by $e_{i} \Pi(F_{4}) e_{j}$ is
$$
\left ( 
\begin{array}{cccc}
4 & 6 & 8 & 4 \\
6 & 12 & 16 & 8 \\
8 & 16 & 24 & 12 \\
4 & 8 & 12 & 8 \\
\end{array}
\right ).
$$
For each entry of this matrix, we compute the Hilbert--Poincar\'{e} series, a polynomial in variables $s$ and $t$ defined so the coefficient of $t^{n}s^{m}$ is dimension of $\Pi^{(n,m)}(F_{4})$, i.e. the path length $n$ and $x$-grading $m$ component of $\Pi(F_{4})$. That is,
\begin{prop}
$h_{\Pi(F_{4})}(s, t) = (1+t^4s) \cdot$  
$$
\footnotesize{
\left ( 
\arraycolsep=3pt
\medmuskip=1mu
\begin{array}{cccc}
(1+t^6s)  & t(1+t^2s+t^4s^s) & t^2(1+s)(1+t^2s) &  t^3s^{1/2}(1+s) \\
t(1+t^2s+t^4s^2) & (1+t^2)(1+t^2s+t^4s) & t(1+s)(1+t^2s) &
                                            t^2s^{1/2}(1+s)(1+t^2) \\
t^2(1+s)(1+t^2s) & t(1+s)(1+t^2s) & (1+s)(1+t^2+t^4)(1+t^2s)
                                      & ts^{1/2}(1+s)(1+t^2+t^4) \\
 t^3s^{1/2}(1+s) &  t^2s^{1/2}(1+s)(1+t^2) & ts^{1/2}(1+s)(1+t^2+t^4)  &             
                                              (1+s)(1+t^6s)\\ 
\end{array}
\right )
}
$$
and 
$$
h_{\Pi(F_{4})}(t, 1) = 6 + 10t + 14t^{2} + 18t^{3} + 20t^{4} + 20t^{5} + 20t^{6} + 18t^{7} + 14t^{8} + 10 t^{9} + 6 t^{10}
$$
where the coefficient in front of $t^{n}$ is the dimension of the vector space of paths of length $n$ in $\Pi(F_{4})$.
\end{prop}

As before, the Frobenius pairing is defined by projecting onto the (10, 3)-bigraded piece and adding the coefficients. In more detail, if $\cB$ is a basis of $\Pi(F_{4})$ with dual basis $\cB^{\vee} = \{ \gamma^{\vee} : \gamma \in \cB \}$, then
$$
(\alpha, \beta) := \sum_{\gamma^{\vee} \in Y^{\vee}} \gamma^{\vee}(\alpha \beta)
$$
where
$$
Y := \{ {}_{1} 1_{2} \beta \alpha \beta \beta {}_{2} 1_{1}, \ \
\alpha \beta \beta \alpha \beta, \ \ 
{}_{3} 1_{2} \beta \alpha \beta \beta {}_{2} 1_{3}, \ \
{}_{4} 1_{3} \cdot{}_{3} 1_{2} \alpha \beta \beta {}_{2} 1_{3} \cdot{}_{3} 1_{4} \}
$$ 
is the set of paths of $(10, 3)$-bidegree in the basis for $\Pi(F_{4})$. This Frobenius pairing deforms to the usual one in $\Pi(E_{6})$, defined analogously.

\begin{cor}
$\Pi(F_{4})$ is a flat Frobenius degeneration of $\Pi(E_{6})$.
\end{cor}

\subsection{Case: $\Pi(B_{n})$}

Consider the decorated double quiver
$$
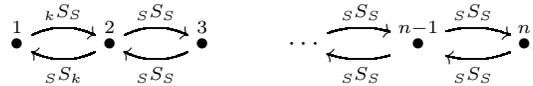
\xymatrix{
\overset{1}{\bullet} \ar@/^/[r]^{{}_{k} S_{S}} & \overset{2}{\bullet} \ar@/^/[l]^{{}_{S} S_{k}} \ar@/^/[r]^{{}_{S} S_{S}} & \overset{3}{\bullet}  \ar@/^/[l]^{{}_{S} S_{S}} &
\cdots \ar@/^/[r]^{{}_{S} S_{S}} & \ar@/^/[l]^{{}_{S} S_{S}}  \overset{n-1}{\bullet} \ar@/^/[r]^{{}_{S} S_{S}} & \overset{n}{\bullet} \ar@/^/[l]^{{}_{S} S_{S}} .
}
$$
The relations in the preprojective algebra are
$$
\begin{cases}
R1 := {}_{1} 1_{2} \otimes_{S} {}_{2}1_{1} & \text{ at vertex } 1 \\
R2 := {}_{2} 1_{1} \otimes_{k} {}_{1} x_{2} +{}_{2} x_{1} \otimes_{k} {}_{1} 1_{2} + {}_{2} 1_{3} \otimes_{S} {}_{3} 1_{2} & \text{ at vertex } 2 \\
Rj := {}_{j} 1_{j+1} \otimes_{S} {}_{j+1} 1_{j} + {}_{j} 1_{j-1} \otimes_{S} {}_{j-1} 1_{j} & \text{ at vertex } j \in \{3, 4, \dots, n-1 \} \\
Rn := {}_{n} 1_{n-1} \otimes_{S} {}_{n-1} 1_{n} & \text{ at vertex } n. \\ 
\end{cases}
$$
These relations are homogeneous in path length and can be made homogeneous in $x$-degree by declaring the $x$-grading of ${}_{i} 1_{i+1}$ and ${}_{i+1} 1_{i}$ to be 1/2 for $i >1$. Therefore, the preprojective algebra inherits both gradings from the the path algebra. 

The idea is to produce a basis for $\Pi(B_{n})$ by realizing it as a subalgebra of $P(Q, D_{Q})$, denoted $P(Q, D_{Q})_{\irr}$. 
The key ingredient is an algebra map 
$$
\opr: P(Q, D_{Q})\rightarrow P(Q, D_{Q})_{\irr}
$$ 
satisfying:
$$
(1) \ \ \ker(\opr) = \ker( P(Q, D_{Q}) \rightarrow \Pi(B_{n}) )
\hspace{1cm} \text{ and } \hspace{1cm}
(2) \ \ \opr \circ i = id_{P(Q, D_{Q})_{\irr}},
$$
where $i: P(Q, D_{Q})_{\irr} \rightarrow P(Q, D_{Q})$ is the natural inclusion. 
The first condition says $\opr$ descends to $\Pi(B_{n})$ and the second condition implies $\opr$ is surjective. Therefore $r$ induces an identification $\Pi(B_{n}) \cong P(Q, D_{Q})_{\irr}$ as algebras. 

To define $\opr$, we choose a partial ordering $\leq$ on paths in $P(B_{n})$. This gives rise to a notion of leading term, for any linear combination of paths, denoted $\oplt$. Then, for each relation $Rj$ in $\Pi(B_{n})$, define a $k$-linear map $r_{j}$ which is the identity on all monomials except taking the leading term $\oplt(Rj)$ to the strictly smaller element $\oplt(Rj) - Rj$. This map is called a reduction and informally it imposes the relation $Rj$ only in the direction that makes the path smaller with respect to the partial order.

\begin{rem}
For any pair of monomials $A, B \in P(Q, D_{Q})$, $A Rj B$ is also a relation and hence one defines the reduction $r_{AjB}$ taking $A(\oplt(Rj))B$ to $A(\oplt(Rj)-Rj)B$. Since $A$ and $B$ will be clear in context, we simplify but abuse notation and write $r_{j}$ for any such reduction.
\end{rem}

 One would like to define $\opr$ to be any finite composition of these reductions such that the resulting expression is irreducible, i.e. every reduction acts as the identity. The existence of $\opr$ will be clear in our context since each path will have a finite disorder index. This index is a natural number measuring the failure of a path to be irreducible and it strictly decreases under reductions, and hence is eventually zero. However, $\opr$ is, in general, \emph{not} well-defined, as it depends on the order that the reductions are applied. 
 
The Diamond Lemma in \cite{Bergman} says that $\opr$ is well-defined, if whenever a product of monomials $\alpha \beta \gamma$ is reduced by $r_{1}$ and $r_{2}$ to elements $r_{1}(\alpha \beta) \gamma$ and $\alpha r_{2}(\beta \gamma)$, each further reduces to the same irreducible element. Such a 5-tuple $(\alpha, \beta, \gamma, r_{1}, r_{2})$ where $\alpha \beta$ is the leading term for $r_{1}$ and $\beta \gamma$ is the leading term for $r_{2}$, is called an \emph{ambiguity}. And we say the ambiguity \emph{resolves} if $r_{1}(\alpha \beta) \gamma$ and $\alpha r_{2}(\beta \gamma)$ have the same reduced form. 


The name Diamond Lemma comes from the visual that every diagram
$$
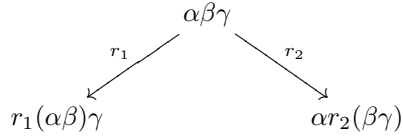
\xymatrix{
& \alpha \beta \gamma \ar[rd]^{r_{2}} \ar[ld]_{r_{1}} & \\
r_{1}(\alpha \beta) \gamma & & \alpha r_{2}(\beta \gamma)
}
$$
called an ambiguity, can be completed to a diagram,
$$
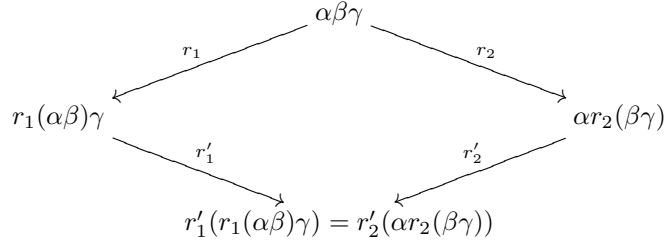
\xymatrix{
& \alpha \beta \gamma \ar[rd]^{r_{2}} \ar[ld]_{r_{1}} & \\
r_{1}(\alpha \beta) \gamma \ar[rd]^{r_{1}'}& & \alpha r_{2}(\beta \gamma) \ar[ld]_{r_{2}'} \\
& r_{1}'( r_{1}(\alpha \beta) \gamma) = r_{2}'(\alpha r_{2}(\beta \gamma))
}
$$
for some composition of reductions, $r_{1}'$ and $r_{2}'$, called a resolution. 

If every ambiguity is resolvable then 
$\opr$ is well-defined and clearly satisfies (1) and (2) above. Hence $i$ is a splitting for $\opr$ and realizes,
$$
P(Q, D_{Q}) \cong P(Q, D_{Q})_{\irr} \oplus I
$$
the decorated preprojective algebras as a \emph{subalgebra}, as opposed to a quotient, of the decorated path algebra. From this perspective a basis for the decorated preprojective algebras is given by a basis for the irreducible paths. 

\begin{exam} \label{exam:C}
Consider $C = k \langle \alpha, \beta \rangle / ( \alpha^{3}, \beta^{2}, (\alpha \beta + \beta \alpha)^{2})$ from the previous section. One defines a partial order on monomials in $k \langle \alpha, \beta \rangle$ by $\gamma \leq \gamma'$ if the length of $\gamma$ is less than the length of $\gamma'$ or they have the same length but $\gamma$ is first alphabetically. Then the three relations are viewed as reductions:
$$
r_{1}(\alpha^{3}) = 0, \ \ r_{2}(\beta^{2}) = 0, \ \ r_{3}(\beta \alpha \beta \alpha) = -\alpha \beta \alpha \beta - \beta \alpha \alpha \beta.
$$
Reducing $\beta \alpha \beta \alpha \alpha^{2}$ using $r_{1}$ gives zero and using $r_{3} \circ r_{3} \circ r_{3}$ gives $\beta \alpha^{2} \beta \alpha^{2} - \alpha^{2} \beta \alpha^{2} \beta + \alpha \beta \alpha^{2} \beta \alpha$. And hence we add the reduction, $r_{4}( \beta \alpha^{2} \beta \alpha^{2})
= \alpha^{2} \beta \alpha^{2} \beta - \alpha \beta \alpha^{2} \beta \alpha $, and the remaining ambiguities, $\beta \beta \alpha^{2} \beta \alpha^{2}$ and $\beta \alpha^{2} \beta \alpha^{2} \alpha$, are resolvable. We conclude that a word is reduced if it has no occurrences of $\alpha^{3}, \beta^{2}, \beta \alpha \beta \alpha$, or $\beta \alpha^{2} \beta \alpha^{2}$.

Reductions $r_{1}$ and $r_{2}$ reduce any word to one alternating between $\beta$ and either $\alpha$ or $\alpha^{2}$. The reduction $r_{3}$ says a reduced word has only a single $\beta \alpha$, which necessarily occurs at the end. The reduction $r_{4}$ says any reduced word starting with $\beta$ cannot have repeating $\alpha^{2}$. Hence the longest reduced word is $\alpha^{2} \beta \alpha^{2} \beta \alpha \beta$ and it is not hard to realize all 24 basis elements of $C$ as subwords. 
\end{exam}
 
Define the following ordering of the generators at each vertex $i$
$$
{}_{i} x_{i-1} > {}_{i} x_{i+1} > {}_{i} 1_{i-1} > {}_{i} 1_{i+1}
$$
with $1 < i < n$ and $_{n} x_{n-1} > {}_{n} 1_{n-1}$ and extend this to a partial ordering on all paths lexicographically. That is, view each path $\gamma$ of length $l$ as a product of length one paths $\gamma = \gamma_{1}\cdot \gamma_{2} \cdot \cdots \cdot \gamma_{l}$ and define $\alpha > \beta$ for paths $\alpha$ and $\beta$ if
\begin{itemize}
\item[(I)] $\alpha$ is longer than $\beta$ or 
\item[(II)] they are the same length with $\alpha$ having the larger $x$-grading or
\item[(III)] they have the same length and $x$-grading with $\alpha_{i} > \beta_{i}$ for the smallest $i$ where $\alpha_{i} \neq \beta_{i}$.
\end{itemize}
To make this definition well-defined one should write the any subpath $x \otimes_{S} 1$ as $1 \otimes_{S} x$, always moving the $x$'s towards the end of the path. 

The relations $R1, R2, \cdots Rn$ give rise to a system of reductions, with respect to the lexicographical ordering.
$$
\begin{array}{ll}
\text{Upward Reductions:} & 
\text{Integral Reductions:} \\
\ \ \ \bullet \ \ r^{U}_{j}({}_{j} 1_{j-1} \otimes_{S} {}_{j-1} 1_{j} ) = {}_{j} 1_{j+1} \otimes_{S} {}_{j+1} 1_{j}  & 
\ \ \ \bullet \ \ r^{I}_{j}( \gamma_{2, 1, j+1, j} ) = \gamma_{2, 3, 1, j} \\
\hspace{2cm} \text{for } 3 \leq j \leq n-1 & 
\hspace{2cm} \text{for } 2 \leq j \leq n-1 \\
\ \ \ \bullet \ \ r^{U}_{n}({}_{n} 1 \otimes_{S} 1_{n}) = 0 &
\text{Death Reductions:} \\
\text{X-Reductions:} & 
\ \ \ \bullet \ \ r^{D}_{i, j}(\gamma_{i, n-j+i+1, 1, j}) = 0 \\
\ \ \ \bullet \ \ r^{X}_{2} ({}_{2} x \otimes_{k} 1_{2} ) = {}_{2} 1 \otimes_{k} x_{2} + {}_{2} 1 \otimes_{S} 1_{2} &
\hspace{2cm} \text{ for } 2 \leq i < j \leq n \\ 
\ \ \ \bullet \ \ r^{X}_{2,1}( {}_{2} 1_{1} \otimes_{k} {}_{1} 1_{2} \otimes_{S} {}_{2} x_{1}) = \gamma_{2, 3, 1} &
\text{Mound Reductions:} \\
& \ \ \ \bullet \ \ r^{M}_{j}(\gamma_{1, j, 1}) = 0, \text{ for } 2 \leq j \leq n
\end{array}
$$
where we use the notation $\gamma_{i_{1}, i_{2}, \dots, i_{n}}$ with $i_{1}, \dots, i_{n} \in Q_{0}$ introduced in Remark \ref{rem:notationforpaths}.

Upward Reductions pull paths towards vertex n and realize paths visiting vertex n consecutively to be zero, see Figure \ref{Upward Reductions}. The $X$-reductions remove occurrences of $x$ when possible and moves them to the end of the path otherwise, see Figure \ref{X Reductions}. The Integral Reductions straighten paths after visiting vertex 1, thus removing any instances of the shape $\int$, see Figure \ref{Integral Reductions}. The Death Reductions identify when paths visiting vertex 1 are too long and hence zero and the Mound Reductions identify paths of the shape $\cap$ starting and ending at vertex 1 to be zero, see Figure \ref{Mound and Death Reductions}. 

One reads these figures as reductions according to the following rules:
\begin{itemize} \label{scheme}
\item[--] The paths are depicted in two dimensions with integer vertical component indicating the vertices of the quiver and the horizontal component increases when the path changes directions, for visualization.
\item[--] Unbroken paths are products of the generators ${}_{i} 1_{i+1}, {}_{i+1} 1_{i}$ for $i \in \{1, \dots,  n-1 \}$. Paths  involving generators ${}_{i} x_{i}$  for $i \in \{2, 3, \dots, n \}$ are broken by $x$ at vertex $i$, i.e. $-x-$ indicates an insertion of the length zero path ${}_{i} x_{i}$ at this vertex.
\item[--] Caution: Signs are not depicted. 
\end{itemize}

\begin{rem}
Starting from the relations of the preprojective algebra, viewed as reductions $r^{M}_{2}, r^{X}_{2}, r^{U}_{j}, r^{U}_{n}$ one realizes the need for the remaining reductions in light of the following ambiguities.
\begin{itemize}
\item $R2 \otimes_{S} {}_{2} 1_{1}$ reduces to zero under $r^{X}_{2}$ and to $\gamma_{2,1,2} {}_{2} x_{1} - \gamma_{2,3,1}$ under $r^{M}_{2}$ giving rise to the additional reduction $r^{X}_{2,1}$ sending $\gamma_{2, 1, 2} {}_{2} x_{1}$ to $\gamma_{2,3,1}$.
\item ${}_{1} 1_{2} \otimes_{S} R2 \otimes_{S} \cdots \otimes_{S} R2 \otimes_{S} {}_{2} 1_{1}$ (with $j-2$ copies of $R2$) reduces to zero under $r^{X}_{2}$ and to $\gamma_{1, j, 1}$ under the Upward Reductions $r^{U}_{j-1} \circ r^{U}_{j-2} \circ \cdots \circ r^{U}_{3}$ and hence give rise to the Mound Reduction $r^{M}_{j}$ taking $\gamma_{1, j, 1}$ to zero.
\item 
$(R2 \otimes_{S} {}_{2} 1_{1} \otimes_{k} {}_{1} 1_{2}
-{}_{2}1_{1} \otimes_{k} {}_{1} 1_{2} \otimes_{S} R2) \otimes_{S} \gamma_{2, j}$ reduces to zero under $r^{X}_{2}$ and to $\gamma_{2, 3, 1, j}-\gamma_{2,1,j+1,j}$ under $r^{U}_{j} \circ \cdots \circ r^{U}_{3}$ and so we add the Integral Reduction $r^{I}_{j}$ sending $\gamma_{2,1,j+1,j}$ to $\gamma_{2,3,1,j}$.
\item The path $\gamma_{i, 1, j} \otimes_{S} ( {}_{j} 1_{j+1} \otimes_{S} {}_{j+1} 1_{j})^{n-j+1}$ is sent to zero by $r^{U}_{n} \circ r^{U}_{n-1} \circ \cdots \circ r^{U}_{j+1}$ and iterated applications of the Integral Reduction $r^{I}_{j}$ gives 
$\gamma_{i, 2} \otimes_{S} ({}_{2} 1_{3} \otimes_{S} {}_{3} 1_{2})^{n-j+1} \gamma_{2, 1, j}$ from which applying the Upper Reductions $r^{U}_{n-j+i+1} \circ \cdots \circ r^{U}_{4}\circ r^{U}_{3}$ gives $\gamma_{i, n-j+i+1, 1, j}$. Hence we add the Death Reduction $r^{D}_{i, j}$ taking this term to zero to resolve the ambiguity. 
\end{itemize} 
\end{rem}

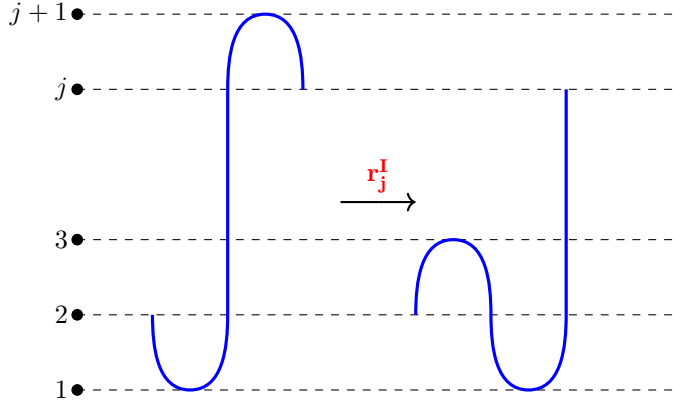
\begin{figure}
\centering

\begin{tikzpicture}

\draw[dashed] (0, 1) - - (8, 1);
\draw[dashed] (0, 2) - - (8, 2);
\draw[dashed] (0, 3) - - (8, 3);
\draw[dashed] (0, 5) - - (8, 5) ;
\draw[dashed] (0, 6) - - (8, 6) ;
\filldraw[black] (0, 6) circle (2pt) node[anchor=east] {$j+1$} ;
\filldraw[black] (0, 5) circle (2pt) node[anchor=east] {$j$} ;
\filldraw[black] (0, 3) circle (2pt) node[anchor=east] {$3$} ;
\filldraw[black] (0, 2) circle (2pt) node[anchor=east] {$2$} ;
\filldraw[black] (0, 1) circle (2pt) node[anchor=east] {$1$} ;

\draw[very thick, blue] (1,2) to [out=270,in=180] (1.5,1) to [out=0,in=-90] (2,2) to [out=90,in=-90] (2,5) to [out=90, in=180] (2.5,6) to [out=0, in=90] (3, 5);

\draw[very thick, blue] (4.5, 2) to [out=90,in=180] (5,3) to [out=0,in=90] (5.5, 2) to [out=-90, in=180] (6, 1) to [out=0, in=-90] (6.5, 2) to [out=90, in =-90] (6.5, 5);

\draw[thick] (3.5, 3.5) -- (4, 3.5) node[anchor=south] {\color{red} $\mathbf{r^{I}_{j}}$};
\draw [thick, ->] (4,3.5) -- (4.5, 3.5);

\end{tikzpicture}

\caption{Integral Reduction $r^{I}_{j}$ for $2 \leq j \leq n-1$}
\label{Integral Reductions}
\end{figure}

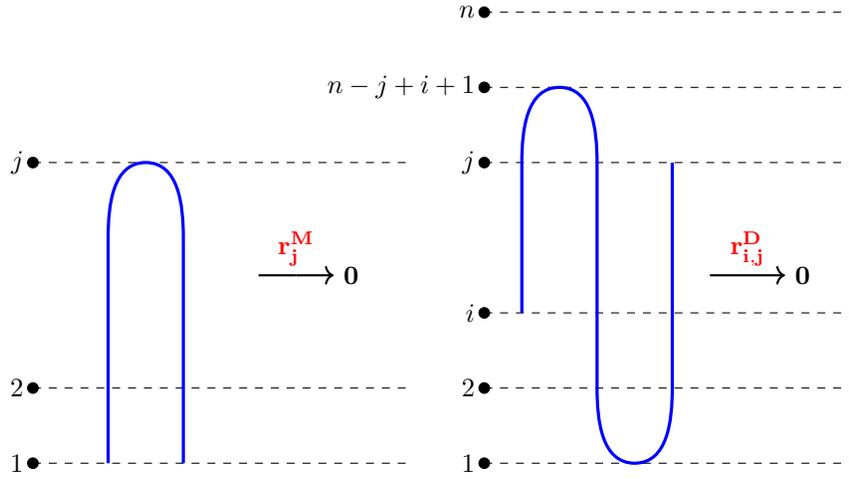
\begin{figure}
\centering

\begin{tikzpicture}

\draw[dashed] (0, 1) - - (5, 1);
\draw[dashed] (0, 2) - - (5, 2);
\draw[dashed] (0, 5) - - (5, 5) ;
\filldraw[black] (0, 5) circle (2pt) node[anchor=east] {$j$} ;
\filldraw[black] (0, 2) circle (2pt) node[anchor=east] {$2$} ;
\filldraw[black] (0, 1) circle (2pt) node[anchor=east] {$1$} ;

\draw[very thick, blue] (1,1) to [out=90,in=-90] (1,4) to [out=90,in=180] (1.5,5) to [out=0,in=90] (2,4) to [out=270, in=90] (2,1);

\draw[thick] (3, 3.5) -- (3.5, 3.5) node[anchor=south] {\color{red} $\mathbf{r^{M}_{j}}$};
\draw [thick, ->] (3.5,3.5) -- (4, 3.5) node[anchor=west] {\textbf{0}};

\draw[dashed] (6, 1) - - (11, 1);
\draw[dashed] (6, 2) - - (11, 2);
\draw[dashed] (6, 3) - - (11, 3);
\draw[dashed] (6, 5) - - (11, 5) ;
\draw[dashed] (6, 6) - - (11, 6);
\draw[dashed] (6, 7) - - (11, 7);
\filldraw[black] (6, 7) circle (2pt) node[anchor=east] {$n$} ;
\filldraw[black] (6, 6) circle (2pt) node[anchor=east] {$n-j+i+1$} ;
\filldraw[black] (6, 5) circle (2pt) node[anchor=east] {$j$} ;
\filldraw[black] (6, 3) circle (2pt) node[anchor=east] {$i$} ;
\filldraw[black] (6, 2) circle (2pt) node[anchor=east] {$2$} ;
\filldraw[black] (6, 1) circle (2pt) node[anchor=east] {$1$} ;

\draw[very thick, blue] (6.5, 3) to [out=90, in=-90] (6.5, 5) to [out=90,in=180] (7,6) to [out=0,in=90] (7.5, 5) to [out=-90, in=90] (7.5, 2) to [out=-90, in=180] (8, 1) to [out=0, in=-90] (8.5, 2) to [out=90, in =-90] (8.5, 5);

\draw[thick] (9, 3.5) -- (9.5, 3.5) node[anchor=south] {\color{red} $\mathbf{r^{D}_{i,j}}$};
\draw [thick, ->] (9.5,3.5) -- (10, 3.5) node[anchor=west] {\textbf{0}};

\end{tikzpicture}

\caption{Mound Reduction $r^{M}_{j}$ for $1 \leq j \leq n$ on the left and Death Reduction $r^{D}_{i,j}$ on the right}
\label{Mound and Death Reductions}
\end{figure}

\begin{figure}
\centering

\begin{tikzpicture}

\draw[dashed] (0, 1) - - (8, 1);
\draw[dashed] (0, -3) - - (8, -3);
\draw[dashed] (0, 2) - - (2, 2); 
\draw[dashed] (2.5, 2) - - (8, 2);
\draw[dashed] (0, -2) - - (.5, -2);
\draw[dashed] (1, -2) - - (5, -2);
\draw[dashed] (5.5, -2) - - (5.75, -2);
\draw[dashed] (6.25, -2) - - (8, -2);
\draw[dashed] (0, 3) - - (8, 3);
\draw[dashed] (0, -1) - - (8, -1);
\filldraw[black] (0, 3) circle (2pt) node[anchor=east] {$3$} ;
\filldraw[black] (0, -1) circle (2pt) node[anchor=east] {$3$} ;
\filldraw[black] (0, 2) circle (2pt) node[anchor=east] {$2$} ;
\filldraw[black] (0, -2) circle (2pt) node[anchor=east] {$2$} ;
\filldraw[black] (0, 1) circle (2pt) node[anchor=east] {$1$} ;
\filldraw[black] (0, -3) circle (2pt) node[anchor=east] {$1$} ;

\draw (2.25,2) node {\large{\textbf{x}}};
\draw[very thick, blue] (1,2) to [out=270,in=180] (1.5,1) to [out=0,in=180] (2,2) 
(2.5, 2) to [out=0, in=180] (3,1);

\draw[very thick, blue] (4.5, 2) to [out=90,in=180] (5,3) to [out=0,in=90] (5.5, 2) to [out=-90, in=180] (6, 1);

\draw[thick] (3.5, 1.3) -- (4, 1.3) node[anchor=south] {\color{red} $\mathbf{r^{X}_{2,1}}$};
\draw [thick, ->] (4,1.3) -- (4.5, 1.3);

\draw (7.75-7,2-4) node {\large{\textbf{x}}};
\draw[very thick, blue] (8-7, 2-4) to [out=0, in=180] (8.5-7,1-4) to [out=0, in=-90] (9-7, 2-4);

\draw[thick] (9.5-7, 1.3-4) -- (10-7, 1.3-4) node[anchor=south] {\color{red} $\mathbf{r^{X}_{2}}$};
\draw [thick, ->] (10-7,1.3-4) -- (10.5-7, 1.3-4);

\draw (12.25-7,2-4) node {\large{\textbf{x}}};
\draw[very thick, blue] (11-7,2-4) to [out=270,in=180] (11.5-7,1-4) to [out=0,in=180] (12-7,2-4) ;

\draw (13-7,2-4) node {\large{\textbf{+}}};

\draw[very thick, blue] (13.5-7, 2-4) to [out=90,in=180] (14-7,3-4) to [out=0,in=90] (14.5-7, 2-4);

\end{tikzpicture}

\caption{X-reductions $r^{X}_{2,1}$ and $r^{X}_{2}$}
\label{X Reductions}
\end{figure}

\begin{figure}
\centering

\begin{tikzpicture}

\draw[dashed] (0, 1) - - (6, 1);
\draw[dashed] (0, 2) - - (6, 2);
\draw[dashed] (0, 3) - - (6, 3) ;
\filldraw[black] (0, 3) circle (2pt) node[anchor=east] {$j+1$} ;
\filldraw[black] (0, 2) circle (2pt) node[anchor=east] {$j$} ;
\filldraw[black] (0, 1) circle (2pt) node[anchor=east] {$j-1$} ;

\draw[very thick, blue] (1, 2) to [out=-90, in=180] (1.5, 1) to [out=0, in=-90] (2, 2);

\draw[thick] (2.5, 1.2) -- (3, 1.2) node[anchor=south] {\color{red} $\mathbf{r^{U}_{j}}$};
\draw [thick, ->] (3,1.2) -- (3.5, 1.2) ;

\draw[very thick, blue]  (4, 2) to [out=90,in=180] (4.5,3) to [out=0,in=90] (5, 2); 

\draw[dashed] (0, 5) - - (6, 5);
\draw[dashed] (0, 6) - - (6, 6);
\filldraw[black] (0, 6) circle (2pt) node[anchor=east] {$n$} ;
\filldraw[black] (0, 5) circle (2pt) node[anchor=east] {$n-1$} ;

\draw[very thick, blue] (1, 6) to [out=-90, in=180] (1.5, 5) to [out=0, in=-90] (2, 6);

\draw[thick] (2.5, 5.2) -- (3, 5.2) node[anchor=south] {\color{red} $\mathbf{r^{U}_{n}}$};
\draw [thick, ->] (3,5.2) -- (3.5, 5.2) node[anchor=west] {\textbf{0}};

\end{tikzpicture}

\caption{Upward Reduction $r^{U}_{j}$ for $3 \leq j \leq n-1$ and $r^{U}_{n}$}
\label{Upward Reductions}
\end{figure}

\begin{prop}
The following is a basis for $e_{i} \Pi(B_{n}) e_{j}$:
$$
\left\lbrace
\begin{array}{l l l}
\gamma_{i, k, j}; \ \ \gamma_{i, k, j} x;  
& \max \{ i, j \} \leq k \leq n
& \\
\gamma_{i, m, 1, j}; \ \ \gamma_{i, m, 1, j} x 
&  i \leq m \leq n-j+i
& \text{if } i, j \geq 2 \\
\gamma_{1, k, j}; \ \ \gamma_{1, k, j}  x 
& j \leq k \leq n 
& \text{if } i=1, j \geq 2 \\
\gamma_{i, k, 1}; \ \ x \gamma_{i, k, 1} 
& i \leq k \leq n
& \text{if } j=1, i \geq 2 \\
\gamma_{1, k} x \gamma_{k, 1}; \ {}_{1}1_{1}
& 1 \leq k \leq n
& \text{if } i=j=1. 
\end{array}
\right .
$$
\end{prop}

\begin{cor}
The bigraded Hilbert--Poincar\'{e} series are: 
\begin{align*}
\begin{split}
h_{e_{i} \Pi(B_{n}) e_{j}}(s, t) {}&=
 t^{|i-j|}s^{|i-j|/2}(1+s)\\
 &\hspace{1cm} \cdot (1+st^2 +s^2t^4 + \cdots + s^{n-\max \{ i, j \}} t^{2n-2 \max \{ i, j \}}) \ \text{ if } i, j \geq 2
\end{split}  \\
\begin{split}
 {}&+
 t^{i+j-2}s^{(i+j-4)/2}(1+s)\\
 &\hspace{1cm} \cdot (1+st^2 +s^2t^4 + \cdots + s^{(n-j-1)/2} t^{n-j-1}) 
\end{split}  \\
\begin{split}
h_{e_{1} \Pi(B_{n}) e_{j}}(s, t) {}&= 
h_{e_{j} \Pi(B_{n}) e_{1}}(s, t) = t^{j-1} s^{(j-1)/2}(1+s)\\
 &\hspace{1cm} \cdot (1+st^2+s^2t^4+ \cdots + s^{n-j} t^{2n-2j})  \text{ if } j \geq 2
\end{split} \\
h_{e_{1} \Pi(B_{n}) e_{1}}(s, t) {}&=1 + st^2+ s^2t^4 + \cdots + s^{n-1}t^{2n-2} 
\end{align*}
and the $x$-degree of ${}_{i} 1_{i+1}$ and ${}_{i+1} 1_{i}$ are $1/2$ to ensure the relations $R2, R3, \dots R(n-1)$ are homogeneous in the $x$-degree. \\
\end{cor}

\begin{proof}
We first show that any path can be reduced to a linear combination of paths in the proposed basis. Let $\alpha$ be a path from $i$ to $j$. \\
\\
\underline{Case: $\alpha$ doesn't visit vertex 1}\\
\indent If $\alpha$ does not contain any occurrences of $x$ then applying the Upward Reductions $r^{U}_{n} \circ r^{U}_{n-1} \circ \cdots \circ r^{U}_{3}$ is either zero or $\gamma_{i, k, j}$ for $k \geq \max \{ i, j \}$. Any occurrence of $x$ can be moved to the end as all tensor products in the path are over $S$ and hence two such occurrences gives zero. \\
\\
\underline{Case: $\alpha$ visits vertex 1 once} 
\\
\indent If $\alpha$ doesn't contain an occurrence of $x$ then by the above the subpaths before and after vertex 1 can be reduced to the form $\gamma_{i, k, 2}$ and $\gamma_{2, l, j}$ and hence $\alpha$ can be reduced to $\gamma_{i, k, 1, l, j}$. \\
\indent If $i=1$ then such a path has the form $\gamma_{1, l, j}$ for $l \geq j$. Similarly, if $j =1$ reading the path backward must have this form. \\
\indent If $i, j >1$ then applying the Integral Reductions $r^{I}_{j} \circ r^{I}_{j+1} \circ \cdots \circ r^{I}_{l-1}$ gives $\gamma_{i, m, 1, j}$, where the path is now direct after visiting vertex 1. Finally, applying the Death Reductions $r_{i, j}$ gives zero unless $m \leq n-j+i$. \\
\indent If $x$ occurs in $\alpha$ before visiting vertex 1, then one can arrange for the path to contain ${}_{2} x_{1}$, since all tensor products are over $S$ before 1. The X-Reduction $r^{X}_{2}$ gives a linear combination of paths where $x$ occurs after vertex 1. So the $x$ can be taken to occur at the end, giving the desired form. \\
\\
\underline{Case: $\alpha$ visits vertex 1 at least twice} \\
\indent If $j >1$ then by applying the X-reductions $r^{X}_{2,1} \circ r^{X}_{2}$ successively one can ensure the $x$'s appear at the end
and so the Mound Reductions $r^{M}_{n} \circ \cdots  \circ r^{M}_{2}$ send the subpath starting and ending at vertex 1 to zero. So the path must end at vertex 1 and cannot visit vertex 1 more than twice. Consequently, the path is concatenation of a path in the previous case ending at vertex 2 (i.e. $\gamma_{i, m, 1, 2}$, $\gamma_{i, m, 1, 2} x$, $\gamma_{1, k, 2}$, $\gamma_{1, k, 2} x)$ with ${}_{2} 1_{1}$ and therefore, after applying the Mound Reductions $r^{M}_{n} \circ \cdots \circ r^{M}_{3} \circ r^{M}_{2}$, is zero or of the form $\gamma_{i, m, 1, 2} {}_{2} x_{1}$ or $\gamma_{1, k} x \gamma_{k,1}$. The $X$-Reduction $r^{X}_{2,1}$ sends $\gamma_{i, m, 1, 2} {}_{2} x_{1}$ to a path visiting vertex 1 a single time and hence any path visiting vertex 1 twice is of the form $\gamma_{1, k} x \gamma_{k,1}$, for $k >1$.

This shows that the above set is a spanning set. It follows from the Diamond Lemma that for linear independence one only needs to show that all ambiguities are resolvable.

The family of Integral Reductions, Mound Reductions, and Death Reductions all arise as ambiguities with Upward Reductions. Moreover, Upward Reductions cannot overlap with X-Reductions nor with other Upward Reductions and hence there are no overlap ambiguities involving them. Clearly their are no ambiguities among two reductions sending the path to zero and so it remains to check Integral Reductions and X-reductions. 

Moreover, Integral Reductions overlap with Mound and Death Reductions in the same way that Upward Reductions do, with the resolution given by other Mound and Death Reductions. Integral Reductions cannot self overlap and the only ambiguity with X-reductions is $r^{X}_{2}$ and $r^{I}_{2}$ applied to $x \gamma_{2,1, 3, 2}$ which both further reduce to $\gamma_{2, 3, 1, 2} x + \gamma_{2, 4, 2}$ by applying $r^{U}_{3} \circ r^{X}_{2} \circ r^{I}_{2}$. The final ambiguity to resolve is $r^{X}_{2}$ and $r^{X}_{2,1}$ applied to $\gamma_{2, 1, 2} x \gamma_{2,1,2}$, with each term further resolving to zero by $r^{M}_{2} \circ r^{M}_{3}$.  
\end{proof}

\begin{rem} 
Carefully reading the above argument, and defining:
$$
r^{U} := r^{U}_{n} \circ \cdots \circ r^{U}_{3}, \hspace{1cm}
r^{M} := r^{M}_{n} \circ \cdots \circ r^{M}_{2},
\hspace{1cm}
r^{I} := r^{I}_{n-1} \circ \cdots \circ r^{I}_{2},
$$
$$
r^{D} := r^{D}_{n-1} \circ \cdots \circ r^{D}_{2}
\hspace{1cm} \text{where}  \hspace{1cm}
r^{D}_{i} := r^{D}_{i, i+1} \circ \cdots \circ r^{D}_{i, n}  
$$
one can define $\opr: P(Q, D_{Q}) \rightarrow P(Q, D_{Q})_{\text{irr}}$ to be the composition
$$
\opr:= r^{M} \circ r^{D} \circ r^{U} \circ r^{I} \circ r^{U} \circ ( r^{X}_{2} \circ r^{X}_{2,1})^{n}. 
$$
That is, one can first reduce a path by moving the $x$'s to the end then pushing upward until Integral Reductions arise and then further pushing upward until Mound or Death Reductions arise. 
\end{rem}

As before, the functional giving the sum of the coefficients of all the maximum length cycles containing an $x$ gives a Frobenius form. 
 
\begin{cor}
$\Pi(B_{n})$ is a flat Frobenius degeneration of $\Pi(A_{2n-1})$.
\end{cor}

\subsection{Case: $\Pi(C_{n})$} \label{subsection:C_n}

Consider the decorated quiver
$$
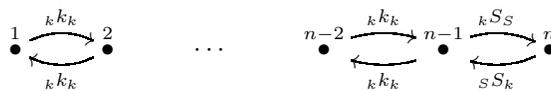
\xymatrix{
\overset{1}{\bullet} \ar@/^/[r]^{{}_{k} k_{k}} & \overset{2}{\bullet} \ar@/^/[l]^{{}_{k} k_{k}} & \cdots & 
\overset{n-2}{\bullet} \ar@/^/[r]^{{}_{k} k_{k}} & \ar@/^/[l]^{{}_{k} k_{k}}  \overset{n-1}{\bullet} \ar@/^/[r]^{{}_{k} S_{S}} & \overset{n}{\bullet} \ar@/^/[l]^{{}_{S} S_{k}}
}
$$
The goal of this section is to compute a basis for the decorated preprojective algebra using the Diamond Lemma for rings, developed in \cite{Bergman} and explained in the previous subsection. 

\begin{prop} \label{Basis Cn 2}
A basis for $e_{i} \Pi(C_{n}) e_{j}$ is given by
$$
\{ \gamma_{i, m, j}; \ \ \gamma_{i, l, n} {}_{n} x_{n} \gamma_{n, j} \ \ \mid \ \ m \leq \text{min} \{i, j \}, i \geq l > \text{max} \{i-j, 0 \} \} 
$$ 
and hence $\dim( e_{i} \Pi(C_{n}) e_{j}) =2 \text{min} \{i, j \}$. In particular, $\Pi(C_{n})$ is a flat degeneration of $\Pi(D_{n+1})$.
\end{prop} 

The relations in the preprojective algebra are
$$
\begin{cases}
{}_{1} 1_{2} \otimes_{k_{2}} {}_{2}1_{1} & \text{ at vertex } 1 \\

{}_{j} 1_{j+1} \otimes_{k} {}_{j+1} 1_{j} + {}_{j} 1_{j-1} \otimes_{k} {}_{j-1} 1_{j} & \text{ at vertex } j \in \{2, 3, 4, \dots, n-1 \} \\
{}_{n-1} 1_{n} \otimes_{S} {}_{n} 1_{n-1} + {}_{n-1} 1_{n-2} \otimes_{k} {}_{n-2} 1_{n-1} & \text{ at vertex } n-1  \\
{}_{n} 1_{n-1} \otimes_{k} {}_{n-1} x_{n} + {}_{n} x_{n-1} \otimes_{k} {}_{n-1} 1_{n} & \text{ at vertex } n. \\
\end{cases}
$$
One views each relation as a reduction reducing the length or the number of times a path changes direction by pulling all paths towards vertex 1. More formally, one defines a partial order, determined on length one paths by 
$$
{}_{i} x_{i-1} > {}_{i} x_{i+1} > {}_{i} 1_{i-1} > {}_{i} 1_{i+1} 
$$
and extended to all paths lexicographically, as described in the previous subsection. Hence the reductions are determined by:
\begin{align*}
&r_{1} ({}_{1} 1_{2} \otimes_{k_{2}} {}_{2}1_{1}) = 0 \\
&r_{j}  ({}_{j} 1_{j+1} \otimes_{k} {}_{j+1} 1_{j}) = -{}_{j} 1_{j-1} \otimes_{k} {}_{j-1} 1_{j}  \ \text{ for } j \in \{ 2, \dots, n-2 \}\\
&r_{n-1}  ({}_{n-1} 1_{n} \otimes_{S} {}_{n} 1_{n-1}) =  -{}_{n-1} 1_{n-2} \otimes_{k} {}_{n-2} 1_{n-1} \\
&r_{n}  ( {}_{n} x_{n-1} \otimes_{k} {}_{n-1} 1_{n} ) = -{}_{n} 1_{n-1} \otimes_{k} {}_{n-1} x_{n}.
\end{align*}
With this set of reductions there are unresolvable ambiguities and hence we introduce additional reductions: \\

\noindent Downward Reductions
\begin{itemize}
\item $r_{1}( {}_{1} 1_{2} \otimes_{k} {}_{2} 1_{1} ) = 0 $
\item $r_{j}({}_{j} 1_{j+1} \otimes_{k} {}_{j+1} 1_{j} ) = -{}_{j} 1_{j-1} \otimes_{k} {}_{j-1} 1_{j} \text{ for } l < j < n-1$
\item $r_{n-1}({}_{n-1} 1_{n} \otimes_{S} {}_{n} 1_{n-1} ) = -{}_{n-1} 1_{n-2} \otimes_{k} {}_{n-2} 1_{n-1}$
\end{itemize}
X-Reductions
\begin{itemize}
\item $r^{X}_{k} ({}_{n} x_{n} \gamma_{n, k-1, k})
= \gamma_{n, n-1, n} {}_{n} x_{n} \gamma_{n, k}  \text{ for } 1 < k \leq n$
\end{itemize}
Death Reductions
\begin{itemize}
\item $r^{D}_{i, j}(\gamma_{i,i-j, n} {}_{n} x_{n} \gamma_{n, j}) = 0 \text{ for } i>j$.\\
\end{itemize}

Visually, Figures \ref{Downward Reductions}, \ref{X-Reductions}, and \ref{Death Reductions} provide schematics for the reductions, where paths are depicted with vertical position giving the vertex of the quiver and the horizontal component is strictly for visualization, as explained in the previous subsection.  

\begin{rem}
The X-reductions are introduced to resolve the ambiguity created by reducing $_{n} x_{n} \gamma_{n, n-1, n, k}$ using $r_{k} \circ \cdots \circ r_{n-2} \circ r_{n-1}$ as opposed to $r_{n}$. The Death Reductions are then introduced to resolve the ambiguity created by reducing $\gamma_{i, n} x_{n} (\gamma_{n, n-1, n})^{j} \gamma_{n, j}$ using $r_{i-j+1} \circ r_{i-j+2} \circ \cdots \circ r_{n} \circ (r_{n}^{X})^{j}$ as opposed to $r_{1} \circ r_{2} \circ \cdots \circ r_{n}$.
\end{rem}
  
\begin{figure}

\centering

\begin{turn}{270}
\begin{tikzpicture}[xscale=-1]

\draw[dashed] (2, 0) - - (2, 5) ;
\draw[dashed] (3, 0) - - (3, 5);
\draw[dashed] (4, 0) - - (4, 5) ;
\draw[dashed] (6, 0) - - (6, 5) ;

\filldraw[black] (2, 0) circle (2pt) node[anchor=north]{\rotatebox{90}{$j-1$}} ;
\filldraw[black] (3, 0) circle (2pt) node[anchor=north]
{\rotatebox{90}{$j$}} ;
\filldraw[black] (4, 0) circle (2pt) node[anchor=north]
{\rotatebox{90}{$j+1$}} ;
\filldraw[black] (6, 0) circle (2pt) node[anchor=north]
{\rotatebox{90}{$i$}} ;

\draw[very thick] (6,1) to [out=180,in=0] (4,1) to [out=180,in=270] (3,1.5) ;
\draw[blue, very thick] (3, 1.5) to [out=90,in=270] (4,2) to [out=90, in=270] (3,2.5) ;

\draw[very thick] (6,4) to [out=180,in=0] (3,4) ;
\draw[blue, very thick] (3, 4) to [out=180,in=270] (2,4.5) to [out=90,in=270] (3,5) ;

\draw[thick] (3.5, 2.8) -- (3.5, 3.1) node[anchor=west] {\rotatebox{90}{\color{red} $\mathbf{r_{j}}$}};
\draw [thick, ->] (3.5,3.1) -- (3.5, 3.5);

\draw[dashed] (8, 0) - - (8, 5) ;
\draw[dashed] (9, 0) - - (9, 5);
\filldraw[black] (8, 0) circle (2pt) node[anchor=north]
{\rotatebox{90}{$1$}} ;
\filldraw[black] (9, 0) circle (2pt) node[anchor=north] 
{\rotatebox{90}{$2$}};

\draw[blue, very thick] (8,1.5) to [out=90,in=270] (9,2) to [out=90, in=270] (8,2.5);

\draw[thick] (8.5, 3.2) -- (8.5, 3.6) node[anchor=west] {\rotatebox{90}{\color{red} $\mathbf{r_{1}}$}};
\draw [thick, ->] (8.5,3.6) -- (8.5, 3.9) node[anchor=south] {\rotatebox{90}{0}} ;

\end{tikzpicture}
\end{turn}
\caption{Downward Reductions}
\label{Downward Reductions}

\end{figure}

\begin{figure}
\centering

\begin{turn}{270}
\begin{tikzpicture}[xscale=-1]

\draw[dashed] (8, 0) - - (8, .5) ;
\draw[dashed] (8, 1) - - (8, 4);
\draw[dashed] (8, 4.5) - - (8, 5);
\draw[dashed] (7, 0) - - (7, 5);
\draw[dashed] (5, 0) - - (5, 5) ;
\draw[dashed] (4, 0) - - (4, 5) ;
\filldraw[black] (8, 0) circle (2pt) node[anchor=north] 
{\rotatebox{90}{$n$}}  ;
\filldraw[black] (7, 0) circle (2pt) node[anchor=north] 
{\rotatebox{90}{$n-1$}}  ;
\filldraw[black] (5, 0) circle (2pt) node[anchor=north]
{\rotatebox{90}{$j$}}  ;
\filldraw[black] (4, 0) circle (2pt) node[anchor=north] 
{\rotatebox{90}{$j-1$}}  ;

\draw (8,0.75) node {\rotatebox{90}{\large{x}}};
\draw[very thick, blue] (8,1) to [out=90,in=180] (7,1.5) to [out=180,in=180] (5,1.5) to [out=180,in=-90] (4,2) to [out=90, in=-90] (5,2.5) ;

\draw[very thick, blue] (8,3) to [out=180,in=-90] (7,3.5) to [out=90,in=-90] (8,4) ;
\draw (8, 4.25) node {\rotatebox{90}{\large{x}}};
\draw[very thick, blue] (8, 4.5) to [out=90, in=180] (7, 5)
to [out=180, in=0] (5, 5) ;

\draw[thick] (6, 2.5) -- (6, 3) node[anchor=west] {\rotatebox{90}{\color{red} $\mathbf{r_{n, j}}$}};
\draw [thick, ->] (6,3) -- (6, 3.5);

\end{tikzpicture}
\end{turn}
\caption{X-Reductions}
\label{X-Reductions}
\end{figure}
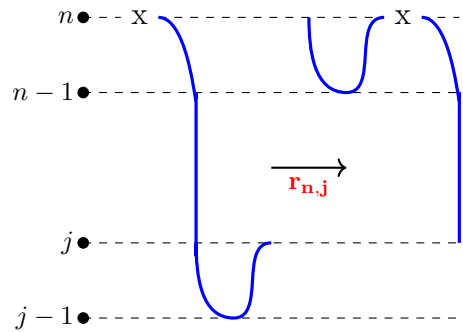

\begin{figure}
\centering

\begin{turn}{270}
\begin{tikzpicture}[xscale=-1]

\draw[dashed] (8, 0) - - (8, 2.3);
\draw[dashed] (8, 2.7) - - (8, 5);
\draw[dashed] (7, 0) - - (7, 5);
\draw[dashed] (5, 0) - - (5, 5) ;
\draw[dashed] (4, 0) - - (4, 5) ;
\draw[dashed] (2, 0) - - (2, 5) ;
\filldraw[black] (8, 0) circle (2pt) node[anchor=north] {\rotatebox{90}{$n$}} ;
\filldraw[black] (5, 0) circle (2pt) node[anchor=north] 
{\rotatebox{90}{$i$}}  ;
\filldraw[black] (4, 0) circle (2pt) node[anchor=north] 
{\rotatebox{90}{$i-j$}}  ;
\filldraw[black] (2, 0) circle (2pt) node[anchor=north] 
{\rotatebox{90}{$j$}}  ;

\draw[very thick, blue] (5,1) to [out=180,in=-90] (4,1.5) to [out=90,in=180] (5,2) to [out=0,in=180] (7,2) to [out=0, in=-90] (8,2.3) ;
\draw (8,2.5) node {\rotatebox{90}{\large{x}}};
\draw[very thick, blue] (8, 2.7) to [out=90, in=180] (7, 3) to [out=180, in=0] (2,3);

\draw[thick] (6, 3.5) -- (6, 3.8) node[anchor=west] 
{\rotatebox{90}{\color{red} $\mathbf{r'_{i, j}}$}};
\draw [thick, ->] (6,3.8) -- (6, 4.2) node[anchor=south] {\rotatebox{90}{0}} ;

\end{tikzpicture}
\end{turn}

\caption{Death Reductions}
\label{Death Reductions}
\end{figure}
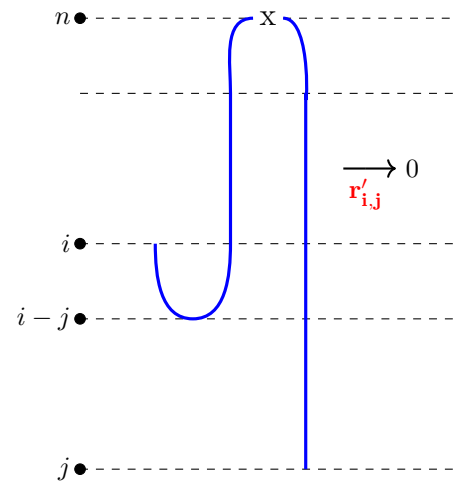

All ambiguities are resolvable:
\begin{itemize}
\item[(I)] The ambiguity giving rise to additional $X$-reductions: \\ 
$r^{X}_{j} - r^{X}_{j-1} \circ r_{j-1} ({}_{n} x_{n} \ \gamma_{n, j-1, j, j-1}) = 0 $
\item[(II)] The ambiguity giving rise to the Death reductions: \\
$ r_{1} - r^{X}_{2} \circ r^{X}_{1} ({}_{n} x_{n} \gamma_{n, 1, 2, 1}) = 0  $
\item[(III)] The Downward-Death ambiguity: \\
$r^{D}_{i,j} -  r^{D}_{i-1, j} \circ r_{i-j} \circ \cdots \circ r_{i-2} \circ r_{i-1} (\gamma_{i-1, i, i-j, n} {}_{n} x_{n} \gamma_{n, i-j}) = 0$ 
\item[(IV)]  The X-Death ambiguity: \\
$r^{D}_{i,j}-  r^{D}_{i, j+1} \circ  r_{i-j} \circ \cdots \circ r_{n-2} \circ r_{n-1} \circ r^{X}_{j+1} ( \gamma_{i, i-j, n} {}_{n} x_{n} \gamma_{n, j, j+1} )=0 $
\item[(V)] The Double X-Death ambiguity: \\
$r^{D}_{i, j} - r^{X}_{n-1} \circ r^{X}_{n-2} \circ \cdots \circ r^{X}_{j} ( {}_{n} x_{n} \gamma_{n, i-j-1, n} {}_{n} x_{n} \ \gamma_{n, j}) = 0$
\end{itemize}

With these reductions one can put any path in normal form. According to the Downward Reductions such a path cannot go from increasing to decreasing, except at vertex $n$. Consequently, paths not containing $x$ must be of the from $\gamma_{i, m, j}$ for $m \leq \text{min} \{ i, j\}$, where equality gives the direct path. Moreover, no path can visit vertex $n$ more than once since between the final two such visits the path can be reduced to $\gamma_{n, m, n}$ at which point iterated application of the X-reductions collides the two $x$'s giving zero.

Applying the X-reductions to a path with one instance of $x$ allows for the path after $x$ to be direct. Therefore, every such path can be reduced to one of the form $\gamma_{i, l, n} {}_{n} x_{n} \gamma_{n, j}$ and the Death Reductions say that these are non-zero if $l >i-j$. This proves Proposition \ref{Basis Cn 2}.


\begin{cor}
The matrix-valued Hilbert--Poincar\'{e} series of $\Pi(C_{n})$ is of the form,
$$
\left( 
\begin{array}{ccccc}
2 & 2 & 2 & \cdots & 2 \\
2 & 4 & 4 & \cdots & 4 \\
2 & 4 & 6 & \cdots & 6 \\
\vdots & \vdots & \vdots & \ddots & \vdots \\
2 & 4 & 6 & \cdots & 2n \\
\end{array}
\right)
$$
The bigraded Hilbert--Poincar\'{e} series 
$$
h_{e_{i} \Pi(C_{n}) e_{j}}(s,t) = 
t^{d(i,j)}(1+t^2+t^4+ \cdots + t^{2 \min \{i,j\}-2})(1+ s t^{2 d(\max \{i, j \}, n)})
$$
where $d(a,b) = | a- b |$ is the distance between the vertices $a$ and $b$. 
\end{cor}

As above, one can define a Frobenius form on the degenerate preprojective algebra using the cycles of maximal length, $\gamma_{i, n} {}_{n}x_{n} \alpha^{i-1} \gamma_{n, i}$ for each $i \in \{1, \dots, n\}$. 

\begin{cor}
$\Pi(C_{n})$ is a flat Frobenius degeneration of $\Pi(D_{n+1}).$
\end{cor} 

This is the final case of maximal degenerations of folding ADE quivers and hence we've completed the proof of Theorem \ref{thm:Dynkin}.

\bibliography{FD-wThesisCorrections}
\bibliographystyle{plain}


\end{document}